\documentclass[11pt]{amsart}
\usepackage{graphicx,rlepsf,latexsym,amssymb}

\usepackage[all]{xy}
\CompileMatrices

\begin{document}
%
%
%
%
%
%

%
%

\theoremstyle{plain}
\newtheorem{theorem}{Theorem}[section]
\newtheorem{lemma}[theorem]{Lemma}
\newtheorem{proposition}[theorem]{Proposition}
\newtheorem{corollary}[theorem]{Corollary}
\newtheorem{conjecture}[theorem]{Conjecture}

\theoremstyle{definition}
\newtheorem{definition}[theorem]{Definition}
\newtheorem{question}[theorem]{Question}
\newtheorem{example}[theorem]{Example}
\newtheorem{remark}[theorem]{Remark}
\newtheorem{summary}[theorem]{Summary}
\newtheorem{notation}[theorem]{Notation}
\newtheorem{problem}[theorem]{Problem}

\theoremstyle{remark}
\newtheorem{claim}[theorem]{Claim}
\newtheorem{sublemma}[theorem]{Sub-lemma}
\newtheorem{innerremark}[theorem]{Remark}

\numberwithin{equation}{section}
\numberwithin{figure}{section}

%
%

\renewcommand{\labelitemi}{$\centerdot$}

%
%

\def \C{\mathbb{C}}
\def \Z{\mathbb{Z}}
\def \N{\mathbb{N}}
\def \Q{\mathbb{Q}}
\def \K{\mathbb{K}}

\def \Cob{\mathcal{C}ob}
\def \Cyl{\mathcal{C}}

\def \F{\mathcal{F}}
\def \I{\mathcal{I}}
\def \M{\mathcal{M}}
\def \P{\mathcal{P}}
\def \T{\mathcal{T}}

\def \Ztilde{\widetilde{Z}}
\def \ZtildeY{\widetilde{Z}^{Y}}

\def \osqcup{\hphantom{}^<_\sqcup}

\def \longrightleftarrows{\begin{array}{c}\longrightarrow \\[-0.2cm] \longleftarrow \end{array}}

\def \Char{\operatorname{Char}}
\def \Coker{\operatorname{Coker}}
\def \deg{\operatorname{deg}}
\def \ideg{\operatorname{i\hbox{-}deg}}
\def \End{\operatorname{End}}
\def \GL{\operatorname{GL}}
\def \Gr{\operatorname{Gr}}
\def \Hom{\operatorname{Hom}}
\def \incl{\operatorname{incl}}
\def \Id{\operatorname{Id}}
\def \Img{\operatorname{Im}}
\def \Ker{\operatorname{Ker}}
\def \Lie{\operatorname{Lie}}
\def \mod{\operatorname{mod}}
\def \ord{\operatorname{ord}}
\def \pr{\operatorname{pr}}
\def \rk{\operatorname{rk}}
\def \sgn{\operatorname{sgn}}
\def \Sp{\operatorname{Sp}}
\def \Tors{\operatorname{Tors}}
\def \ud{\operatorname{d}}
\def \GLike{\operatorname{G}}
\def \Prim{\operatorname{P}}

\def \sp{\mathfrak{sp}}

\newcommand{\set}[1]{\lfloor #1\rceil}
\newcommand{\Star}[3]{\stackrel{{\scriptsize #1},{\scriptsize #2}}{\star}}

\newcommand{\figtotext}[3]{\begin{array}{c}\includegraphics[width=#1pt,height=#2pt]{#3}\end{array}}

\newcommand{\thetagraph}
{\hspace{-0.2cm} \figtotext{12}{12}{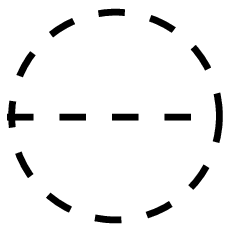} \hspace{-0.2cm}}

\newcommand{\strutgraph}[2]
{\begin{array}{c} \\[-0.2cm] \!\! {\relabelbox \small
\epsfxsize 0.07truein \epsfbox{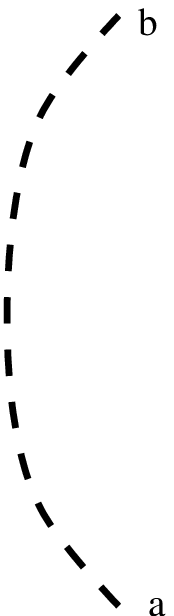}
\adjustrelabel <0.0cm,-0cm> {a}{\scriptsize $#1$}
\adjustrelabel <0.05cm,-0cm> {b}{\scriptsize $#2$}
\endrelabelbox} \end{array}}

\newcommand{\strutgraphbot}[2]
{\begin{array}{c} \\[-0.2cm] \!\! {\relabelbox \small
\epsfxsize 0.3truein \epsfbox{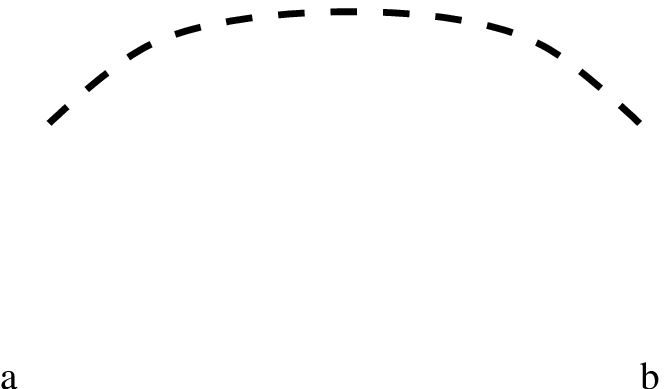}
\adjustrelabel <0cm,0.05cm> {a}{\scriptsize $#1$}
\adjustrelabel <0cm,0.05cm> {b}{\scriptsize $#2$}
\endrelabelbox} \end{array}}

\newcommand{\Hgraph}[4]
{\begin{array}{c} \\[-0.2cm] \!\! {\relabelbox \small
\epsfxsize 0.2truein \epsfbox{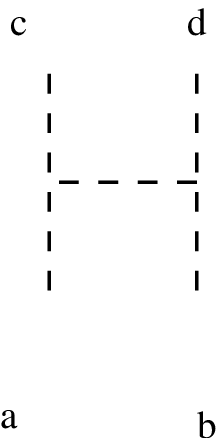}
\adjustrelabel <0cm,0.05cm> {a}{\scriptsize $#1$}
\adjustrelabel <0cm,0.05cm> {b}{\scriptsize $#2$}
\adjustrelabel <0cm,0.0cm> {c}{\scriptsize $#3$}
\adjustrelabel <0cm,0.0cm> {d}{\scriptsize $#4$}
\endrelabelbox} \end{array}}

\newcommand{\phigraphtop}[2]
{\begin{array}{c} \\[-0.2cm] \!\! {\relabelbox \small
\epsfxsize 0.25truein \epsfbox{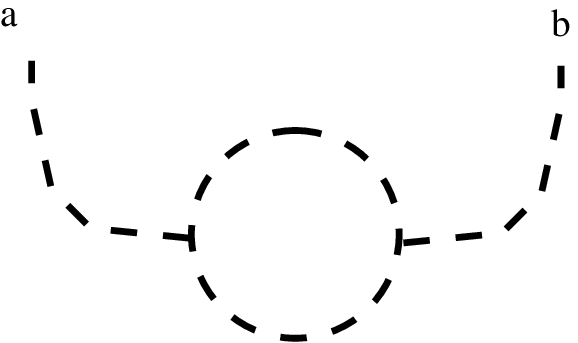}
\adjustrelabel <0cm,0.05cm> {a}{\scriptsize $#1$}
\adjustrelabel <0cm,0.1cm> {b}{\scriptsize $#2$}
\endrelabelbox} \end{array}}

\newcommand{\Ygraphtop}[3]
{\begin{array}{c} \\[-0.2cm] \!\! {\relabelbox \small
\epsfxsize 0.3truein \epsfbox{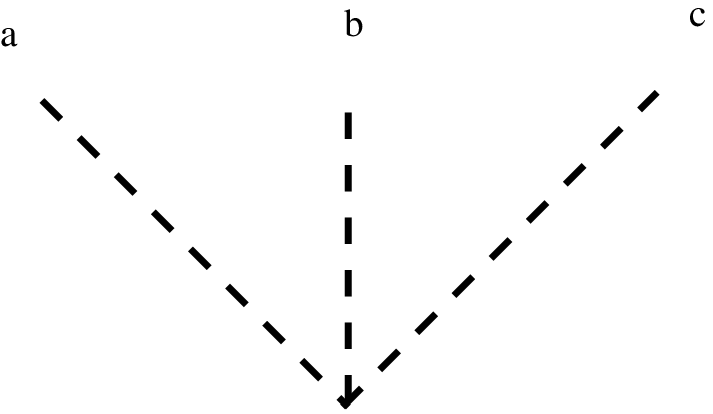}
\adjustrelabel <-0cm,-0cm> {a}{\scriptsize $#1$}
\adjustrelabel <-0cm,-0cm> {b}{\scriptsize $#2$}
\adjustrelabel <-0cm,-0cm> {c}{\scriptsize $#3$}
\endrelabelbox} \end{array}}

\newcommand{\Ygraphbottoptop}[3]
{\begin{array}{c} \\[-0.2cm] \!\! {\relabelbox \small
\epsfxsize 0.2truein \epsfbox{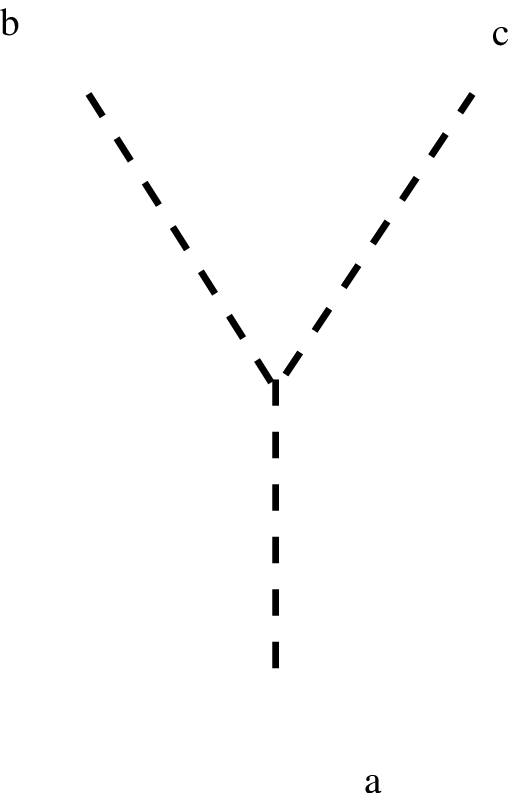}
\adjustrelabel <0cm,-0cm> {a}{\scriptsize $#1$}
\adjustrelabel <0cm,-0cm> {b}{\scriptsize $#2$}
\adjustrelabel <0cm,-0cm> {c}{\scriptsize $#3$}
\endrelabelbox} \end{array}}

\newcommand{\Ygraphbotbottop}[3]
{\begin{array}{c} \\[-0.2cm] \!\! {\relabelbox \small
\epsfxsize 0.3truein \epsfbox{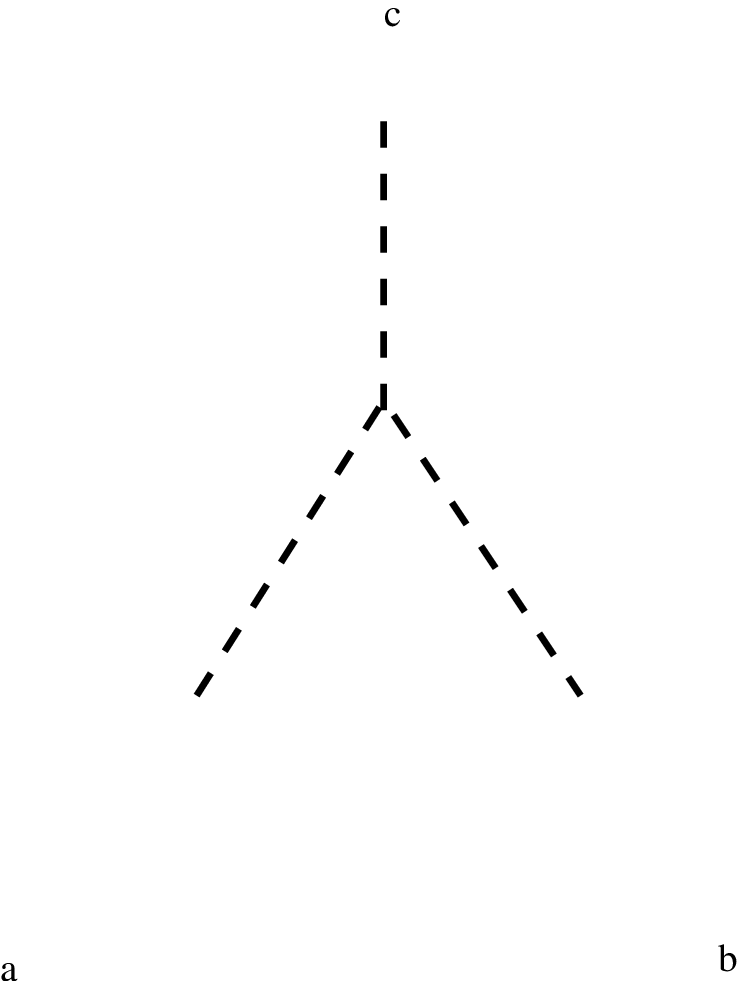}
\adjustrelabel <0cm,-0cm> {a}{\scriptsize $#1$}
\adjustrelabel <0cm,-0cm> {b}{\scriptsize $#2$}
\adjustrelabel <0cm,-0cm> {c}{\scriptsize $#3$}
\endrelabelbox} \end{array}}

\newcommand\nc{\newcommand}

\nc\mapcyl{\mathbf{c}}
\nc\mapcylhat{\widehat{\mapcyl}}

\nc\sgg{\Sigma_{g,1}}
\nc\tsgg{\tilde\Sigma_{g,1}}
\nc\sg{\Sigma_g}
\nc\Igg{\I_{g,1}}
\nc\Ig{\I_g}
\nc\Igghat{\widehat{\I}_{g,1}}
\nc\Mgg{\mathcal{M}_{g,1}}
\nc\Cgg{\Cyl_{g,1}}
\nc\Cg{\Cyl_{g}}
\nc\Cgghat{\widehat{\Cyl}_{g,1}}
\nc\LCob{\mathcal{L}\mathcal{C}ob}

\nc\bfi {\mathbf i}

\nc\Yhat{\widehat{Y}}

\nc\oQ{\otimes\Q}

\nc\LMO{\operatorname{LMO}}

\nc\simeqto{\overset{\simeq}{\longrightarrow}}
\nc\xto[1]{\overset{#1}{\longrightarrow}}

\nc\glim{\varinjlim_g}

\nc\og{$(\omega ,<)$}
\nc\omegaQ{\omega_\Q}
\nc\HQ{H_\Q}
\nc\LHQ{\Lambda^3\HQ}
\nc\SpHQ{\Sp(\HQ)}
\nc\sHQ{\left(\HQ\right)}

\nc\A{\mathcal{A}}
\nc\AY{\mathcal{A}^Y}
\nc\AYc{\mathcal{A}^{Y,c}}
\nc\AgHQ{\A^<}
\nc\AgcHQ{\A^{<,c}}
\nc\AtildeHQ{\widetilde{\A}}
\nc\AHQ{\A}
\nc\AcHQ{\A^c}
\nc\Ac[1]{\A^c_{#1}}
\nc\Ag[1]{\A^<_{#1}}
\nc\Agc[1]{\A^{<,c}_{#1}}
\nc\tsA{{}^{ts}\negthinspace\negthinspace\A}
\nc\cI{\check I}

\nc\Aut{\operatorname{Aut}}
\nc\Grp{\operatorname{Grp}}
\nc\bbI{\mathbb I}
\nc\bbX{\mathbb X}
\nc\calT{\mathcal T}
\nc\bbR{\mathbb R}
\nc\ev{\operatorname{ev}}
\nc\od{\operatorname{od}}

\nc\Span{\operatorname{Span}}

\nc\MalcevGroup[1]{\mathcal{G}(#1)}
\nc\MalcevLie[1]{\mathcal{P}(#1)}

\title[]{Symplectic Jacobi diagrams and\\ the Lie algebra of homology cylinders}

\date{July 6, 2009}

\author[]{Kazuo Habiro}
\address{Research Institute for Mathematical Sciences, Kyoto University, Kyoto 606-8502, Japan}
\email{habiro@kurims.kyoto-u.ac.jp}

\author[]{Gw\'ena\"el Massuyeau}
\address{Institut de Recherche Math\'ematique Avanc\'ee, Universit\'e de Strasbourg \& CNRS,
7 rue Ren\'e Descartes, 67084 Strasbourg, France}
\email{massuyeau@math.u-strasbg.fr}



\begin{abstract}
  Let $S$ be a compact connected oriented surface, whose boundary is connected or empty.
  A homology cylinder over the surface $S$ is  a cobordism between $S$ and itself,
  homologically equivalent to the cylinder over $S$.  
  The $Y$-filtration on the monoid of homology   cylinders over $S$ is defined by clasper surgery.
  Using a functorial extension of the Le--Murakami--Ohtsuki invariant, we show
  that the graded Lie algebra associated to the $Y$-filtration is
  isomorphic to the Lie algebra of ``symplectic Jacobi  diagrams''.  
  This Lie algebra consists of the primitive elements of a certain Hopf algebra 
  whose multiplication is a diagrammatic analogue of the Moyal--Weyl product.

  The mapping cylinder construction embeds the Torelli group into the
  monoid of homology cylinders, sending the lower central series to
  the $Y$-filtration.  We give a combinatorial description of the
  graded Lie algebra map induced by this embedding, by connecting
  Hain's infinitesimal presentation of the Torelli group to the Lie
  algebra of symplectic Jacobi diagrams.  This Lie algebra map is
  shown to be injective in degree two, and the question of the
  injectivity in higher degrees is discussed.
\end{abstract}

\maketitle

\setcounter{tocdepth}{1}
{\small \tableofcontents}

\section{Introduction and statement of the results}

\label{sec:intro}

Let $\sgg$ be a compact connected oriented surface of genus $g$ with one boundary component. 
The first homology group $H_1(\sgg;\Z)$ is denoted by $H$ and is equipped 
with the intersection pairing 
$$
\omega: H \otimes H \longrightarrow \Z.
$$
This is a non-degenerate skew-symmetric form, the group of isometries of which is denoted by $\Sp(H)$.
Similarly, $H_\Q:= H \oQ$ is equipped with the rational extension of $\omega$ 
and $\Sp(H_\Q)$ denotes the group of isometries of the symplectic vector space $H_\Q$. 

\subsection{The Torelli group}

Let $\Igg$ be the {\em Torelli group} of $\sgg$, which is the
subgroup of the mapping class group $\Mgg$ of $\sgg$
consisting of the elements acting trivially on homology.  
A good introduction to the Torelli group is found in Johnson's survey \cite{Johnson_survey}.  

{\em Commutator calculus} is one of the most important tools in the study of the Torelli group. 
The group $\Igg$ is filtered by its lower central series
$$
\Igg = \Gamma_1 \Igg \supset \Gamma_2 \Igg \supset \Gamma_3 \Igg \supset \cdots.
$$
The \emph{pronilpotent completion} of $\Igg$ is
$$
\Igghat := \varprojlim_i \Igg/ \Gamma_i \Igg
$$
and the canonical map $\Igg \to \Igghat$ is injective.
The graded Lie algebra over $\Z$ associated to the lower central series of $\Igg$, namely
$$
\Gr^\Gamma \Igg = \bigoplus_{i\geq 1} \Gr^\Gamma_i\Igg,
\quad \hbox{where } \Gr^\Gamma_i \Igg := {\Gamma_i \Igg}/{\Gamma_{i+1} \Igg},
$$
is called the \emph{Torelli Lie algebra} of $\sgg$.
With rational coefficients, the Torelli Lie algebra
$$
\Gr^\Gamma \Igg \oQ
$$ 
is, by a general fact, canonically isomorphic to the graded Lie algebra 
associated to the complete lower central series of the Malcev Lie algebra of $\Igg$.

The Torelli Lie algebra with rational coefficients 
is generated by its degree $1$ part $\Gr_1^\Gamma\Igg\oQ$.
If $g\geq 3$, this vector space can be identified with $\LHQ$
by extending the first Johnson homomorphism
\begin{gather*}
  \tau_1\colon \Igg \longrightarrow \Lambda^3 H
\end{gather*}
to rational coefficients \cite{Johnson}. Hence a Lie algebra epimorphism 
$$
J: \Lie(\Lambda^3 H_\Q) \longrightarrow \Gr^\Gamma\Igg\oQ,
$$
where $\Lie(\LHQ)$ is the free Lie algebra over $\LHQ$.
The ideal of \emph{relations} of the Torelli Lie algebra is
$$
\hbox{R}\left(\Igg\right)=\bigoplus_{i\ge1}\hbox{R}_i\left(\Igg\right):=\operatorname{ker}(J).
$$
Let $\overline{J}: \Lie(\Lambda^3 H_\Q)/ \hbox{R}\left(\Igg\right) \to \Gr^\Gamma\Igg\oQ$
be the isomorphism induced by $J$.
The following theorem is proved by Hain in \cite{Hain}.

\begin{theorem}[Hain]
\label{th:Hain}
If $g\ge3$, then the Malcev Lie algebra of $\Igg$ is isomorphic 
to the completion of $\Gr^\Gamma\Igg\oQ$.
Moreover, the ideal ${\rm{R}}\left(\Igg\right)$ 
is generated by  ${\rm{R}}_2\left(\Igg\right)$
for $g\geq 6$, and by ${\rm{R}}_2\left(\Igg\right)+
{\rm{R}}_3\left(\Igg\right)$ for $g=3,4,5$.
\end{theorem}

\noindent
The first half of Theorem \ref{th:Hain} implies that the Torelli Lie algebra
has all the information about the Malcev completion of $\Igg$.  
The second half implies that one obtains a
presentation of the Torelli Lie algebra by
computing the quadratic/cubic relations (see \cite{Hain,HS} for $g\geq 6$).

\subsection{The monoid of homology cylinders}

\emph{Homology cylinders} over $\sgg$ are cobordisms from $\sgg$ 
to itself with the same homology type as $\sgg \times [-1,1]$.
The set $\Cgg$ of homeomorphism types (relative to boundary parameterization) 
of homology cylinders is a monoid, with multiplication being the usual pasting operation of cobordisms.  
Homology cylinders are introduced in \cite{Goussarov,Habiro}, 
see also  \cite{Levine,GL_tree-level,Habegger,MM,Sakasai_cylinders,Massuyeau} where 
the terminology ``homology cylinder''  sometimes refers to a wider class of cobordisms.

{\em Calculus of claspers} \cite{Goussarov_clovers,Habiro} works as
``topological commutator calculus'' on the monoid $\Cgg$, where the
usual algebraic commutator calculus does not work.  The role of the
lower central series is played by the family of the $Y_i$-equivalence
relations.  For each $i\ge1$, the {\em $Y_i$-equivalence} on homology
cylinders is generated by surgeries along $Y_i$-claspers, which are
connected graph claspers with $i$ nodes.  
The $Y_i$-equivalence is also generated by {\em Torelli surgeries of class
$i$}, i.e$.$ surgeries along an embedding of the surface $\Sigma_{h,1}$
with any $h\ge0$ using any element of $\Gamma_i\mathcal{I}_{h,1}$ (see \cite{Habiro,Massuyeau}).  
The $Y_i$-equivalence becomes finer as $i$ increases.

For each $i\ge1$, the quotient monoid $\Cgg/Y_i$ is a finitely generated, 
nilpotent group, and there is a sequence of surjective
homomorphisms
\begin{gather*}
  \cdots \longrightarrow
  \Cgg/Y_3 \longrightarrow
  \Cgg/Y_2 \longrightarrow
  \Cgg/Y_1 = \{1\}.
\end{gather*}
The completion 
\begin{gather*}
  \Cgghat :=\varprojlim_i\Cgg/Y_i
\end{gather*}
is called the {\em group of homology cylinders}\footnote{The group 
$\Cgghat$ is different from the
{\em homology cobordism group} of homology cylinders introduced by
Levine \cite{Levine}, which is a non-trivial quotient of $\Cgg$.}.  
Conjecturally, the canonical homomorphism $\Cgg \to \Cgghat$
is injective.

Denoting by $Y_i\Cgg$ the submonoid of $\Cgg$ consisting of
homology cylinders which are $Y_i$-equivalent to the trivial cylinder, 
one obtains the \emph{$Y$-filtration}
\begin{equation}
\label{eq:Y-filtration}
\Cgg = Y_1\Cgg \supset  Y_2\Cgg \supset Y_3\Cgg \supset \cdots
\end{equation}
for the monoid $\Cgg$. The quotient monoid $Y_k \Cgg /Y_l$ is a subgroup of
$\Cgg/Y_l$ for all $l\ge k\ge 1$ and, furthermore, the inclusion 
\begin{equation}
\label{eq:N-series}
\left[\ Y_{j}\Cgg/Y_l\ ,\ Y_{k} \Cgg /Y_l\ \right] \subset Y_{j+k} \Cgg/Y_l
\end{equation}
is satisfied for all $j,k\geq 1$ and $l\geq j+k$.
Thus, there is a graded Lie algebra over $\Z$
$$
\Gr^Y \Cgg := \bigoplus_{i\geq 1} Y_i\Cgg /Y_{i+1},
$$
which we call the {\em Lie algebra of homology cylinders}.\\

As proposed in \cite{Habiro} and established in \cite{CHM}, 
there is a diagrammatic version of the Lie algebra $\Gr^Y\Cgg\oQ$.
Similar diagrammatic constructions have also been considered 
by Garoufalidis and Levine \cite{GL_tree-level} and by Habegger \cite{Habegger}.
Our diagrammatic description of $\Gr^Y\Cgg\oQ$ involves a graded Lie algebra of Jacobi diagrams
\begin{gather*}
  \A^{<,c}\sHQ=\bigoplus_{i\ge1}\A^{<,c}_i\sHQ.
\end{gather*}
Here, the vector space $\A^{<,c}_i\sHQ$ is spanned by connected Jacobi diagrams with
$i$ internal vertices and with external vertices totally ordered and labeled by elements of $H_\Q$,
modulo the AS, IHX, STU-like and multilinearity relations. 
The following theorem is essentially proved in \cite{CHM}, see \S~\ref{subsec:LMO}. 

\begin{theorem}
  \label{th:surgery-LMO}
  For $g\ge0$, there are graded Lie algebra isomorphisms
  \begin{gather}
    \label{e5}
    \A^{<,c}\sHQ \overset{\psi}{\underset{\LMO}{\longrightleftarrows}}
    \Gr^Y\Cgg\oQ ,
  \end{gather}
  which are inverse to each other.
\end{theorem}

\noindent
The isomorphism $\psi$ is a ``surgery'' map sending each Jacobi diagram to surgery
along its corresponding graph clasper in the trivial cylinder over $\sgg$.
The isomorphism $\LMO$ comes from a functorial version of the Le--Murakami--Ohtsuki invariant \cite{LMO,BGRT} 
constructed in \cite{CHM}.  

The $Y$-filtration (\ref{eq:Y-filtration}) on $\Cgg$ induces a similar filtration on the group $\Cgghat$
\begin{equation}
\label{eq:Yhat-filtration}
\Cgghat = \Yhat_1\Cgghat \supset  \Yhat_2\Cgghat \supset \Yhat_3\Cgghat \supset \cdots
\end{equation}
which is called the \emph{$\Yhat$-filtration} and is defined by 
$$
\Yhat_j \Cgghat := \varprojlim_{i\geq j} Y_j \Cgg/ Y_i.
$$
In the appendix, we define the Malcev Lie algebra of a filtered group,
which corresponds to the usual notion of Malcev Lie algebra when the group is filtered by the lower central series.
We also extend to this setting some well-known properties of Malcev Lie algebras and Malcev completions.
In particular, we consider in \S~\ref{subsec:Malcev} the Malcev Lie algebra of the group $\Cgghat$ 
endowed with the $\Yhat$-filtration and prove it to be isomorphic  to the  completion of $\A^{<,c}\sHQ$.
This enhances Theorem \ref{th:surgery-LMO} since, by a general fact (proved in the appendix), 
the graded Lie algebra associated to a filtered group is canonically isomorphic to the graded Lie algebra associated 
to the canonical filtration on its Malcev Lie algebra.

In \S~\ref{sec:SJD} we give an alternative description of the Lie algebra $\A^{<,c}\sHQ$. 
For this, we consider the graded vector space $\A^c\sHQ$ spanned by  connected Jacobi
diagrams with external vertices labeled by elements of $H_\Q$, 
subject to the AS, IHX and multilinearity relations.
There is no ordering of the external vertices anymore.  
We define a Lie algebra structure on $\A^c\sHQ$, which is isomorphic to that
of $\A^{<,c}\sHQ$ via a ``symmetrization'' map
\begin{gather*}
  \chi\colon\A^c\sHQ\simeqto \A^{<,c}\sHQ.
\end{gather*}
The Lie algebra  $\A^c(H_\Q)$ consists of the primitive elements
of a  Hopf algebra $\A\sHQ$ of ``symplectic Jacobi diagrams'', 
whose associative multiplication is a diagrammatic analogue of the Moyal--Weyl product.
This analogy is justified by considering weight systems associated to metrized Lie algebras.

\subsection{The mapping cylinder construction} 

As proposed by the first author in \cite{Habiro}, the injective monoid homomorphism
\begin{equation}
\label{eq:mapping_cylinder}
\mapcyl\colon\Igg \hookrightarrow \Cgg
\end{equation}
defined by the mapping cylinder construction serves as a useful tool
in the study of the Torelli group.
It follows from the inclusion \eqref{eq:N-series} that
$\mapcyl$ sends the lower central series of $\Igg$ to the $Y$-filtration of $\Cgg$:
\begin{gather*}
  \mapcyl(\Gamma_i \Igg) \subset Y_i \Cgg \quad \quad   \text{for all $i\ge1$}.
\end{gather*}
So, $\mapcyl$ induces a group homomorphism
$$
\mapcylhat: \Igghat \longrightarrow \Cgghat,
$$
as well as a graded Lie algebra homomorphism
\begin{equation}
\label{eq:mapping_cylinder_graded}
\Gr \mapcyl\colon \Gr^\Gamma \Igg \longrightarrow \Gr^Y \Cgg.
\end{equation}
It is natural to ask whether the homomorphisms $\mapcylhat$ and
$\Gr\mapcyl$ are injective or not.
For example, in degree $1$, the homomorphism
\begin{gather*}
  \Gr_1 \mapcyl: \Igg/[\Igg,\Igg] \longrightarrow \Cgg/Y_2  
\end{gather*}
is an isomorphism for $g\geq 3$ \cite{Habiro,MM}.
\begin{question}
\label{ques:strong_injectivity}
Is the graded Lie algebra homomorphism $\Gr \mapcyl\colon \Gr^\Gamma \Igg \to \Gr^Y \Cgg$
injective when $g\geq 3$?
\end{question}
\noindent
If $g=2$, $\Gr \mapcyl$ is certainly not injective because
$\Gr^\Gamma_1 \mathcal{I}_{2,1}$ is not finitely generated\footnote{
This follows, for instance, from the fact that the Torelli group of a closed surface of genus $2$
is free of infinite rank \cite{Mess}.}.\\

The question has been asked by the first author in \cite{Habiro} 
to clarify the relationship between Hain's presentation of the Torelli Lie algebra 
and diagrammatic descriptions of the Lie algebra of homology cylinders.
The following result is a starting point of studies in this direction. 

\begin{theorem}
  \label{th:diagram}
  If $g\ge3$, then the following diagram in the category of graded Lie algebras 
  with $\Sp(H_\Q)$-actions is commutative:
  \begin{equation}
    \label{eq:diagram}
    \xymatrix{
      {\Lie\left(\LHQ\right) \left/ {\rm{R}}(\Igg) \right. }
      \ar[d]_-{\overline{J}}^-\simeq \ar[rr]^-{\overline{Y}} 
      && {\A^{<,c}\sHQ} \ar[d]_-{\psi}^-\simeq\\
	 {\Gr^\Gamma \Igg \oQ} \ar[rr]^{\Gr \mapcyl \oQ}
	 && {\Gr^Y \Cgg \oQ}. \ar@/_1pc/[u]_-{\LMO}
    }
  \end{equation}
  Here, $\overline{Y}$ is induced by the Lie algebra homomorphism
  $Y: \Lie\left(\LHQ\right)\to \A^{<,c}\sHQ$ which, in degree $1$,
  sends the trivector $x\wedge y\wedge z$ to the $Y$-graph $\Ygraphtop{x}{y}{z}$.
\end{theorem}

\noindent
Thus, Theorem \ref{th:diagram} reduces the study of the map $\Gr \mapcyl \oQ$ 
to the understanding of the algebraically-defined map
\begin{gather}
  \label{e13}
  \overline{Y}\colon\Lie(\LHQ)/{\rm{R}}(\Igg)\longrightarrow \A^{<,c}\sHQ,
\end{gather}
the source of which  is described by Hain's result (Theorem \ref{th:Hain}). 
Theorem \ref{th:diagram} is proved in \S~\ref{sec:mapping_cylinder_construction},
where the symplectic actions in diagram (\ref{eq:diagram}) are also specified.

In \S~\ref{sec:quadratic}  we use the Lie algebra $\AcHQ\sHQ$ 
to compute the Lie bracket of $\AgcHQ\sHQ$ in degree $1+1$. 
Thus, we obtain the following result:

\begin{theorem}
\label{th:degree_2}
If $g\ge3$, then the kernel of 
$Y_2: \Lie_2(\Lambda^3 H_\Q) \to \A_2^{<,c}\sHQ$ 
coincides with the submodule ${\rm{R}}_2(\Igg)$.
\end{theorem}
\noindent
This can be regarded as a diagrammatic formulation 
of previous results by Morita \cite{Morita_Casson_I,Morita_Casson_II} and Hain \cite{Hain},
so that some computations done in \S~\ref{sec:quadratic} to prove it 
should be essentially well-known to experts.
We also identify the image of $\chi_2^{-1} \circ Y_2\colon\Lie_2(\LHQ)\to\Ac{2}\sHQ$
with the {\em even part} $\A^{c}_{2,\ev}\sHQ$,
consisting of linear combinations of Jacobi diagrams whose first Betti number is even. 

Theorems \ref{th:diagram} and  \ref{th:degree_2}
give a partial answer to Question \ref{ques:strong_injectivity} in degree $2$: 

\begin{corollary}
  \label{r13}
  If $g\ge3$, then the map
  $\Gr_2\mapcyl\oQ\colon\Gr^\Gamma_2\Igg\oQ\longrightarrow\Gr^Y_2\Cgg\oQ$
  is injective.  
\end{corollary}
\noindent

There is also a ``stabilized'' form of Question \ref{ques:strong_injectivity}.
\begin{conjecture}
  \label{r3}
  For $g\geq 3$, the map
  \begin{gather}
    \label{e11}
    \glim\Gr \mapcyl\colon
    \glim\Gr^\Gamma \Igg \longrightarrow \glim\Gr^Y \Cgg
  \end{gather}
  is injective, where the spaces and the map are induced by a sequence
  of surface inclusions
  $\Sigma_{0,1}\subset \Sigma_{1,1}\subset \Sigma_{2,1}\subset \cdots$.
\end{conjecture}
\noindent
Conjecture \ref{r3} is equivalent to the conjecture stated in \cite{Habiro} that the lower central
 series of $\Igg$ and the restriction to $\Igg$ of the $Y$-filtration of $\Cgg$ are stably equal.
Some stability properties are discussed in \S~\ref{sec:stability}.

The case of a closed connected oriented surface $\Sigma_g$ of genus $g$ is also considered in \S~\ref{sec:closed_case},
where analogues of Theorems \ref{th:surgery-LMO}, \ref{th:diagram} and \ref{th:degree_2} are proved.

The final  \S~\ref{sec:conclusion} concludes with further problems and remarks.
The problems of determining the kernel and the image of the map $Y$ in higher degree are discussed, 
and such problems are related to questions about Johnson homomorphisms.\\

\noindent
\textbf{Acknowledgements.} The authors thank the anonymous referee for her/his careful reading of the manuscript.
The first author was partially supported by Grant-in-Aid for Scientific Research (C) 19540077.

\vspace{0.5cm} 

\section{Diagrammatic description of the Lie algebra of homology cylinders}

\label{sec:surgery-LMO}

In this section, we recall from \cite{Habiro,CHM} the main ingredients to obtain Theorem \ref{th:surgery-LMO},
which gives a diagrammatic description of the Lie algebra of homology cylinders.
Furthermore, we produce from the LMO invariant a diagrammatric description
of the Malcev Lie algebra of the group of homology cylinders.

\subsection{The algebra $\A^<\sHQ$ and the Lie algebra $\A^{<,c}\sHQ$}

\label{subsec:ordered_Jacobi_diagrams}

First of all, we recall the definition of the cocommutative Hopf algebra $\A^<\sHQ$,
which was introduced in \cite{Habiro} and used in \cite{CHM}. 

A \emph{Jacobi diagram} is a finite graph whose vertices
have valence $1$ (\emph{external} vertices) or $3$ (\emph{internal}
vertices).  Each internal vertex is \emph{oriented}, in the sense that
its incident edges are cyclically ordered.  A Jacobi diagram is
\emph{colored} by a set $S$ if a map from the set of its external
vertices to $S$ is specified. A \emph{strut} is a Jacobi diagram with
only two external vertices and no internal vertex.  
The {\em internal degree} of a Jacobi diagram is the number of its internal vertices.

We define the following $\Q$-vector space
$$
\A^<(H_\Q) := 
\frac{\Q\cdot \left\{ \begin{array}{c} \hbox{Jacobi diagrams without strut component and with}\\
\hbox{external vertices totally ordered and colored by } \HQ  \end{array} \right\}}
{\hbox{AS, IHX, STU-like, multilinearity}},
$$
which is also denoted simply by $\AgHQ$.
Here, the AS and IHX relations among Jacobi diagrams are the usual ones, namely\\

\centerline{\relabelbox \epsfxsize 3.5truein \epsfbox{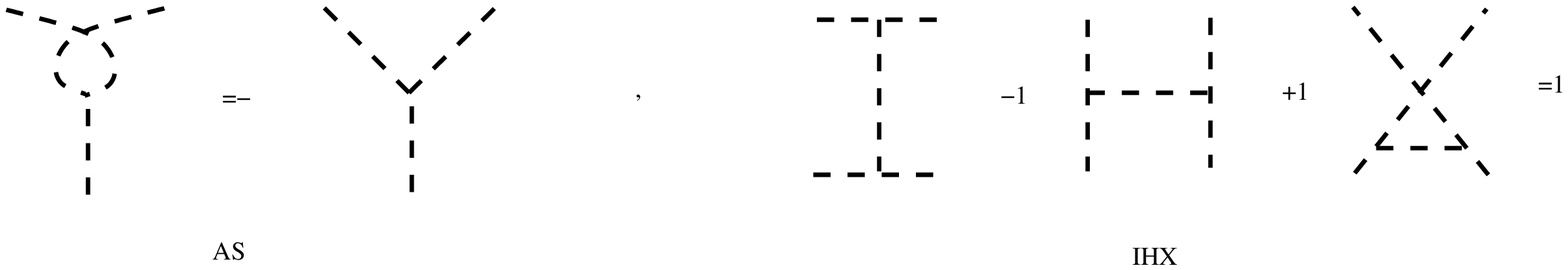}
\adjustrelabel <0cm,-0.2cm> {AS}{AS}
\adjustrelabel <0cm,-0.2cm> {IHX}{IHX}
\adjustrelabel <0cm,-0cm> {=-}{$= \ -$}
\adjustrelabel <0cm,-0cm> {,}{,}
\adjustrelabel <0cm,-0cm> {-1}{$-$}
\adjustrelabel <0cm,-0.05cm> {+1}{$+$}
\adjustrelabel <0cm,-0.1cm> {=1}{$=0$,}
\endrelabelbox}
\vspace{0.5cm} 
\noindent
where, as usual, the vertex orientation is given by the trigonometric orientation.
The STU-like and multilinearity relations are defined by\\

\centerline{\relabelbox \small
\epsfxsize 3.5truein \epsfbox{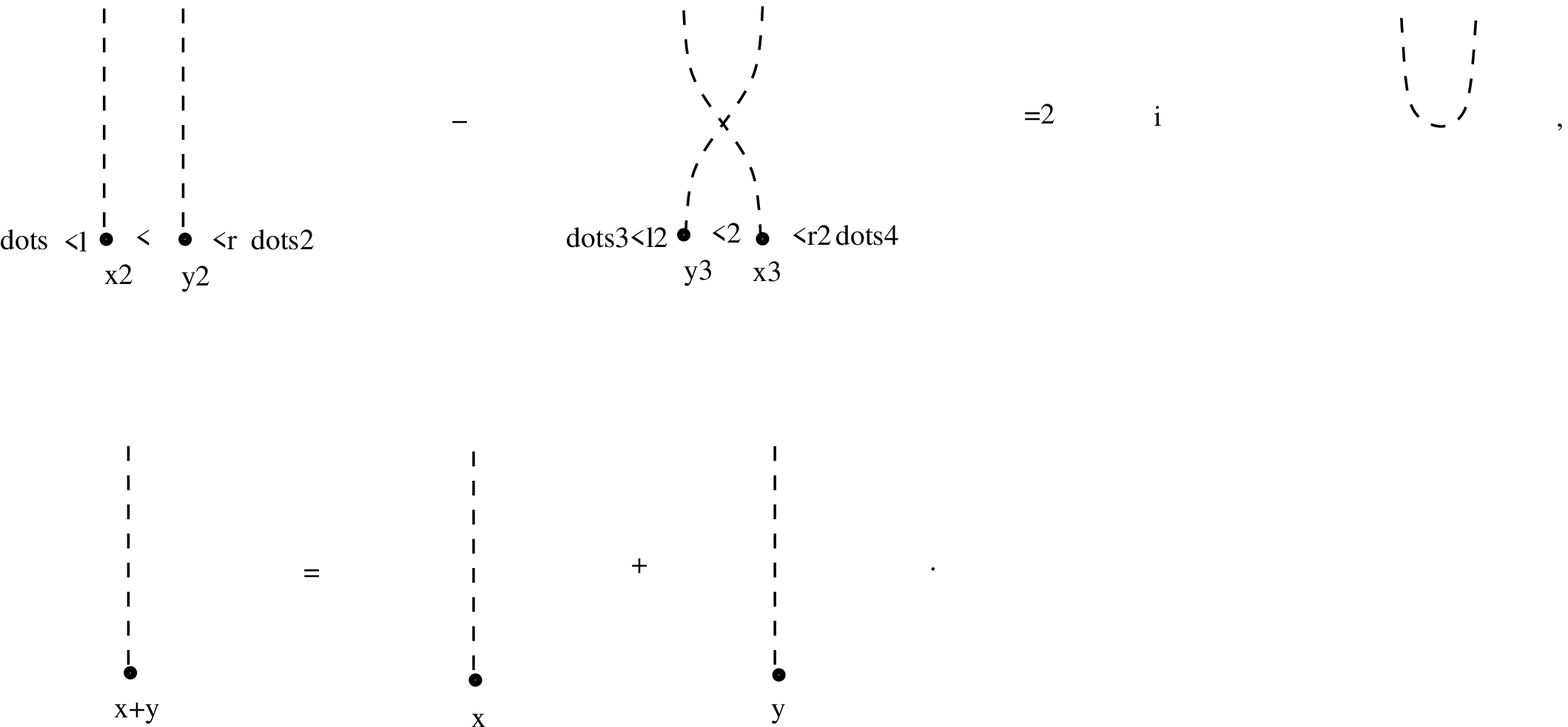}
\adjustrelabel <-0.05cm,-0.05cm> {x}{$x$}
\adjustrelabel <-0.05cm,-0.05cm> {y}{$y$}
\adjustrelabel <-0.35cm,-0.05cm> {x+y}{$x+y$}
\adjustrelabel <-0.1cm,-0.1cm> {x2}{$x$}
\adjustrelabel <-0.05cm,-0.05cm> {y2}{$y$}
\adjustrelabel <-0.05cm,-0.05cm> {x3}{$x$}
\adjustrelabel <-0.1cm,-0.05cm> {y3}{$y$}
\adjustrelabel <-0.1cm,-0.05cm> {<}{$<$}
\adjustrelabel <-0.05cm,-0.05cm> {<2}{$<$}
\adjustrelabel <-0.1cm,-0cm> {<l}{$<$}
\adjustrelabel <-0.05cm,-0cm> {<l2}{$<$}
\adjustrelabel <-0.05cm,-0.05cm> {<r}{$<$}
\adjustrelabel <-0.05cm,-0.05cm> {<r2}{$<$}
\adjustrelabel <-0.2cm,-0cm> {dots}{$\cdots$}
\adjustrelabel <0.05cm,-0.05cm> {dots2}{$\cdots$}
\adjustrelabel <-0.2cm,-0cm> {dots3}{$\cdots$}
\adjustrelabel <0.05cm,-0.05cm> {dots4}{$\cdots$}
\adjustrelabel <-0.2cm,-0cm> {-}{$-$}
\adjustrelabel <-0.1cm,-0.05cm> {+}{$+$}
\adjustrelabel <-0.3cm,-0cm> {i}{$\omega(x,y)$}
\adjustrelabel <0cm,-0cm> {=}{$=$}
\adjustrelabel <-0.2cm,-0cm> {=2}{$=$}
\adjustrelabel <0cm,-0cm> {,}{,}
\adjustrelabel <0cm,-0cm> {.}{,}
\endrelabelbox}
\vspace{0.5cm}

\noindent
where $x,y\in H_\Q$.
The space $\Ag{}\sHQ$ is graded by the internal degree of Jacobi diagrams.
Its degree completion will also be denoted by $\Ag{}\sHQ$.

There is also a space $\A^{<}(-H_\Q)$ defined as $\A^{<}(H_\Q)$ except that
one uses the symplectic form $-\omega$ in  the STU-like relation instead of $\omega$.
There is a canonical isomorphism 
$$
s: \A^{<}(-H_\Q) \longrightarrow \A^{<}(H_\Q)
$$
defined by $s(D) = (-1)^{\chi(D)}D$ for any Jacobi diagram $D$ with Euler characteristic $\chi(D)$.

\begin{remark}
Note that $\A^{<}(H_\Q)$ depends not only on the vector space
$\HQ$ but also on the symplectic form $\omega$, with which $\HQ$ is implicitly equipped.
The space $\A^{<}(H_\Q)$ is denoted by $\A(\Sigma_{g,1})$ in \cite{Habiro},
while $\A^{<}(-H_\Q)$ corresponds to the space $\A(\Sigma_{g,1})$ in \cite{CHM}.
\end{remark}

The multiplication $D \osqcup E$ of two Jacobi diagrams $D,E \in \AgHQ$
is the disjoint union of $D$ and $E$, the external vertices of $E$ being considered
as ``larger'' than those of $D$. Then, $\AgHQ$ is an associative algebra
whose unit element is the empty diagram.
Like many other algebras of Jacobi diagrams in the literature, the
algebra $\AgHQ$ has a structure of a cocommutative Hopf algebra \cite{CHM}. 
The comultiplication $\Delta: \AgHQ \to \AgHQ \otimes \AgHQ$ 
for a Jacobi diagram $D\in\AgHQ$ is defined by
\begin{gather*}
  \Delta(D)=\sum_{D=D'\sqcup D''} D'\otimes D'',
\end{gather*}
where the sum is over all the decompositions of $D$ into two families
of connected components $D',D''$;
in the right-hand side, the orders of the external vertices of
$D'$ and $D''$ are induced by that of $D$.
The counit $\varepsilon\colon\AgHQ\to\Q$ for a  diagram $D\in\AgHQ$ is defined by
\begin{gather*}
  \varepsilon(D) =
  \begin{cases}
    1&\text{if $D$ is empty},\\
    0&\text{otherwise}.
  \end{cases}
\end{gather*}
The antipode $S\colon\AgHQ\to\AgHQ$ is the unique algebra
anti-automorphism satisfying $S(D) = -D$ for each non-empty connected Jacobi diagram $D\in\AgHQ$.

As is well known, the set of primitive elements $\Prim(A)$ in a Hopf algebra $A$ forms a Lie algebra, 
with the Lie bracket given by $[x,y]=xy-yx$.  
Thus, we have the Lie algebra $\Prim(\AgHQ)$ of primitives in $\AgHQ$.
Moreover, the Hopf algebra $\AgHQ$ being cocommutative, the Milnor--Moore theorem asserts
that $\AgHQ$ is canonically isomorphic to the universal enveloping algebra $U\Prim(\AgHQ)$ of $\Prim(\AgHQ)$.
Let $\A^{<,c}\sHQ$ (or simply $\AgcHQ$) 
denote the subspace of $\A^<\sHQ$ spanned by the connected Jacobi diagrams.

\begin{lemma}
\label{r4}
We have $\AgcHQ=\Prim(\AgHQ)$.
\end{lemma}

\begin{proof}
Clearly, connected Jacobi diagrams are primitive.  Thus we have
$\AgcHQ\subset \Prim(\AgHQ)$.
Using the STU-like relation, one can check that $\AgcHQ$ is a Lie
subalgebra of $\Prim(\AgHQ)$
and that the algebra $\AgHQ$ is generated by $\AgcHQ$.
Since $\AgHQ=U\Prim(\AgHQ)$, it follows from the Poincar\'e--Birkhoff--Witt
theorem that $\AgcHQ=\Prim(\AgHQ)$.
\end{proof}

The natural $\SpHQ$-action on $\HQ$ induces an $\SpHQ$-action on
$\A^<\sHQ$, which is easily seen to be compatible with the Hopf algebra structure. 
In particular, $\A^{<,c}\sHQ$ is equipped with an
$\SpHQ$-action compatible with the Lie algebra structure.

\subsection{The surgery map $\psi$}

\label{subsec:surgery}

As suggested in \cite{Habiro}, there is a canonical linear isomorphism
$$
\psi: \AgcHQ \longrightarrow \Gr^Y \Cgg\oQ
$$
defined by mapping each connected Jacobi diagram $D$ to
the $3$-manifold obtained from the cylinder $\sgg \times [-1,1]$
by surgery along a graph clasper $C(D)$ obtained from $D$ as follows:
\begin{itemize}
\item Thicken $D$ to an oriented surface using the vertex-orientation of $D$
(vertices are thickened to disks, and edges to bands). 
Cut a smaller disk in the interior of each disk
that has been produced from an external vertex of $D$.
This leads to an oriented compact surface $S(D)$,
decomposed into disks, bands and annuli (corresponding to internal
vertices, edges and external vertices of $D$ respectively).
Use the induced orientation on $\partial S(D)$ to orient the cores of the annuli.
\item Next, embed $S(D)$ into the interior of $\sgg \times [-1,1]$ in such a way that
each annulus of $S(D)$ represents in $\HQ$ the color of the corresponding external vertex of $D$.
Moreover, the annuli should be in disjoint ``horizontal slices'' of $\sgg \times [-1,1]$
and their ``vertical height'' along $[-1,1]$  should respect the total ordering of the external vertices of $D$.
Such an embedding defines a graph clasper $C(D)$ in $\sgg \times [-1,1]$.
\end{itemize}
That $\psi$ is well-defined and surjective follows from clasper calculus \cite{Habiro,Goussarov_clovers,GGP}. 
For instance, the fact that the STU-like relation is satisfied in $\Gr^Y \Cgg$
is proved using Move 2 and Move 7 from \cite{Habiro}.
The detail of the degree $1$ case, where the STU-like relation amounts to saying that
the order of the external vertices does not matter, is done in \cite{MM}.
The higher degree case, where one has to consider also the IHX relation,
is similar but needs the zip construction \cite{Habiro}.
Using clasper calculus, one can also check that $\psi$ is a Lie algebra homomorphism.
See also \cite{GL_tree-level} and \cite{Habegger} for similar constructions.

To prove the injectivity of $\psi$, one needs the LMO invariant.

\subsection{The LMO map}

\label{subsec:LMO}

In a joint work with Cheptea \cite{CHM}, the authors extended the LMO
invariant of homology $3$-spheres to a functor on a category
of Lagrangian cobordisms, which are cobordisms between
surfaces with connected boundary, satisfying certain homological conditions.
(Some extensions of the LMO invariant to cobordisms were previously
constructed by Murakami and Ohtsuki \cite{MO} and by Cheptea and Le \cite{CL}.)
Since homology cylinders over $\sgg$ are Lagrangian cobordisms, 
the LMO functor restricts to a monoid homomorphism 
$$
\ZtildeY: \Cgg \longrightarrow \AY(\set{g}^+ \cup \set{g}^-)
$$
with values in a certain complete Hopf algebra of Jacobi diagrams.
The latter is isomorphic via a certain map $\varphi$ 
defined in \cite{CHM} to  $\A^<(-H_\Q)$.
Thus, the composition $s \circ \varphi \circ \ZtildeY$ defines a monoid homomorphism
\begin{equation}
\label{eq:LMO}
\Cgg \longrightarrow \A^<(H_\Q).
\end{equation}
Since $\ZtildeY$ is an isomorphism at the level of graded Lie algebras \cite{CHM},
the monoid homomorphism (\ref{eq:LMO}) induces a graded Lie algebra isomorphism
\begin{equation}
\label{eq:graded_LMO}
  \LMO\colon \Gr^Y\Cgg\oQ \simeqto \A^{<,c}(H_\Q).
\end{equation}
Taking care of signs, we also deduce from \cite{CHM} that $\psi$ and $\LMO$ 
are inverse to each other. Thus we have Theorem \ref{th:surgery-LMO}.

\subsection{The Malcev Lie algebra of the group of homology cylinders}

\label{subsec:Malcev}

The LMO functor can be used to prove more than Theorem \ref{th:surgery-LMO}:
It also produces a diagrammatic description of the Malcev Lie algebra of $\Cgghat$.
Here, the Malcev Lie algebra is defined with respect to the $\Yhat$-filtration (\ref{eq:Yhat-filtration})
rather than the lower central series of $\Cgghat$, 
the former being more natural than the latter from the point of view of finite-type invariants. 
Malcev completions and Malcev Lie algebras of filtered groups are presented in the appendix.
 
To deal with the Malcev Lie algebra of $\Cgghat$, 
we come back to the monoid homomorphism (\ref{eq:LMO}), which we also denote by $\LMO$:
\begin{equation}
\label{eq:LMO_homomorphism}
\LMO : \Cgg \longrightarrow \A^<.
\end{equation}
It is shown in \cite{CHM} that, if an $M \in \Cgg$ is $Y_i$-equivalent to the trivial cylinder, then
$\ZtildeY(M)-\varnothing$ starts in internal degree $i$. So, the $\LMO$ map induces a multiplicative map
$\LMO : \Cgg/Y_i \longrightarrow \A^</\A^<_{\geq i}$ for all $i\geq 1$, 
where $\A^<_{\geq i}$ denotes the internal degree at least $i$ part of $\A^<$.
By passing to the limit, we obtain 
\begin{equation}
\label{eq:LMO_homomorphism_complete}
\LMO : \Cgghat \longrightarrow \A^<.
\end{equation}
Since $\ZtildeY$ takes group-like values and since $\varphi$ and $s$
are Hopf algebra isomorphisms, the map (\ref{eq:LMO_homomorphism}) takes group-like values 
and so, by continuity of the coproduct, the map (\ref{eq:LMO_homomorphism_complete}) does too.
Thus, we get a Hopf algebra homomorphism
\begin{equation}
\label{eq:big_LMO}
\LMO : \Q[\Cgghat] \longrightarrow \A^<.
\end{equation}
There are two filtrations on the group algebra $\Q[\Cgghat]$. 
On one hand, let $F$ be the filtration defined in the appendix at (\ref{eq:filtration_group_algebra})
and induced by the $\Yhat$-filtration on the group $\Cgghat$.
On the other hand, let $F'$ be the filtration defined by rational finite-type invariants:
An $x\in \Q[\Cgghat]$ is declared to belong to $F'_i \Q[\Cgghat]$
if $f(x)=0$ for any finite-type invariant $f: \Cgg \to \Q$ of degree at most $i-1$.
By clasper calculus, it can be proved that $F=F'$ (see \cite{Habiro,Massuyeau}).
Since $\ZtildeY$ is universal among rational finite-type invariants \cite{CHM},
(\ref{eq:big_LMO}) induces  a monomorphism
$$
\LMO : \Q[\Cgghat]/ F_i \Q[\Cgghat] \longrightarrow \A^</\A^<_{\geq i}.
$$
Actually, this map is an isomorphism since it is bijective at the graded level \cite{CHM}.
Thus, passing to the limit, we finally obtain an isomorphism
$$
\LMO : \widehat{\Q}[\Cgghat] \longrightarrow \A^<
$$
of complete Hopf algebras, where $\widehat{\Q}[\Cgghat]$ denotes the completion
of $\Q[\Cgghat]$ with respect to the filtration $F$. Thus, we deduce the following 

\begin{theorem}
Let $\MalcevGroup{\Cgghat}$ be the  Malcev completion of the group $\Cgghat$ 
endowed with the $\Yhat$-filtration,
and let $\MalcevLie{\Cgghat}$ be its Malcev Lie algebra. 
Then, the LMO invariant induces an isomorphism of filtered groups
$$
\LMO:  \MalcevGroup{\Cgghat} \simeqto \GLike(\A^<)
$$
as well as an isomorphism of filtered Lie algebras
$$
\LMO:  \MalcevLie{\Cgghat} \simeqto \Prim(\A^<).
$$ 
\end{theorem}

\noindent
The second part of this statement and Theorem \ref{th:graded_Malcev_Lie_algebra}
gives back Theorem \ref{th:surgery-LMO}.
Besides, it proves that the filtration on the Malcev Lie algebra of $\Cgghat$ 
comes from a grading, which is not true for an arbitrary filtered group.

\vspace{0.5cm}

\section{The Lie algebra of symplectic Jacobi diagrams}

\label{sec:SJD}

In this section, we define the algebra of symplectic Jacobi diagrams,
which is isomorphic to the algebras $\A^<\sHQ$ and is more convenient in some occasions.  
We interpret the multiplication of symplectic Jacobi diagrams as an analogue of the Moyal--Weyl product.

\subsection{The algebra of symplectic Jacobi diagrams}

\label{subsec:SJD}

We define the following vector space
$$
\A(H_\Q) := 
\frac{\Q\cdot \left\{ \begin{array}{c} \hbox{Jacobi diagrams without strut component}\\ 
\hbox{and with external vertices colored by } \HQ  \end{array} \right\}}
{\hbox{AS, IHX, multilinearity}},
$$
which is also simply denoted by $\AHQ$.
It is graded by the internal degree of Jacobi diagrams,
and its degree completion is also denoted by $\A\sHQ$.

There is a  graded linear map
$$
\chi : \AHQ \longrightarrow \AgHQ
$$
defined, for all Jacobi diagram $D \in \AHQ$ with $e$  external vertices, by
$$
\chi(D) := \frac{1}{e!} \cdot \left(\hbox{sum of all ways of ordering
 the $e$ external vertices of $D$} \right).
$$

\begin{proposition} 
\label{lem:chi}
The ``symmetrization'' map $\chi$ is an isomorphism. Its inverse is given on
a Jacobi diagram $D\in \AgHQ$, with external vertices $v_1 < \cdots < v_e$  
colored by $c(v_1),\dots,c(v_e) \in \HQ$ respectively, by the formula
\begin{equation}
\label{eq:chi_inverse}
\chi^{-1}(D) = D+\sum_{p=1}^{[e/2]} \frac{1}{2^p} \sum_{\{i_1,j_1\},\dots, \{i_p,j_p\}}\
\prod_{k=1}^p \omega\left(c(v_{i_k}),c(v_{j_k})\right)\ \cdot
D_{(v_{i_1} = v_{j_1},\dots,v_{i_p} = v_{j_p})}.
\end{equation}
Here, the second sum is taken over all ways of doing $p$ pairings $\{i_1,j_1\},\dots, \{i_p,j_p\}$ 
inside the set $\{1,\dots,e\}$ (with $i_1 < \cdots < i_p$ and $i_1<j_1, \dots,i_p<j_p$)
and the diagram $D_{(v_{i_1} = v_{j_1},\dots,v_{i_p} = v_{j_p})}$ is obtained 
from $D$ by gluing the vertices that are paired (and by forgetting the order of the remaining vertices).
\end{proposition}

\begin{proof}
Let $\sigma(D)$ be the quantity defined by the right term of (\ref{eq:chi_inverse}),
for all Jacobi diagram $D$ colored by $\HQ$ and with external vertices  $v_1 < \cdots < v_e$.
Let $(l,l+1)\cdot D$ be the same diagram, but with the order of $v_l$ and $v_{l+1}$ reversed.
Then, in the difference $\sigma(D) - \sigma\left((l,l+1)\cdot D\right)$, all terms
cancel except for those corresponding to pairings that match $v_l$ and $v_{l+1}$:
\begin{eqnarray*}
&&\sigma(D) - \sigma\left((l,l+1)\cdot D\right)\\
&=& \sum_{p=1}^{[e/2]} \frac{1}{2^{p-1}} 
\sum_{\substack{\{i_1,j_1\},\dots, \{i_p,j_p\} \\ \exists r, (i_r,j_r)=(l,l+1)  }}\
\prod_{k=1}^p \omega\left(c(v_{i_k}),c(v_{j_k})\right)\  \cdot
D_{(v_{i_1} = v_{j_1},\dots,v_{i_p} = v_{j_p})}\\
&=& \omega\left(c(v_l),c(v_{l+1})\right) \cdot \sigma\left(D_{(v_l=v_{l+1})}\right).
\end{eqnarray*}
Thus, the STU-like relation is satisfied, 
and we get a linear map $\sigma:\AgHQ \to \AHQ$.

Let $D$ be a Jacobi diagram colored by $\HQ$ and with external vertices  $v_1 < \cdots < v_e$.
The STU-like relation implies that $\chi(D) = D$ modulo some terms with fewer external vertices.
(Here, the $D$ to which $\chi$ applies is obtained from $D$ by forgetting the order.)
Moreover,  $\chi(D) = D$ if $D$ has no external vertex $(e=0)$. 
This proves, by an induction on $e$, that
$D$ belongs to the image of $\chi$. So, $\chi$ is surjective.

Thus, it is enough to prove that $\sigma \circ \chi$ is the identity. 
For this, we define the space
\begin{equation}
\label{eq:big_space}
\widetilde{\A}\sHQ := \frac{\Q\cdot \left\{ \begin{array}{c} 
\hbox{Jacobi diagrams without strut component and with}\\
\hbox{external vertices totally ordered and colored by } \HQ  \end{array} \right\}}
{\hbox{AS, IHX, multilinearity}},
\end{equation}
which we simply denote by $\AtildeHQ$.
(The space $\widetilde{\A}\sHQ$ is used also in the proof of Theorem \ref{th:weight_systems}.)
Thus, the quotient of $\AtildeHQ$ by the STU-like relation is $\AgHQ$
while its quotient by the ``forgetting orders'' relation is $\AHQ$.
Let $\widetilde{\gamma}: \AtildeHQ \to \AtildeHQ$ be the linear map
defined on each Jacobi diagram $D$ with external vertices $v_1 < \cdots < v_e$  by
$$
\widetilde{\gamma}(D) := \sum_{1 \leq i<j \leq e} \omega(c(v_i),c(v_j)) \cdot D_{(v_i=v_j)}
$$
if $e\geq 2$, and by $\widetilde{\gamma}(D)=0$ if $e=0,1$. 
Then, we define $\widetilde{\sigma}: \AtildeHQ \to \AtildeHQ$ by 
$$
\widetilde{\sigma} := \exp_\circ\left( \widetilde{\gamma}/2 \right) =
\sum_{k\geq 0} \frac{\widetilde{\gamma}^k}{2^k k!}.
$$
Let also $\widetilde{\chi}: \AtildeHQ \to \AtildeHQ$ be the ``symmetrization'' map 
sending all Jacobi diagram $D$ to
$$
\widetilde{\chi}(D) = \frac{1}{e!} \cdot 
\left(\hbox{sum of all ways of permuting the $e$ external vertices of $D$}\right).
$$
The following diagram is commutative:
$$
\xymatrix{
{\AtildeHQ} \ar[r]^-{\widetilde{\chi}} \ar@{->>}[d] & 
\!\!{\AtildeHQ} \ar[r]^-{\widetilde{\sigma}} \ar@{->>}[d] & {\AtildeHQ} \ar@{->>}[d]\\
{\AHQ} \ar[r]^-{\chi} & \!\! {\AgHQ} \ar[r]^-{\sigma} & {\AHQ.}
}
$$
Since $\widetilde{\gamma} \circ \widetilde{\chi} = 0$, we have 
$\widetilde{\sigma} \circ \widetilde{\chi} = \widetilde{\chi}$,
which implies that $\sigma \circ \chi =\Id$.
\end{proof}

Thus, we can pull back by $\chi$ the product on $\AgHQ$ 
to an associative multiplication $\star$ on the space $\AHQ$, i.e$.$ we set
\begin{gather}
  \label{e6}
  D\star E := \chi^{-1}\left(\chi(D)\osqcup\chi(E)\right)
\end{gather}
for all $D,E\in\AHQ$. More generally, the full Hopf algebra structure
on $\AgHQ$ gives one for $\AHQ$. The comultiplication in $\AHQ$ is
given on a Jacobi diagram $D$ by
\begin{gather*}
\Delta(D)=\sum_{D=D'\sqcup D''} D'\otimes D'',
\end{gather*}
the counit is given by $\varepsilon(D)=\delta_{D,\emptyset}$, and the
antipode is the unique algebra anti-automorphism satisfying 
$S(D) = -D$ if $D$ is connected and non-empty.
The primitive part $\Prim(\AHQ)$ of $\AHQ$
is the subspace $\AcHQ$ spanned by the connected diagrams.

\begin{definition}
The Hopf algebra of \emph{symplectic Jacobi diagrams} is 
$\left(\A,\varnothing,\star,\varepsilon,\Delta,S\right)$.
\end{definition}

The multiplication $\star$ of Jacobi diagrams can also be defined directly as follows:

\begin{proposition}
\label{prop:multiplication}
Let $D,E\in\AHQ$ be Jacobi diagrams colored by $\HQ$,
and whose sets of external vertices are denoted by $V$ and $W$ respectively.
Then, we have 
$$
D \star E = \sum_{\substack{V' \subset V,\ W' \subset W\\ 
\beta\ :\ V' \stackrel{\simeq}{\longrightarrow} W'}}\
\frac{1}{2^{|V'|}} \cdot \prod_{v\in V'} \omega\left(c(v),c(\beta(v))\right)\ \cdot (D \cup_\beta E).
$$
Here, the sum is taken over all ways of identifying a subset $V'$ of $V$ with a subset $W'$ of $W$,
and $D \cup_\beta E$ is obtained from $D \sqcup E$
by gluing each vertex $v \in V'$ to $\beta(v) \in W'$. 
\end{proposition}

\noindent
Consequently, the commutator $[D,E]_\star=D\star E -E \star D$ of $D$ and $E$ is given by
\begin{equation}
\label{eq:commutator}
[D,E]_\star = \sum_{\substack{ \beta\ :\  V \supset V' \stackrel{\simeq}{\longrightarrow} W'  \subset W\\
|V'| = |W'| \equiv 1\ \mod\ 2}}\
\frac{1}{2^{|V'|-1}}\prod_{v\in V'} \omega(c(v),c(\beta(v)))\
\cdot (D \cup_\beta E),
\end{equation}
where the sum is taken over all ways of identifying a subset $V'$ of $V$ 
of odd cardinality with a subset $W'$ of $W$.

\begin{proof}[Proof of Proposition \ref{prop:multiplication}]
Denote the external vertices of $D$ by $v_1,\dots,v_d$, and those of $E$ by $w_1,\dots,w_e$.
By \eqref{e6}, we have
$$
D \star E  = \sum_{\delta \in S_d,\ \varepsilon \in S_e}  
\frac{1}{d! \cdot e!} \cdot \chi^{-1}\left(  D^\delta\ \osqcup\  E^\varepsilon \right),
$$
where $\delta$ is a permutation of $\{1,\dots,d\}$ and $D^\delta \in \AgHQ$ is obtained from $D$ 
by ordering its external vertices as  $v_{\delta(1)} < \cdots < v_{\delta(d)}$. 
The diagram $E^\varepsilon$ is defined similarly from the permutation $\varepsilon$ of $\{1,\dots,e\}$.
If we apply formula (\ref{eq:chi_inverse}) 
to $\chi^{-1}\left(  D^\delta\ \osqcup\  E^\varepsilon \right)$, 
two kinds of terms appear in the resulting sum: Either, the gluings performed on 
$D^\delta\ \osqcup\  E^\varepsilon$ are all ``mixed'', 
or at least one gluing is not mixed and involves, say, two external vertices of $D^\delta$.
In the latter case, the corresponding term will appear with an opposite sign in 
$\chi^{-1}(D^{\tau \circ \delta}\ \osqcup\ E^\varepsilon)$ 
where $\tau \in S_d$ is the transposition of the indices
of those two external vertices. Thus, we can assume that formula 
(\ref{eq:chi_inverse}) applied to $\chi^{-1}\left(  D^\delta\ \osqcup\ E^\varepsilon \right)$
involves only ``mixed'' gluing, in which case the orderings of the external vertices in $D^\delta$
and in $E^\varepsilon$ do not matter. The conclusion follows.
\end{proof}

There is an obvious action of $\Sp(H_\Q)$ on $\A\sHQ$,
such that $\chi$ is $\SpHQ$-equivariant.
Thus, the Hopf algebra structure on $\A\sHQ$ is compatible with this $\SpHQ$-action.
In particular, the Lie bracket $[-,-]_\star$ on $\A^c(H_\Q)$ is $\SpHQ$-equivariant.

\begin{remark}
  \label{r9}
  Garoufalidis and Levine \cite{GL_tree-level} attempted to define a
  Lie bracket on the graded space $\Ac{}$, but, as pointed out by
  Habegger and Sorger \cite{HS}, their Lie bracket is \emph{not}
  $\SpHQ$-equivariant so that \cite[Theorem 6]{GL_tree-level} fails.
  Yet the approach in \cite{GL_tree-level} can be fixed at the tree
  level, as is done in \cite[\S 3]{HS}.
\end{remark}

\subsection{Loop filtration}

\label{subsec:loop_filtration}

The {\em loop degree} of a Jacobi diagram is defined to be its first
Betti number.\footnote{In the literature, the loop degree is sometimes
defined as the first Betti number minus $1$.}  
For example, the loop degree of a
tree diagram is $0$, and the loop degree of $\thetagraph$ is $2$.
The loop degree is additive under disjoint union of diagrams.

Let $\F_k(\AgHQ)$ be the subspace of $\AgHQ$ spanned by the
Jacobi diagrams of loop degree at least $k$.  We have a filtration
$$
  \AgHQ=
  \F_0(\AgHQ)\supset
  \F_1(\AgHQ)\supset
  \F_2(\AgHQ)\supset\cdots.
$$
This filtration induces a filtration on $\A_i^<$, for each $i\geq 1$, and is an algebra filtration:
\begin{gather*}
  \F_k(\AgHQ)\;\osqcup\;\F_l(\AgHQ)\subset\F_{k+l}(\AgHQ).
\end{gather*}

Similarly, let $\F_k(\AHQ)$ be the subspace of $\AHQ$ spanned by
the Jacobi diagrams of loop degree at least $k$. Again, we have a filtration 
$$
  \AHQ=
  \F_0(\AHQ)\supset
  \F_1(\AHQ)\supset
  \F_2(\AHQ)\supset\cdots,
$$
which induces a filtration on $\A_i$, for each $i\geq 1$, and is an algebra filtration:
\begin{gather*}
  \F_k(\AHQ)\;\star\;\F_l(\AHQ)\subset\F_{k+l}(\AHQ).
\end{gather*}
The above two filtrations are connected by the symmetrization isomorphism:
$$
\chi(\F_k(\AHQ))=\F_k(\AgHQ).
$$

The above algebra filtrations also induce Lie algebra filtrations on $\AgcHQ$ and $\AcHQ$. 
To be more specific, if we define
\begin{gather*}
  \F_k(\AgcHQ) := \F_k(\AgHQ) \cap \AgcHQ \quad \hbox{and} \quad
  \F_k(\AcHQ) := \F_k(\AHQ) \cap \AcHQ,
\end{gather*}
then we have $[\F_k,\F_l]\subset \F_{k+l}$ in the two cases.

These ``loop filtrations'' are closely related to clasper calculus.  
For example, one can prove that if a graph clasper $C$ in the trivial cylinder over $\Sigma_{g,1}$ has $k$ loops, 
then the LMO invariant of the homology cylinder obtained  by surgery along $C$ belongs to $\varnothing + \F_k(\AgHQ)$.

For each $i\ge1$, the internal degree $i$ part of $\A^c$ is itself graded as a vector space by the loop degree:
\begin{gather}
  \label{e9}
  \A^c_i=\bigoplus_{0\le k\le d_i}\Ac{i,k},
\end{gather}
where the bound is $d_i=(i+2)/2$ if $i$ is even, and $d_i=(i-1)/2$ if $i$ is odd\footnote{
This can be checked from the fact that a Jacobi diagram with only one external vertex vanishes in the space $\A$:
See for instance \cite{Vogel}.}.
Set
\begin{gather*}
  \A^{c}_{*,\ev} :=\bigoplus_{k\;\text{even},\;i} \Ac{i,k}
  \quad \quad \hbox{and}\quad \quad
  \A^{c}_{*,\od} :=\bigoplus_{ k\;\text{odd},\;i} \Ac{i,k}.
\end{gather*}
Then \eqref{eq:commutator} implies that $\A^{c}_{*,\ev}$ is a Lie
subalgebra of $\AcHQ$, and that the decomposition
\begin{gather*}
  \AcHQ=\A^{c}_{*,\ev} \oplus \A^{c}_{*,\od}
\end{gather*}
defines a $\Z/2\Z$-graded Lie algebra structure on $\AcHQ$. In particular, we have
\begin{gather}
  \label{eq:a}
  a\sHQ\subset \A^{c}_{*,\ev}\sHQ,
\end{gather}
where $a\sHQ$ is the Lie subalgebra of $\A^c\sHQ$ generated by the degree $1$ part $\Ac{1}$.

\subsection{Weight systems and the Moyal--Weyl product}

\label{subsec:MW}

Let us recall the definition of the Moyal--Weyl product.
For this, we consider a $\Q$-vector space $V$
together with a symplectic form $s:V \otimes V \to \Q$.

\begin{definition}
The \emph{Weyl algebra} generated by $V$ is the quotient of the tensor algebra 
of $V$ by the relations ``$\ u\otimes v - v \otimes u= s(u,v)\ $'':
$$
\mathcal{W}(V) := T(V) \left/ 
\langle u\otimes v - v \otimes u - s(u,v)\ |\ u,v \in V\rangle_{\rm{ideal}} \right. .
$$
\end{definition}

The ``symmetrization'' map $\chi: S(V) \longrightarrow \mathcal{W}(V)$ is defined by
$$
\chi(v_1 \cdot \dots \cdot v_n)
:=  \frac{1}{n!} \sum_{ \sigma \in S_n } 
\left\{ v_{\sigma(1)} \otimes \cdots \otimes v_{\sigma(n)} \right\}.
$$ 
By formally the same argument as in Proposition \ref{lem:chi}, 
it can be shown that $\chi$ is an isomorphism, which justifies the following.

\begin{definition}
The associative multiplication on the vector space $S(V)$ 
corresponding to $\otimes$ on $\mathcal{W}(V)$
is denoted by $\star$ and is called the \emph{Moyal--Weyl product}.
\end{definition}

\begin{remark}
Recall that a \emph{deformation quantization} of a Poisson algebra $(A,\cdot,\{-,-\})$
is a $\Q[[h]]$-linear associative multiplication $\star_h$ on the space $A[[h]]$, such that 
$a\star_h b = a \cdot b + O(h)$ and $a\star_h b - b\star_h a = \{a,b\} \cdot h + O(h^2)$,
for all $a,b \in A$. See \cite{BFFLS}.
 
We are considering here the commutative algebra $(S(V),\cdot)$ 
with Poisson bracket defined by $\{u,v\}:= s(u,v)$ for all $u,v\in V$.
The \emph{Moyal--Weyl product} usually refers to its deformation quantization 
$$
\star_h : S(V)[[h]] \otimes S(V)[[h]] \longrightarrow S(V)[[h]]
$$  
defined, for all $A,B \in S(V)$, by 
$$
A \star_h B := 
\sum_{l = 0}^\infty \frac{h^l}{2^l l!}
\sum_{\substack{i_1,\dots, i_l \in \{1,\dots,d\} \\ j_1,\dots, j_l \in \{1,\dots,d\} }}
\left(\prod_{k=1}^l s(x_{i_k},x_{j_k})\right)
 \frac{\partial^l A}{\partial x_{i_1} \cdots \partial x_{i_l}}
\frac{\partial^l B}{\partial x_{j_1} \cdots \partial x_{j_l}}.
$$

\noindent
Here, a basis $(x_1,\dots,x_d)$ of $V$ has been chosen so that $S(V)$ 
is identified with the polynomial algebra $\Q[x_1,\dots,x_d]$.
Using an analogue of Proposition  \ref{prop:multiplication}  for the product
$\star$ on $S(V)$, it is easily checked that $\star$ coincides with $\star_h$ at $h=1$.
\end{remark}

To connect the multiplication $\star$ of symplectic Jacobi diagrams to the Moyal--Weyl product,
we consider a \emph{metrized} Lie algebra $\mathfrak{g}$.
Thus, $\mathfrak{g}$ is a finite-dimensional Lie algebra 
together with a symmetric bilinear form
$\kappa: \mathfrak{g} \times \mathfrak{g} \to \Q$, which is $\mathfrak{g}$-invariant and non-degenerate.

It is well-known that such a data defines a linear map $\A(\varnothing) \to \Q[[t]]$,
called the \emph{weight system} associated to $\mathfrak{g}$: 
This is the case of homology spheres considered in \cite{LMO} or, equivalently, 
the case of homology cylinders of genus $g=0$. This construction extends to higher genus as follows.
First, we equip $\mathfrak{g} \otimes \HQ$  with the symplectic form $\kappa \otimes \omega$,
where $\omega$ is the intersection pairing on $\sgg$.

\begin{theorem}
\label{th:weight_systems}
We can define non-trivial algebra homomorphisms
$$
W_{\mathfrak{g}} : \ \left(\ \A^<\sHQ\ ,\ \osqcup\ \right) \longrightarrow 
\left(\ \mathcal{W}\left(\mathfrak{g} \otimes \HQ \right)[t]\ ,\ \otimes\ \right)
$$
and
$$
W_{\mathfrak{g}} : \ (\ \A\sHQ\ ,\ \star\ ) \longrightarrow 
\left(\ S(\mathfrak{g} \otimes \HQ)[t]\ ,\ \star\ \right),
$$
sending the internal degree to the $t$-degree and
such that the following diagram commutes in the category of graded algebras:
\begin{equation}
\label{eq:weight_systems}
\xymatrix{
{\A^<\sHQ} \ar[rr]^-{W_{\mathfrak{g}}} & &
{\mathcal{W}(\mathfrak{g} \otimes \HQ)[t]} \\
{\A\sHQ} \ar[u]^-{\chi}_-\simeq \ar[rr]_-{W_{\mathfrak{g}}} & &
{S(\mathfrak{g} \otimes \HQ)[t]}. \ar[u]_-{\chi}^-\simeq 
}
\end{equation}
\end{theorem}

\begin{proof}
Let $K \in S^2 \mathfrak{g}$ be the $2$-tensor corresponding 
to $\kappa \in  S^2 \mathfrak{g}^* \simeq (S^2 \mathfrak{g})^*$ 
by the isomorphism $\mathfrak{g} \to \mathfrak{g}^*$ adjoint to $\kappa$. 
Let also $B \in \Lambda^3 \mathfrak{g}^* \simeq \left(\Lambda^3 \mathfrak{g}\right)^*$ 
be the alternating trilinear form defined by $x \wedge y \wedge z \mapsto \kappa([x,y],z)$.

Then, any Jacobi diagram $D$ whose external vertices are numbered from $1$ to $e$ 
defines a tensor in $\mathfrak{g}^{\otimes e}$: Each internal vertex is replaced by a copy of $B$, 
each edge by a copy of $K$ and contractions are performed. 
If $D$ is now colored by $\HQ$, then we get a tensor $\widetilde{w}_{\mathfrak{g}}(D)$ in 
$\mathfrak{g}^{\otimes e} \otimes \HQ^{\otimes e} \simeq (\mathfrak{g} \otimes \HQ)^{\otimes e}$.
Let $\widetilde{\A}\sHQ$ be the space of Jacobi diagrams defined at (\ref{eq:big_space}) 
and define a linear map 
$$
\widetilde{W}_{\mathfrak{g}}: \widetilde{\A}\sHQ 
\longrightarrow T(\mathfrak{g} \otimes \HQ)[t]
$$
by $D\mapsto \widetilde{w}_{\mathfrak{g}}(D) \cdot t^{i}$ 
for all Jacobi diagram $D$ with $i$ internal vertices.
This map is well-defined since the antisymmetry and the Jacobi identity
satisfied by the Lie bracket of $\mathfrak{g}$ are mapped to the AS
and IHX relations, respectively, as usual.
Observe that the ordered disjoint union $\osqcup$ defines 
an associative multiplication on $\widetilde{\A}\sHQ$,
and that $\widetilde{W}_{\mathfrak{g}}$ is then multiplicative.

Because the contraction of $K \otimes \kappa \otimes K$ gives $K$, the map 
$\widetilde{W}_{\mathfrak{g}}$ induces a linear map 
$W_{\mathfrak{g}}: \A^<\sHQ \to \mathcal{W}\left(\mathfrak{g} \otimes \HQ\right)[t]$.
Obviously, the multiplicativity of the former implies the multiplicativity of the latter.

Clearly, the  map $\widetilde{W}_{\mathfrak{g}}$ 
induces a linear map $W_{\mathfrak{g}}: \A\sHQ \to \hbox{S}(\mathfrak{g} \otimes \HQ)[t]$ as well.
The commutativity of (\ref{eq:weight_systems}) follows from the definitions 
and shows that the multiplicativity of its top map implies that of its bottom map.
\end{proof}

\begin{remark}
Since the form $\kappa$ is $\mathfrak{g}$-invariant, the above two maps $W_{\mathfrak{g}}$
actually take values on the $\mathfrak{g}$-invariant subspaces.
\end{remark}

\vspace{0.5cm}

\section{Algebraic description of the mapping cylinder construction}

\label{sec:mapping_cylinder_construction}

In this section, we prove Theorem \ref{th:diagram}, which 
gives an algebraic description of the mapping cylinder
construction $\mapcyl: \Igg \to \Cgg$ at the level of graded Lie algebras.
To start with, we shall describe the action of $\Sp(H_\Q)$ 
on the four graded Lie algebras of diagram (\ref{eq:diagram}).

\subsection{Symplectic actions}

We start by recalling how $\Sp(H_\Q)$ acts on $\Gr^\Gamma\Igg\oQ$.
The conjugation action of the mapping class group
$\Mgg$ on $\Igg$ induces an $\Sp(H)$-action on
$\Gr^\Gamma \Igg$ such that the Lie bracket is $\Sp(H)$-equivariant.
There is also the standard action of $\Sp(H)$ on $\Lambda^3 H$, 
and the first Johnson homomorphism is $\Sp(H)$-equivariant:
Thus, for $g\geq 3$, the Lie algebra map
$$
J: \Lie(\Lambda^3 H_\Q) \longrightarrow \Gr^\Gamma\Igg\oQ
$$ 
is $\Sp(H)$-equivariant. 
It follows that $\Ker(J)$ is an $\Sp(H)$-submodule and so,
by the ``algebraicity lemma'' of \cite[\S 2.2.8]{AN}, it is an $\Sp(H_\Q)$-submodule as well. 
Consequently, the Lie algebra epimorphism $J$ transports 
the action of $\Sp(H_\Q)$ on $\Lie(\Lambda^3 H_\Q)$ 
onto an action of $\Sp(H_\Q)$ on $\Gr^\Gamma\Igg\oQ$,
and this extends the canonical action of $\Sp(H)$.

Let us now specify how $\Sp(H_\Q)$ acts on the Lie algebra of homology cylinders.

\begin{lemma}
\label{lem:action}
Let $g\geq 0$.  There is a natural action of $\Sp(H)$ on the
Lie algebra $\Gr^Y \Cgg$, which is compatible with the usual action of $\Sp(H)$ on  $\Gr^\Gamma \Igg$.
Moreover, the Lie bracket of $\Gr^Y \Cgg$ is $\Sp(H)$-equivariant.
\end{lemma}

\begin{proof} Let $\Cob(g,g)$ be the monoid of cobordisms from $\sgg$ to $\sgg$. 
The mapping cylinder construction defines
an inclusion $\Mgg \hookrightarrow \Cob(g,g)$. 
Thus, $\Mgg$ acts on $\Cob(g,g)$ by conjugation: 
$$
\Mgg \times \Cob(g,g) \ni \ (f,M) \longmapsto f \circ M \circ f^{-1} \ \in \Cob(g,g).
$$
The Mayer--Vietoris theorem shows that this action preserves the submonoid $\Cgg$ of $\Cob(g,g)$.
The $Y_i$-equivalence being generated by surgeries along graph claspers with $i$ nodes,
this action also preserves the submonoid $Y_i \Cgg$ of $\Cgg$. 
This follows from the general fact
\begin{equation}
\label{eq:pairs}
\begin{array}{l}
\forall \hbox{ graph clasper } G \subset \sgg \times [-1,1], \ \forall f\in \Mgg, \\[0.1cm]
\left(\ f\circ\left(  \sgg \times [-1,1]\right) \circ f^{-1}\ ,\ G\ \right) \cong 
\left(\ \sgg \times [-1,1]\ ,\ (f\times \Id_{[-1,1]})(G)\ \right)
\end{array}
\end{equation}
where  the symbol ``$\cong$'' means a homeomorphism of pairs (3-manifold with boundary, embedded graph).
Therefore, the group $\Mgg$ acts on $\Gr^Y \Cgg$.
But, inclusion (\ref{eq:N-series}) also shows that 
$$
\forall f \in \Igg, \ \forall M \in Y_i \Cgg,\ f \circ M \circ f^{-1} \sim_{Y_{i+1}} M. 
$$
So, the action of $\Mgg$ on $\Gr^Y \Cgg$ factorizes to $\Mgg/\Igg \simeq \Sp(H)$. 
The second statement of the lemma is easily checked.
\end{proof}

Next, we have the following result generalizing the $\Sp(H)$-equivariance of Johnson's homomorphism:

\begin{lemma}
\label{lem:equivariance_LMO}
Let $g\geq 0$.  The Lie algebra isomorphisms $\psi$ and  $\LMO$,
defined in \S~\ref{subsec:surgery} and \S~\ref{subsec:LMO} respectively,
are both $\Sp(H)$-equivariant.
\end{lemma}

\noindent
So, we can transport the action of $\SpHQ$ on $\A^{<,c}\sHQ$
to an action on $\Gr^Y \Cgg\oQ$,
which extends the natural action of $\Sp(H)$ given by Lemma \ref{lem:action}.

\begin{proof}[Proof of Lemma \ref{lem:equivariance_LMO}]
It is enough to show that $\psi$ is $\Sp(H)$-equivariant. 
Let $D \in \AgcHQ$ be a connected Jacobi diagram and let $F \in \Sp(H)$. 
Choose $f\in \Mgg$ which induces $F$ in homology. Then, we have
$$
\psi(F\cdot D) = \left\{ \left(\sgg \times [-1,1]\right)_{C(F\cdot D)} \right\}
= \left\{ \left( \sgg \times [-1,1] \right)_{\left(f\times \Id_{[-1,1]}\right)(C(D))}\right\}
$$
$$
\stackrel{(\ref{eq:pairs})}{=} 
\left\{ f \circ  \left( \sgg \times [-1,1] \right)_{C(D)} \circ f^{-1} \right\}
= F \cdot \left\{ \psi(D) \right\}.
$$
\end{proof}

\begin{remark}
\label{rem:basis}
The construction of the LMO functor in \cite{CHM} and, a fortiori, 
the definition of the LMO homomorphism (\ref{eq:LMO})
$$
\hbox{LMO} = s \circ \varphi \circ \ZtildeY : \Cgg \longrightarrow \AgHQ
$$ 
depends on a choice of meridian and parallel curves 
$(\alpha_1,\dots, \alpha_g,$ $\beta_1,\dots,\beta_g)=:(\alpha,\beta)$ shown in Figure \ref{fig:surface}.
Another choice $(\alpha_1',\dots, \alpha'_g,$ $\beta_1',\dots,\beta'_g)=:(\alpha',\beta')$ 
would lead to ``another'' invariant
$$
\hbox{LMO}' :  \Cgg \longrightarrow \AgHQ.
$$
\begin{figure}[h]
\centerline{\relabelbox \small
\epsfxsize 5truein \epsfbox{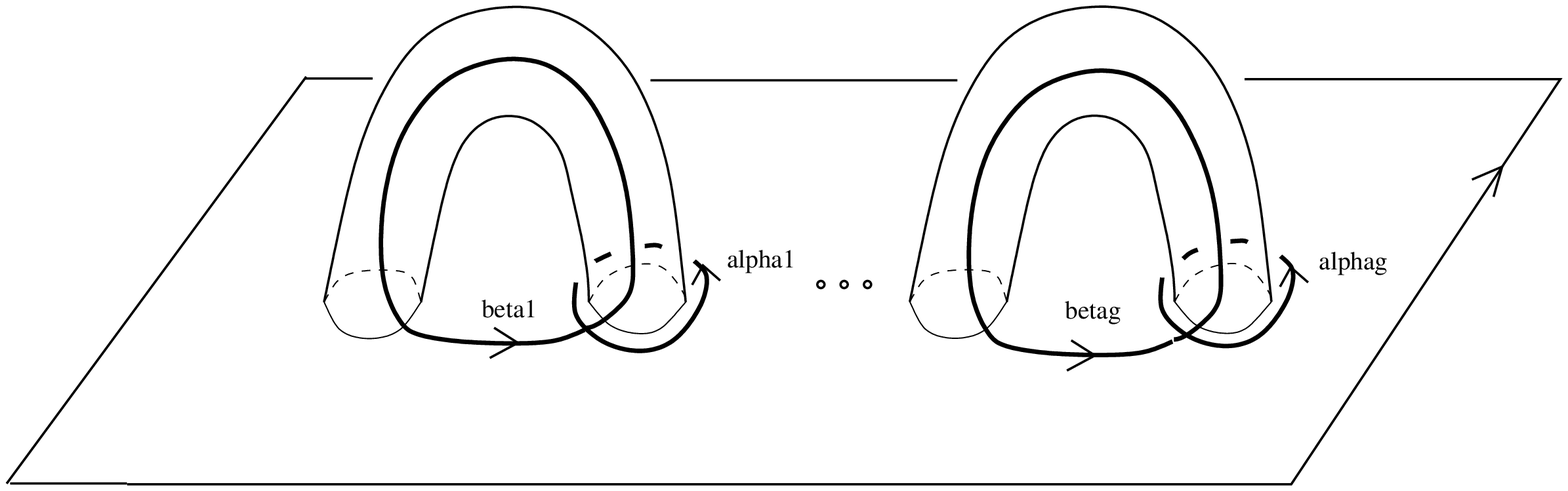}
\adjustrelabel <0cm,-0cm> {alpha1}{$\alpha_1$}
\adjustrelabel <0cm,-0cm> {alphag}{$\alpha_g$}
\adjustrelabel <0cm,-0cm> {beta1}{$\beta_1$}
\adjustrelabel <0cm,-0cm> {betag}{$\beta_g$}
\endrelabelbox}
\caption{The surface $\sgg$ and its system of meridians and parallels $(\alpha,\beta)$.}
\label{fig:surface}
\end{figure}
Let $f: \sgg \to \sgg$ be a homeomorphism sending the curves $\alpha,\beta$ 
to $\alpha',\beta'$ respectively.
Then, the connection between the latter invariant and the former one is as follows:
$$
\forall M \in \Cgg, \ \hbox{LMO}'(M) = f_* \cdot \hbox{LMO}\left( f^{-1} \circ M \circ f\right),
$$
where $f_* \in \Sp(H)$ denotes the action of $f$ on $H$. 
Therefore, Lemma \ref{lem:equivariance_LMO} says that the LMO homomorphism (\ref{eq:LMO})
does not depend on the choice of $(\alpha,\beta)$ \emph{at the graded level.}
\end{remark}

\subsection{The mapping cylinder construction at the level of graded Lie algebras}

We can now prove Theorem \ref{th:diagram}. 
Let $Y: \Lie(\LHQ) \to \A^{<,c}\sHQ$ 
be the Lie algebra homomorphism defined by identifying $\LHQ$ to $\A^{<,c}_1\sHQ$ as follows:
$$
a\wedge b \wedge c \mapsto \Ygraphtop{a}{b}{c},
$$
the total ordering of the external vertices being irrelevant in this case. 
Since the Lie bracket of $\AgcHQ$ is equivariant under the action of $\SpHQ$, 
the map $Y$ is equivariant as well.
So, we have the following diagram in the category of graded Lie
algebras with $\SpHQ$-actions: 
\begin{equation}
\label{eq:pre-diagram}
\xymatrix{
{\Gr^\Gamma \Igg \oQ} \ar[rr]^{\Gr \mapcyl \oQ} && {\Gr^Y \Cgg \oQ} \ar[d]^-{\rm{LMO}}_-\simeq\\
{\Lie\left(\LHQ\right) } 
\ar[u]^-J \ar[rr]^-{Y} && {\AgcHQ}.
}
\end{equation}
This commutes in degree $1$ since the $Y$-part of the LMO invariant $\Gr \ZtildeY$ 
defined in \cite{CHM} corresponds to the first Johnson homomorphism.
Since the Lie algebra $\Lie\left( \LHQ \right)$ is generated by its degree $1$ elements,
that diagram commutes in any degree. This completes the proof of Theorem \ref{th:diagram}.

\vspace{0.5cm}

\section{The degree two case}

\label{sec:quadratic}

In this section, we recall a few facts about classical representation theory of $\Sp_{2g} \C$.
We also recall the quadratic relations of the Torelli Lie algebra, as given explicitly in \cite{HS}.
Then, we compute the Lie bracket of $\AcHQ$ in degree $1+1$, which
allows us to prove Theorem \ref{th:degree_2}.

\subsection{Representation theory of $\Sp_{2g} \C$}

For basics of the representation theory of the Lie group  $\Sp_{2g} \C$, 
the reader is referred to \cite[\S\S~16-17]{FH}, the notations of which we will follow.
Thus, we denote
$$
\Sp_{2g} \mathbb{A} := \left\{ M \in \hbox{GL}_{2g}\mathbb{A}  : {}^t M \Omega M = \Omega \right\}, \
\hbox{ where } \Omega := \left(\begin{array}{cc} 0 & I_g \\ -I_g & 0 \end{array}\right)
\hbox{ and } \mathbb{A} := \Z, \Q, \C. 
$$
The representation theory of the complex Lie group $\Sp_{2g} \C$ is the same as that of its Lie algebra 
$$
\sp_{2g} \C  = \left\{  X \in \mathfrak{gl}_{2g} \C : {}^t X \Omega + \Omega X = 0\right\}.
$$
The diagonal matrices in $\sp_{2g} \C$ form a Cartan subalgebra $\mathfrak{h}$. Set
$H_i := E_{i,i} -E_{g+i,g+i}$, where $E_{i,j}$ denotes 
the elementary matrix with only one $1$ in position $(i,j)$.
Then, $(H_1,\dots,H_g)$ is a basis of $\mathfrak{h}$
whose dual basis of $\mathfrak{h}^*$ is denoted by $(L_1,\dots,L_g)$.

With respect to the above Cartan subalgebra $\mathfrak{h}$, 
the set of roots of $\sp_{2g} \C$ is  $R = \{ \pm L_i \pm L_j \}$, 
and here are the corresponding eigenvectors:
$$
\begin{array}{c|c}
\hbox{eigenvalue} &  \hbox{eigenvector} \\
\hline
L_i-L_j \ (i\neq j) & X_{i,j} := E_{i,j} - E_{g+j,g+i}\\
\hline
L_i + L_j \ (i\neq j) & Y_{i,j} := E_{i,g+j} + E_{j,g+i}\\
\hline
-L_i -L_j \ (i\neq j) & Z_{i,j} := E_{g+i,j} + E_{g+j,i}\\
\hline
2 L_i & U_i := E_{i,g+i}\\
\hline
-2 L_i&  V_i := E_{g+i,i}
\end{array}
$$
One can declare the positive roots to be $R^+ = \{L_i+L_j | i \leq j \} \cup \{L_i -L_j | i< j\}$,
so that the primitive positive roots are the $L_i - L_{i+1}$'s, for $i=1,\dots,g-1$, and $2L_g$.

The weight lattice of $\sp_{2g} \C$ is spanned by the $L_i$'s, and the fundamental weights are
$$
\omega_1 := L_1 ,\ \omega_2 := L_1 + L_2,\ \dots,\ \omega_g := L_1 +L_2 +\cdots + L_g.
$$
Then, to each $g$-tuple of non-negative integers $(a_1,\dots,a_g)$
corresponds a unique irreducible representation $\Gamma_{a_1 \omega_1 + \cdots + a_g \omega_g}$
of $\sp_{2g} \C$ with highest weight $a_1 \omega_1 + \cdots + a_g \omega_g$.
The data $(a_1,\dots,a_g)$ can be thought of as the Young diagram with $a_i$ columns of height $i$ or,
equivalently, as the partition $\lambda = (a_g + \cdots +a_2 + a_1, a_g + \cdots +a_2, \dots, a_g)$ 
of length $|\lambda| = \sum_{i=1}^g i \cdot a_i$. Thus, irreducible $\sp_{2g} \C$-modules or, equivalently, 
irreducible $\Sp_{2g} \C$-modules are indexed by partitions $\lambda$ with no more than $g$ parts.

Actually, the $\sp_{2g} \C$-module $\Gamma_{a_1 \omega_1 + \cdots + a_g \omega_g}$ can be 
realized as the ``symplectic'' Schur module  $\mathbb{S}_{\langle\lambda\rangle} \mathbb{C}^{2g}$
associated to $\lambda$,
i.e$.$ the intersection in $\left(\mathbb{C}^{2g}\right)^{\otimes |\lambda|}$
of the ordinary Schur module $\mathbb{S}_\lambda \mathbb{C}^{2g}$
with the kernels of all possible contractions 
$\left(\mathbb{C}^{2g}\right)^{\otimes |\lambda|} \to \left(\mathbb{C}^{2g}\right)^{\otimes (|\lambda|-2)}$
defined by the symplectic form $\Omega$.
It follows that each representation $\Gamma_{a_1 \omega_1 + \cdots + a_g \omega_g}$ exists with
rational coefficients, and defines an irreducible $\Sp_{2g} \Q$-module as well.

\begin{example}
\label{ex:fundamental}
For all $k=1,\dots, g$, the fundamental representation 
$\Gamma_{\omega_k}$ is the symplectic Schur module given 
by the Young diagram with only one column of height $k$. So, $\Gamma_{\omega_k}$ is the kernel
of the contraction map
$$
\Lambda^k \mathbb{C}^{2g} \longrightarrow \Lambda^{k-2} \mathbb{C}^{2g}, \
v_1 \wedge \cdots \wedge v_k \longmapsto \sum_{i<j} (-1)^{i+j} \Omega(v_i,v_j) \cdot
v_1 \wedge \cdots \widehat{v_i} \cdots \widehat{v_j} \cdots \wedge v_k.
$$
\end{example}

In the sequel, we will meet some $\Sp_{2g} \C$-modules that are restrictions of $\GL_{2g} \C$-modules
via the canonical inclusion $\Sp_{2g} \C \subset \GL_{2g} \C$. 
In particular, the ordinary Schur module $\mathbb{S}_\lambda \mathbb{C}^{2g}$ can be regarded as an $\Sp_{2g} \C$-module,
and a ``restriction formula'' by Littlewood gives its irreducible decomposition 
when the partition $\lambda$ has no more than $g$ parts:
\begin{equation}
\label{eq:Littlewood}
\mathbb{S}_\lambda \mathbb{C}^{2g} \simeq 
\bigoplus_{\mu} N_{\mu \lambda} \cdot \mathbb{S}_{\langle \mu \rangle} \mathbb{C}^{2g}.
\end{equation}
Here, the sum is over all partitions $\mu$ with no more than $g$ parts and
$$
N_{\mu \lambda} = \sum_{\eta} N_{\eta \mu \lambda}
$$
is the sum of the Littlewood--Richardson coefficients $N_{\eta \mu \lambda}$ over all
partitions $\eta$ with each part occurring an even number of times. 
See \cite[(4.4)]{Littlewood} or \cite[(25.39)]{FH} for details.

For example, the irreducible decomposition of $\Lambda^2 \Lambda^3 \C^{2g}$
as an $\Sp_{2g} \C$-module can be computed using this method.
We restrict to the case $g\geq 3$ since this will be enough for our purposes.

\begin{lemma}
\label{lem:l2l3}
We have the following isomorphism of $\Sp_{2g} \C$-modules:
\begin{equation}
\label{eq:decomposition}
\Lambda^2 \Lambda^3 \C^{2g} \simeq
\left\{\begin{array}{ll}
2 \Gamma_0 + 3 \Gamma_{\omega_2} + \Gamma_{2 \omega_2} + \Gamma_{\omega_1+ \omega_3} + 2 \Gamma_{\omega_4} + \Gamma_{\omega_2 + \omega_4} + \Gamma_{\omega_6} & \hbox{if } g\geq 6,\\
2 \Gamma_0 + 3 \Gamma_{\omega_2} + \Gamma_{2 \omega_2} + \Gamma_{\omega_1+ \omega_3} + 2 \Gamma_{\omega_4} + \Gamma_{\omega_2 + \omega_4} & \hbox{if } g=5,\\
2 \Gamma_0 + 3 \Gamma_{\omega_2} + \Gamma_{2 \omega_2} + \Gamma_{\omega_1+ \omega_3} +\phantom{2}   \Gamma_{\omega_4} + \Gamma_{\omega_2 + \omega_4} & \hbox{if } g=4,\\
2 \Gamma_0 + 2 \Gamma_{\omega_2} + \Gamma_{2 \omega_2} + \Gamma_{\omega_1+ \omega_3} & \hbox{if } g=3.\\
\end{array} \right.
\end{equation}
\end{lemma} 

\begin{proof}
The irreducible decomposition of $\Lambda^2 \Lambda^3 \C^{2g}$
as a $\GL_{2g} \C$-module can be deduced from Pieri's formula for all $g\geq 0$:
$$
\Lambda^2 \Lambda^3 \C^{2g} \simeq \mathbb{S}_{(1,1,1,1,1,1)} \C^{2g} 
\oplus \mathbb{S}_{(2,2,1,1)} \C^{2g},
$$
see \cite[Exercice 15.32]{FH}. 
We deduce from Example \ref{ex:fundamental} that
$$
\mathbb{S}_{(1,1,1,1,1,1)} \C^{2g} = \Lambda^6 \C^{2g} \simeq  \left\{
\begin{array}{lcll}
&& \Gamma_0 + \Gamma_{\omega_2} + \Gamma_{\omega_4} + \Gamma_{\omega_6} & \hbox{if } g\geq 6,\\
\Lambda^4 \C^{2g} &\simeq & \Gamma_0 + \Gamma_{\omega_2} + \Gamma_{\omega_4} & \hbox{if } g=5,\\
\Lambda^2 \C^{2g}  &\simeq  & \Gamma_0 + \Gamma_{\omega_2} & \hbox{if } g=4,\\
\C  &\simeq & \Gamma_0 & \hbox{if } g=3.\\
\end{array} \right.
$$
Moreover, Littlewood's restriction formula (\ref{eq:Littlewood}) shows\footnote{When $g=3$, 
the partition $(2,2,1,1)$ has too many parts to apply directly (\ref{eq:Littlewood}).
Nevertheless, there is a trick to bypass this restriction: See \cite[\S 6 (ii)]{Littlewood}.} that
$$
\mathbb{S}_{(2,2,1,1)} \C^{2g} \simeq 
\left\{ \begin{array}{ll}
\Gamma_0 + 2 \Gamma_{\omega_2} + \Gamma_{2 \omega_2}
+ \Gamma_{\omega_1 + \omega_3} + \Gamma_{\omega_4} + \Gamma_{\omega_2 + \omega_4} 
& \hbox{if } g\geq 4,\\
\Gamma_0 + 2 \Gamma_{\omega_2} + \Gamma_{2 \omega_2}+ \Gamma_{\omega_1 + \omega_3} & \hbox{if } g=3.
\end{array} \right.
$$
The conclusion follows.
\end{proof}

\subsection{Quadratic relations of the Torelli Lie algebra}

\label{subsec:quadratic_relations}

Let $\left(\alpha_1, \dots , \alpha_g, \beta_1 , \dots,
\beta_g\right)$ be a system of meridians and parallels on the surface $\sgg$, as shown in Figure \ref{fig:surface}. 
This defines a symplectic basis of $\HQ$, so that $\SpHQ$ is identified with $\Sp_{2g} \Q$.
 
By abuse of notation, let $\omega \in \Lambda^2 \HQ$ denote the bivector dual to the symplectic form $\omega$, namely
$$
\omega := \sum_{i=1}^g \alpha_i \wedge \beta_i.
$$
Define $r_1,r_2\in\Lie_2(\LHQ)$ by
$$
\begin{array}{l}
r_1 :=\begin{cases}
\left[ \alpha_1 \wedge \alpha_2 \wedge \beta_2 , \alpha_3 \wedge
  \alpha_4 \wedge \beta_4  \right] 
&\text{if $g\ge4$},\\
0&\text{if $g=3$},
\end{cases}\\
r_2 := \left[ \alpha_1 \wedge \alpha_2 \wedge \beta_2 , \alpha_g \wedge \omega \right]
\quad \quad \quad \quad \;\;\text{if $g\ge3$}.
\end{array}
$$

The following theorem is proved in \cite{HS} by completing Hain's
arguments \cite[\S 11]{Hain}:

\begin{proposition}[Hain, Habegger--Sorger] 
\label{prop:HHS}
If  $g\geq 6$, then the $\Sp\sHQ$-module of quadratic relations 
${\rm{R}}_2\left(\Igg\right)$ is spanned by $r_1$ and $r_2$.
\end{proposition}

\noindent
Actually, Proposition \ref{prop:HHS} and its proof extend to all $g\ge3$: See \eqref{e10} below.
This proposition together with Hain's result (Theorem \ref{th:Hain}) 
provides a quadratic presentation of the Torelli Lie algebra in genus $g\geq 6$.

\subsection{The Lie bracket $b_2$}

In order to prove Theorem \ref{th:degree_2}, we need to compute the Lie bracket of $\A^{c}(H_\Q)$ in degree $1+1$:
\begin{gather*}
  b_2:=[-,-]_\star : \Lambda^2 \Ac{1}  \longrightarrow \Ac{2}.
\end{gather*}
The following formula for $b_2$ is deduced from (\ref{eq:commutator}).

\begin{lemma}
\label{lem:bracket}
For all $x_1,x_2,x_3 \in \HQ$ and  $y_1,y_2,y_3 \in \HQ$, we have
$$
\left[\Ygraphbotbottop{x_2}{x_3}{x_1},\Ygraphbotbottop{y_2}{y_3}{y_1}\right]_\star
= \sum_{\substack{i \in \Z/3\Z\\j\in \Z/3\Z}} \omega(x_i, y_j) 
{\begin{array}{c} \\[-0.2cm] \!\! {\relabelbox \small
\epsfxsize 0.4truein \epsfbox{H.eps}
\adjustrelabel <0cm,0.3cm> {a}{\scriptsize $x_{i+2}$}
\adjustrelabel <0cm,0.35cm> {b}{\scriptsize $y_{j+1}$}
\adjustrelabel <0cm,-0cm> {c}{\scriptsize $x_{i+1}$}
\adjustrelabel <0cm,-0cm> {d}{\scriptsize $y_{j+2}$}
\endrelabelbox} \end{array}}
- \frac{1}{4} \left| {\scriptsize \begin{array}{ccc} \omega(x_1, y_1) & \omega(x_1, y_2) & \omega(x_1 , y_3)\\
\omega(x_2 , y_1) & \omega(x_2, y_2) & \omega(x_2 , y_3)\\
\omega(x_3, y_1) & \omega(x_3 , y_2) & \omega(x_3 , y_3)
\end{array} }\right|\
{\hspace{-0.2cm} \cdot \figtotext{20}{20}{theta.eps} \hspace{-0.2cm}}.
$$
\end{lemma}

The source of the map $b_2$ is decomposed into irreducible $\Sp\sHQ$-modules according to Lemma \ref{lem:l2l3}.
As for the target of the map $b_2$, \eqref{e9} gives the following decomposition:
\begin{gather*}
  \Ac{2} = \Ac{2,0}\oplus \Ac{2,1} \oplus \Ac{2,2},
\end{gather*}
where 
\begin{gather*}
  \begin{split}
    \Ac{2,0} &= \left\langle \Hgraph yzxw\,\Bigg|\,x,y,z,w\in\HQ \right\rangle_\Q,\\
    \Ac{2,1} &= \left\langle\left. \phigraphtop xy\,\right|\,x,y\in\HQ \right\rangle_\Q,\\
    \Ac{2,2} &= \left\langle \thetagraph \right\rangle_\Q.
  \end{split}
\end{gather*}
For any $g\ge0$, we have
\begin{gather*}
    \Ac{2,0} \simeq \mathbb{S}_{(2,2)}\HQ,\quad 
    \Ac{2,1} \simeq S^2\HQ,\quad 
    \Ac{2,2} \simeq \Q.
\end{gather*}
So, for $g\geq 3$,
we have the following irreducible decompositions into $\SpHQ$-modules:
\begin{gather*}
  \Ac{2,0} \simeq \Gamma_{0}+\Gamma_{\omega_2}+\Gamma_{2\omega_2}, \quad
  \Ac{2,1} \simeq \Gamma_{2\omega_1}, \quad
  \Ac{2,2} \simeq \Gamma_0,
\end{gather*}
where the decomposition of $\Ac{2,0}$ is deduced from Littlewood's formula \eqref{eq:Littlewood}.

\subsection{The image of $b_2$}

$\Img(b_2)$ has the following description.

\begin{proposition}
  \label{prop:im_bracket}
  If $g\ge3$, then we have
  \begin{gather*}
    \Img(b_2) =       \A_{2,\ev}^{c} = \Ac{2,0} \oplus \Ac{2,2} \quad \quad
      \left(\simeq 2\Gamma_0+\Gamma_{\omega_2}+\Gamma_{2\omega_2}\right).
  \end{gather*}
\end{proposition}

\noindent
This corresponds to Morita's result \cite{Morita_Casson_I,Morita_Casson_II} that
$\left(\Gamma_2 \Igg / \Gamma_3 \Igg\right) \oQ$ is classified
by the second Johnson homomorphism (onto $\Ac{2,0} $) and by the
Casson invariant (onto $\Ac{2,2} $).

\begin{proof}[Proof of Proposition \ref{prop:im_bracket}]
By \eqref{eq:a}, we have $\Img(b_2)\subset\A_{2,\ev}^{c}$.
To prove the converse inclusion, we use decompositions into irreducible $\SpHQ$-modules.
As mentioned in \cite[\S 3]{Sakasai}, some highest weight vectors of $\A_{2,0}^{c}$ are given by
$$
\begin{array}{c|c}
\hbox{eigenvalue} & \hbox{eigenvector} \\
\hline 
2 \omega_2 & \Hgraph{\alpha_1}{\alpha_2}{\alpha_2}{\alpha_1}\\
\hline
\omega_2 & \sum_{i=1}^g \Hgraph{\alpha_2}{\alpha_i}{\alpha_1}{\beta_i}
 =: \Ygraphbotbottop{\alpha_1}{\alpha_2}{\omega} \\
\hline
0 & \sum_{i,j=1}^{g} \Hgraph{\beta_i}{\alpha_j}{\alpha_i}{\beta_j} 
=: \strutgraphbot{\omega}{\omega}
\end{array}
$$
and this is easily checked.  Since
$$
\left[\Ygraphbotbottop{\alpha_1}{\alpha_2}{\beta_3},
\Ygraphbotbottop{\alpha_2}{\alpha_1}{\alpha_3} \right]_\star 
=  \Hgraph{\alpha_1}{\alpha_2}{\alpha_2}{\alpha_1}
\quad
\hbox{ and }
\quad
\sum_{i=2}^g \left[ \Ygraphbotbottop{\alpha_1}{\alpha_2}{\beta_1} , 
\Ygraphbotbottop{\beta_i}{\alpha_i}{\alpha_1} \right]_\star 
=  \Ygraphbotbottop{\alpha_1}{\alpha_2}{\omega},
$$
the summand $\Gamma_{2 \omega_2} + \Gamma_{\omega_2}$ of $\A_{2,0}^{c}$ 
is in the image of $[-,-]_\star$. 
In order to prove that the summand $2 \Gamma_0$ of $\A_{2,\ev}^{c}$
is in the image as well, we need to compute the Lie bracket of the $\SpHQ$-invariants of $\Lambda^2 \Ac{1}$.
Those $\SpHQ$-invariants of $\Lambda^2 \LHQ$ can be obtained 
by Morita's method \cite[\S 4.2]{Morita_survey}: We find
$$
T_1 := \sum_{i,j,k=1}^g \Ygraphbotbottop{\alpha_j}{\beta_j}{\alpha_i} \wedge
\Ygraphbotbottop{\alpha_k}{\beta_k}{\beta_i}
$$
corresponding to the trivalent graph with only one edge and two looped edges, and
$$
T_2 := \sum_{i,j,k=1}^g 
\Ygraphbotbottop{\alpha_j}{\alpha_k}{\alpha_i} \wedge \Ygraphbotbottop{\beta_j}{\beta_k}{\beta_i}
- \Ygraphbotbottop{\alpha_j}{\beta_k}{\alpha_i} \wedge \Ygraphbotbottop{\beta_j}{\alpha_k}{\beta_i}
- \Ygraphbotbottop{\alpha_k}{\beta_i}{\alpha_j} \wedge \Ygraphbotbottop{\beta_k}{\alpha_i}{\beta_j}
- \Ygraphbotbottop{\alpha_i}{\beta_j}{\alpha_k} \wedge \Ygraphbotbottop{\beta_i}{\alpha_j}{\beta_k}
$$
corresponding to the theta-shaped graph. Then, computations by hand lead to
$$
\begin{array}{l}
[T_1]_\star = - \frac{g(g-1)}{4}\ \thetagraph + (g-1)\ \strutgraphbot{\omega}{\omega}, \\

[T_2]_\star = -\frac{g (g-1) (2g-1)}{2}\ \thetagraph + 6 (g-1)\ \strutgraphbot{\omega}{\omega}.
\end{array}
$$
Those two vectors of $\A_{2,\ev}^{c}$ are not colinear since
$$
\left| \begin{array}{cc}
\frac{g(g-1)}{4} &  \frac{g (g-1) (2g-1)}{2} \\
(1-g) & 6 (1-g)
\end{array} \right| = g(g-1)^2 (g-2)\neq 0.
$$
We deduce that both $\thetagraph$ and $\strutgraphbot{\omega}{\omega}$ 
are in the image of $[-,-]_\star$, thus proving the lemma. 
\end{proof}

\subsection{The kernel of $b_2$}

$\Ker(b_2)$ has the following description.

\begin{lemma}
\label{lem:ker_bracket}
If $g\ge3$, then we have 
$$
\Ker(b_2) = \langle r_1, r_2 \rangle_{\SpHQ},
$$
where $r_1,r_2 \in \Lambda^2 \Ac{1} \simeq \Lie_2\left(\Lambda^3 H_\Q\right)$
are defined in \S~\ref{subsec:quadratic_relations}.
\end{lemma}

\begin{proof}
Lemma \ref{lem:bracket} gives $\left[ r_1\right]_\star =0$  and
$\left[ r_2\right]_\star = 0$ (by the IHX relation).
Since the bracket $[-,-]_\star$ is $\SpHQ$-equivariant, we have 
\begin{equation}
\label{eq:inclusion_one}
\left\langle  r_1,  r_2 \right\rangle_{\SpHQ} \subset \Ker\left(b_2\right).
\end{equation}

Moreover, we have the following inclusions
\begin{equation}
\label{eq:inclusions}
\left\langle  r_1,  r_2 \right\rangle_{\SpHQ} \supset
\left\{\begin{array}{ll}
2 \Gamma_{\omega_2} + \Gamma_{\omega_1+\omega_3} + 2 \Gamma_{\omega_4} + \Gamma_{\omega_2+\omega_4} + \Gamma_{\omega_6}
& \hbox{if } g\geq 6,\\
2 \Gamma_{\omega_2} + \Gamma_{\omega_1+\omega_3} + 2 \Gamma_{\omega_4} + \Gamma_{\omega_2+\omega_4}
& \hbox{if } g=5,\\
2 \Gamma_{\omega_2} + \Gamma_{\omega_1+\omega_3} +\phantom{2} \Gamma_{\omega_4} + \Gamma_{\omega_2+\omega_4} 
& \hbox{if } g=4,\\
\phantom{2} \Gamma_{\omega_2} + \Gamma_{\omega_1+ \omega_3} & \hbox{if } g=3,
\end{array} \right.
\end{equation}
which is shown in \cite[\S 2]{HS} for $g\geq 6$:
It is easy to adapt the arguments used there to the cases $g=3,4,5$.
Then, it follows from (\ref{eq:decomposition}) and (\ref{eq:inclusions}) that
$\left\langle  r_1,  r_2 \right\rangle_{\SpHQ}$ contains a submodule
whose complement in  $\Lambda^2 \Ac{1}$ is isomorphic to $2 \Gamma_0 + \Gamma_{\omega_2} + \Gamma_{2 \omega_2}$.
By Proposition \ref{prop:im_bracket}, this submodule is isomorphic to $\Ker\left(b_2\right)$.
Thus, we conclude thanks to (\ref{eq:inclusion_one}).
\end{proof}

We can now complete the proof of Theorem \ref{th:degree_2}.
Johnson's formula computes the first Johnson homomorphism $\tau_1$ on
``Bounding Pair'' (BP) maps \cite[Corollary p.233]{Johnson}.
For $i=1,2$, one easily finds a pair of BP maps $f_i$ and $g_i$
having disjoint supports and satisfying $\tau_1(f_i) \wedge \tau_1(g_i) =
r_i$.
Consequently, the commutative diagram
(\ref{eq:pre-diagram}) gives in degree $2$
$$
\xymatrix{
{\Gr^\Gamma_2 \Igg \oQ} \ar[rr]^-{\Gr_2 \mapcyl \oQ} && 
{\Gr^Y_2 \Cgg \oQ} \ar[d]^-{\LMO_2}_-\simeq\\
{\Lie_2\left(\LHQ\right)/ \langle r_1,r_2 \rangle_{\SpHQ}}
\ar[u]^-{\overline{J}_2} \ar[rr]^-{\overline{Y}_2} && {\A_2^{<,c}}.
}
$$
The kernel of $Y_2 = \chi_2 \circ b_2$ is $\langle r_1,r_2 \rangle_{\SpHQ}$ (by Lemma \ref{lem:ker_bracket}),
so that the map $\overline{J}_2$ in the prevous diagram is injective. Thus, we obtain
\begin{gather}
  \label{e10}
\hbox{R}_2(\Igg) = \langle r_1,r_2 \rangle_{\SpHQ}\quad
\text{for $g\ge3$},
\end{gather}
which slightly generalizes Proposition \ref{prop:HHS}.  
We conclude that $\Ker(Y_2) = {\rm{R}}_2(\Igg)$.

\vspace{0.5cm}

\section{Stability with respect to the genus}

\label{sec:stability}

In this short section, we consider stability with respect to the genus 
for the filtrations $\Gamma$ and $Y$ on $\Igg$ and $\Cgg$, respectively.  
We start by fixing surface inclusions
$$
\Sigma_{0,1}\subset \Sigma_{1,1}\subset \Sigma_{2,1} \subset  \cdots.
$$

Thus, we have group monomorphisms  
$$
\I_{0,1} \hookrightarrow \I_{1,1} \hookrightarrow \I_{2,1} \hookrightarrow \cdots,
$$
which allows us to regard $\Igg$ as a subgroup of $\glim\Igg$.
Then, the \emph{stabilized} lower central series of $\Igg$ is defined by
$$ 
  \Gamma^{\mathrm{stab}}_i\Igg:=
  \Igg\cap
  \Gamma_i\glim\Igg
  =  \Igg\cap
  \glim\Gamma_i\Igg.
$$

Similarly, we can regard $\Cgg$ as a submonoid of $\glim\Cgg$.
The \emph{stabilized} $Y$-filtration of $\Cgg$ is defined by
$$ 
  Y_i^{\mathrm{stab}} \Cgg := \Cgg\cap \glim Y_i \Cgg.
$$
The following proposition means that the $Y$-filtration for homology
cylinders is {\em stable}.

\begin{proposition}
\label{prop:stability_Y}
For all $g\ge0$ and for all $i\ge1$, we have
\begin{gather*}
  Y_i \Cgg = Y_i^{\mathrm{stab}} \Cgg.
\end{gather*}
\end{proposition}

Proposition \ref{prop:stability_Y} is proved at the end of this section,
and we now use it to state Conjecture \ref{r3} in a different way:

\begin{conjecture}
  \label{r15}
  For any $g\ge0$, the stabilized lower central series of the Torelli
  group $\Igg$ coincides with the $Y$-filtration:
  \begin{gather}
    \label{e12}
    \Gamma^{\mathrm{stab}}_j\Igg = \Igg \cap
    Y_j \Cgg \quad \text{for all $j\ge1$}.
  \end{gather}
\end{conjecture}

\noindent
The inclusion ``$\subset$'' in \eqref{e12} holds true, 
and it can be deduced from Proposition \ref{prop:stability_Y} as follows.  
If $x\in\Gamma^{\mathrm{stab}}_j\Igg$, then
$x\in\Igg$ and $x\in\Gamma_j \I_{g',1}$ for some $g'\ge g$.
The latter implies that $x\in Y_j \Cyl_{g',1}$ and, since $x\in\Cgg$, 
we deduce that $x\in Y_j^{\mathrm{stab}} \Cgg=Y_j \Cgg$.

\begin{proof}[Proof that Conjectures \ref{r3} and \ref{r15} are
    equivalent]
  Conjecture \ref{r3} is equivalent to the injectivity of
  \begin{gather}
    \label{e14}
    \glim\Gr_i \mapcyl \colon
    \glim\Gamma_i\Igg/\Gamma_{i+1}\Igg
    \longrightarrow
    \glim Y_i \Cgg /Y_{i+1}
  \end{gather}
  for all $i\geq1$.
  By Proposition \ref{prop:stability_Y}, injectivity of \eqref{e14} is equivalent
  to the statement that, if $x\in\Gamma_i\Igg$ satisfies $x\in Y_{i+1} \Cgg$
  then $x\in\Gamma^{\mathrm{stab}}_{i+1}\Igg$. 
  So,  Conjecture \ref{r3} is equivalent to the statement that
  \begin{gather}
    \label{e15}
    \Gamma_i\Igg\cap Y_{i+1}\Cgg \subset \Gamma^{\mathrm{stab}}_{i+1}\Igg
    \quad \text{for all $i\ge1$}.
  \end{gather}

  Clearly, the inclusion ``$\supset$'' in \eqref{e12}   implies \eqref{e15}. 
  Conversely, using Proposition \ref{prop:stability_Y} again,
  it is easily shown by induction on $j\geq 1$ that \eqref{e15} implies the
  inclusion ``$\supset$'' in \eqref{e12}.  
  Thus Conjecture \ref{r3} is equivalent to Conjecture \ref{r15}.
\end{proof}

\begin{proof}[Proof of Proposition \ref{prop:stability_Y}]
  The inclusion ``$\subset$'' is obvious.  To prove ``$\supset$'',
  suppose that $x\in Y_i^{\mathrm{stab}}\Cgg$ is
  represented by a homology cylinder $M$ over $\sgg$.  Then we have
  $x\in\Cgg$ and $x\in Y_i \Cyl_{g',1}$ for some $g' > g$.  
  In other words, the homology cylinder
  \begin{gather*}
    M' = M \cup_{\partial\sgg\times[-1,1]} 
    ((\Sigma_{g',1}\setminus\operatorname{int}\sgg)\times[-1,1])
  \end{gather*}
  over $\Sigma_{g',1}$, which represents
  $x\in Y_i \Cyl_{g',1}$, is $Y_i$-equivalent to the trivial cylinder
  $M_0 :=\Sigma_{g',1}\times[-1,1]$.  Hence there are mutually disjoint,
  connected graph claspers $C_1,\dots,C_p$ ($p\ge0$) in $M_0$, each having $i$ nodes, 
  such that the result $(M_0)_{C_1,\dots,C_p}$ from $M_0$ of surgery along
  $C_1,\dots,C_p$ is homeomorphic to $M'$ relative to boundary.

  Let $\tsgg$ denote a surface obtained from $\sgg$ by attaching
  a collar $N:=S^1\times[0,1]$ along $\partial\sgg$.  Thus $\tsgg$
  is a compact, oriented surface of genus $g$ with one boundary
  component, which contain $\sgg$ in its interior.  
  There is a (not proper) embedding
  \begin{gather*}
    f\colon \Sigma_{g',1}\times[-1,1] \hookrightarrow \tsgg\times[-1,1]
  \end{gather*}
  such that
  \begin{itemize}
  \item $f$ is the identity on $\sgg\times[-1,1]$,
  \item $f$ maps $\partial\Sigma_{g',1}\times[-1,1]$ homeomorphically
  onto $\partial\tsgg\times[-1,1]$.
  \end{itemize}
  Let $M'':=(\tsgg\times[-1,1])_{f(C_1),\dots,f(C_p)}$ be the homology
  cylinder over $\tsgg$ obtained from the cylinder $\tsgg\times[-1,1]$
  by surgery along $f(C_1),\dots,f(C_p)$.  If we regard $M''$ as a
  homology cylinder over $\sgg$ by a homeomorphism $\tsgg\cong\sgg$
  which is identity outside a small neighborhood of the collar
  $N\subset\tsgg$, then $M''$ is homeomorphic to $M$ relative to boundary.
  Hence $M$ is $Y_i$-equivalent to the trivial cylinder, i.e$.$ $x\in Y_i \Cgg$.
\end{proof}

\vspace{0.5cm}

\section{The closed surface case}

\label{sec:closed_case}

In this section, we extend our results to the case of a closed connected oriented surface of genus $g$,
which we denote by $\Sigma_g$. We set $H:=H_1(\Sigma_g)$ and $\omega: H \otimes H \to \Z$ denotes 
the intersection pairing of $\Sigma_g$. Let $\Ig$ be the Torelli group of the surface $\Sigma_g$,
and let $\Cg$ be the monoid of homology cylinders over $\Sigma_g$.

\subsection{The ideals $I^<\subset \Ag{}$ and $I^{<,c}\subset \Agc{}$}

\label{subsec:ideals}

In this subsection, we introduce an ideal $I^<$ in the algebra $\Ag{}$ and an
ideal $I^{<,c}$ in the Lie algebra $\Agc{}$.  The latter appears in \cite{Habiro}.

By an {\em $\omega $-diagram}, we mean a Jacobi diagram such that each
external vertex is labeled by either an element of $\HQ$ or the
symbol ``$\omega $'', and such that each component has at least one internal vertex
or has at least one vertex labeled $\omega $.  An external
vertex labeled $\omega $ is called an {\em $\omega $-vertex}.  
The {\em degree} of an $\omega $-diagram $D$ is defined to be the sum of the internal degree
of $D$ and the number of $\omega $-vertices in $D$.

By an {\em \og-diagram}, we mean an $\omega $-diagram whose external vertices are totally ordered.  
We associate to each \og-diagram an element in $\AgHQ$ as follows:
\begin{equation}
\label{eq:split_omega}
\begin{array}{c}
{\relabelbox \small
\epsfxsize 3truein \epsfbox{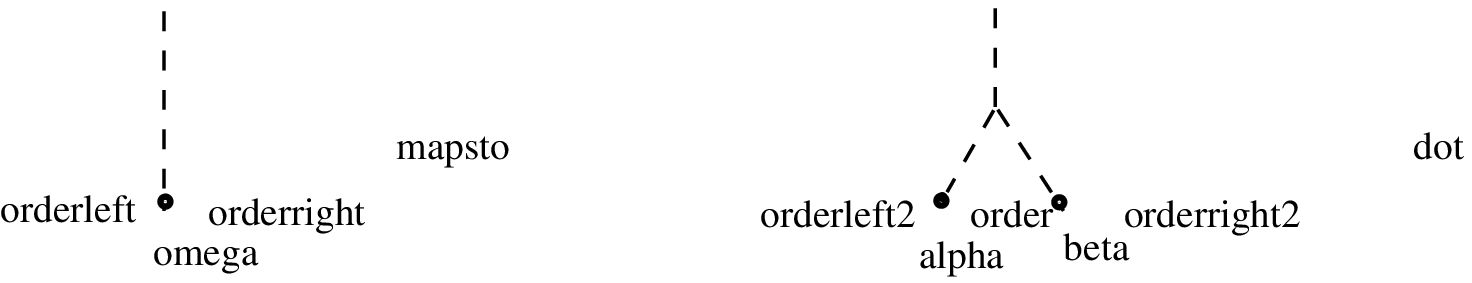}
\adjustrelabel <-0.1cm,-0cm> {orderleft}{$\cdots<$}
\adjustrelabel <0cm,-0cm> {orderright}{$<\cdots$}
\adjustrelabel <0cm,-0cm> {orderleft2}{$\cdots<$}
\adjustrelabel <0cm,-0cm> {orderright2}{$<\cdots$}
\adjustrelabel <0cm,-0cm> {order}{$<$}
\adjustrelabel <0cm,-0cm> {mapsto}{${\displaystyle \longmapsto \quad \sum_{i=1}^g}$}
\adjustrelabel <0cm,-0cm> {omega}{$\omega$}
\adjustrelabel <0cm,-0.0cm> {alpha}{$\alpha_i$}
\adjustrelabel <0cm,-0.05cm> {beta}{$\beta_i$}
\adjustrelabel <0cm,-0cm> {dot}{.}
\endrelabelbox}
\end{array}
\end{equation}

\noindent
Using this rule, we regard \og-diagrams as elements of $\Ag{}$.
Some basic properties for \og-diagrams are in order.

\begin{lemma}[STU-relation for $\omega $-vertex]
\label{lem:STU-omega}
In the space $\Ag{}$, we have the following identities:

\centerline{\relabelbox \small
\epsfxsize 4.5truein \epsfbox{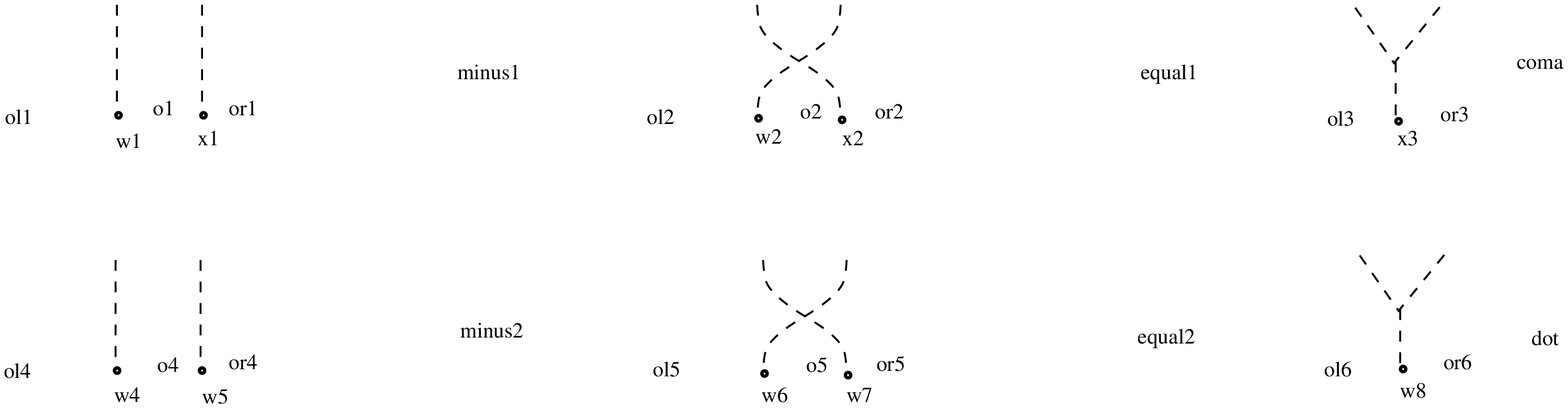}
\adjustrelabel <-0.1cm,-0.05cm> {ol1}{$\cdots<$}
\adjustrelabel <0cm,-0.1cm> {or1}{$<\cdots$}
\adjustrelabel <0cm,-0.1cm> {o1}{$<$}
\adjustrelabel <-0.1cm,-0.05cm> {ol2}{$\cdots<$}
\adjustrelabel <0cm,-0.1cm> {or2}{$<\cdots$}
\adjustrelabel <-0.1cm,-0.1cm> {o2}{$<$}
\adjustrelabel <-0.1cm,-0.05cm> {ol4}{$\cdots<$}
\adjustrelabel <0cm,-0.1cm> {or4}{$<\cdots$}
\adjustrelabel <-0.05cm,-0.1cm> {o4}{$<$}
\adjustrelabel <-0.1cm,-0.05cm> {ol5}{$\cdots<$}
\adjustrelabel <0cm,-0.1cm> {or5}{$<\cdots$}
\adjustrelabel <-0.1cm,-0.1cm> {o5}{$<$}
\adjustrelabel <-0.35cm,-0.05cm> {ol3}{$\cdots<$}
\adjustrelabel <-0.1cm,-0.05cm> {or3}{$<\cdots$}
\adjustrelabel <-0.3cm,-0cm> {ol6}{$\cdots<$}
\adjustrelabel <-0.15cm,-0.05cm> {or6}{$<\cdots$}
\adjustrelabel <-0cm,-0cm> {minus1}{$-$}
\adjustrelabel <0cm,-0cm> {minus2}{$-$}
\adjustrelabel <-0.4cm,-0cm> {equal1}{$= \quad -$}
\adjustrelabel <-0.4cm,-0cm> {equal2}{$= \quad -$}
\adjustrelabel <0cm,-0.1cm> {w1}{$\omega$}
\adjustrelabel <0cm,-0.1cm> {x1}{$x$}
\adjustrelabel <0cm,-0.1cm> {w2}{$x$}
\adjustrelabel <0cm,-0.1cm> {x2}{$\omega$}
\adjustrelabel <0cm,-0.1cm> {x3}{$x$}
\adjustrelabel <0cm,-0.1cm> {w4}{$\omega$}
\adjustrelabel <0cm,-0.1cm> {w5}{$\omega$}
\adjustrelabel <0cm,-0.1cm> {w6}{$\omega$}
\adjustrelabel <0cm,-0.1cm> {w7}{$\omega$}
\adjustrelabel <0cm,-0.1cm> {w8}{$\omega$}
\adjustrelabel <0.4cm,-0cm> {coma}{,}
\adjustrelabel <0.4cm,-0cm> {dot}{.}
\endrelabelbox}
\end{lemma}

\begin{proof}
The first identity follows from $2g$ applications of the STU-like relation in the space $\Ag{}$.
The second identity is proved by applying $2g$ times the first identity and using the IHX relation.
\end{proof}

\begin{lemma}[Commutation identity for $\omega $-vertex]
\label{lem:commutation}
In the space $\Ag{}$, an $\omega$-vertex commutes with any \og-diagram, i.e.
  
\centerline{\relabelbox \small
\epsfxsize 5truein \epsfbox{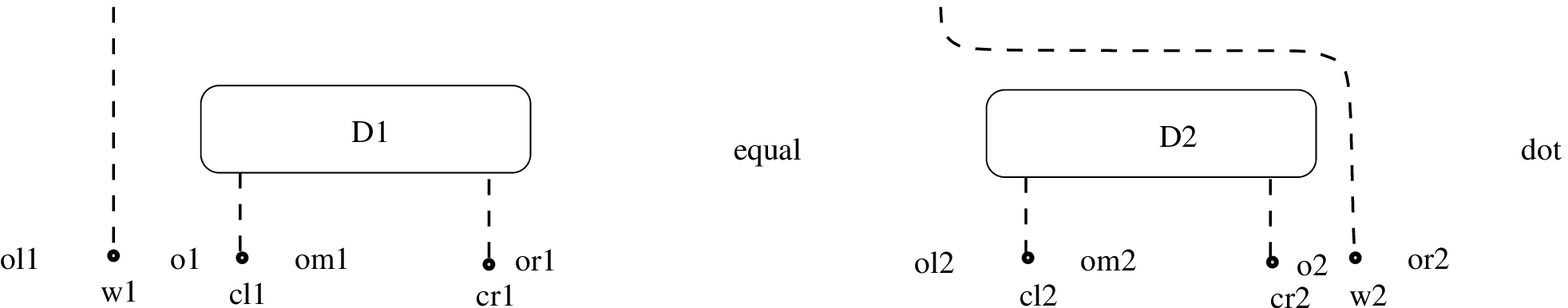}
\adjustrelabel <0cm,-0cm> {ol1}{$\cdots<$}
\adjustrelabel <0cm,-0cm> {or1}{$<\cdots$}
\adjustrelabel <0cm,-0cm> {o1}{$<$}
\adjustrelabel <0cm,-0cm> {om1}{$<\cdots<$}
\adjustrelabel <0cm,-0cm> {w1}{$\omega$}
\adjustrelabel <0cm,-0cm> {cl1}{$x_1$}
\adjustrelabel <0cm,-0cm> {cr1}{$x_r$}
\adjustrelabel <0cm,-0cm> {D1}{$D$}
\adjustrelabel <0cm,-0cm> {ol2}{$\cdots<$}
\adjustrelabel <0cm,-0cm> {or2}{$<\cdots$}
\adjustrelabel <0cm,-0cm> {o2}{$<$}
\adjustrelabel <0cm,-0cm> {om2}{$<\cdots<$}
\adjustrelabel <0cm,-0cm> {w2}{$\omega$}
\adjustrelabel <0cm,-0cm> {cl2}{$x_1$}
\adjustrelabel <0cm,-0cm> {cr2}{$x_r$}
\adjustrelabel <0cm,-0cm> {D2}{$D$}
\adjustrelabel <0cm,-0cm> {equal}{$=$}
\adjustrelabel <0cm,-0cm> {dot}{,}
\endrelabelbox}

\noindent
where $x_1,\dots,x_r$ belong to $H_\Q \cup \{\omega\}$.
\end{lemma}

\begin{proof}
According to Lemma \ref{lem:STU-omega},
the difference between the left-hand side term and the right-hand side term is given by 
$$
d:= \begin{array}{c}
{\relabelbox \small
\epsfxsize 1.5truein \epsfbox{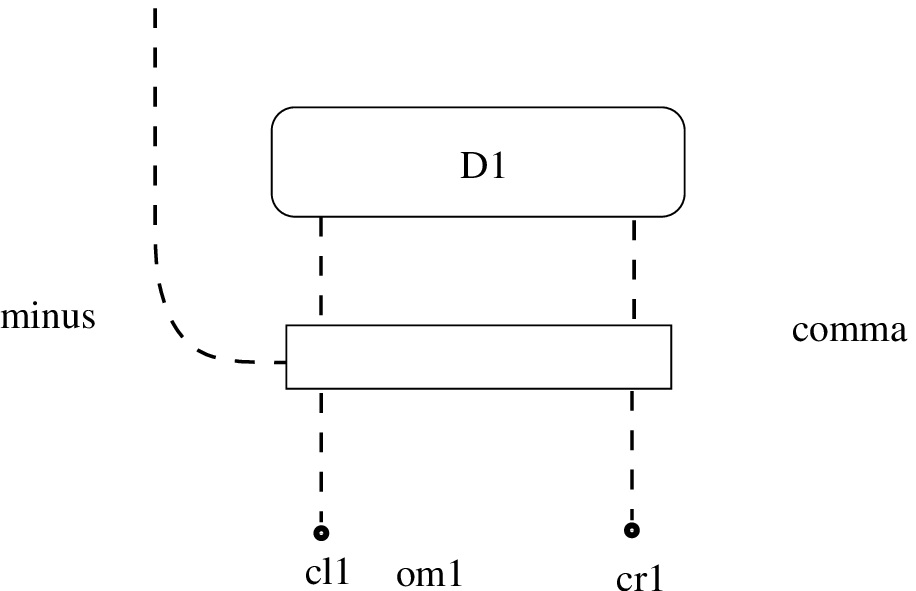}
\adjustrelabel <-0.15cm,0.15cm> {om1}{\scriptsize $<\cdots<$}
\adjustrelabel <0cm,-0cm> {cl1}{\scriptsize $x_1$}
\adjustrelabel <0cm,-0cm> {cr1}{\scriptsize $x_r$}
\adjustrelabel <0cm,-0.05cm> {D1}{$D$}
\adjustrelabel <0cm,-0cm> {minus}{$-$}
\adjustrelabel <0cm,-0cm> {comma}{,}
\endrelabelbox} \end{array}
$$
where the ``box'' notation is recalled in Figure \ref{fig:box}.
Then, this box can be slided upwards (see Figure \ref{fig:slide_box}) which shows that $d=0$.
\end{proof}

\begin{figure}[h]
\centerline{\relabelbox \small
\epsfxsize 5truein \epsfbox{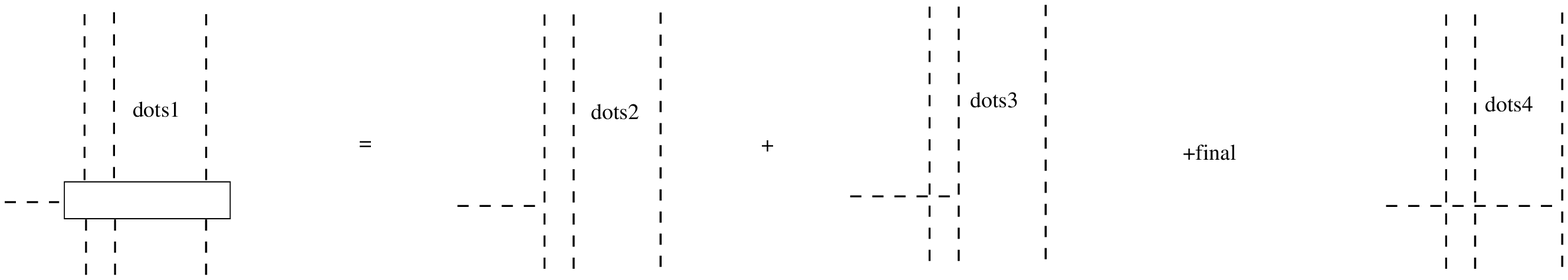}
\adjustrelabel <-0cm,-0.cm> {dots1}{$\dots$}
\adjustrelabel <-0cm,-0.cm> {dots2}{$\dots$}
\adjustrelabel <-0cm,-0.cm> {dots3}{$\dots$}
\adjustrelabel <0.05cm,-0.cm> {dots4}{$\dots$}
\adjustrelabel <-0cm,-0.cm> {=}{$:=$}
\adjustrelabel <-0cm,-0.cm> {+}{$+$}
\adjustrelabel <-0cm,-0.cm> {+final}{$+\cdots +$}
\endrelabelbox}
\caption{The box notation.}
\label{fig:box}
\end{figure}

\begin{figure}[h]
\centerline{\relabelbox \small
\epsfxsize 4truein \epsfbox{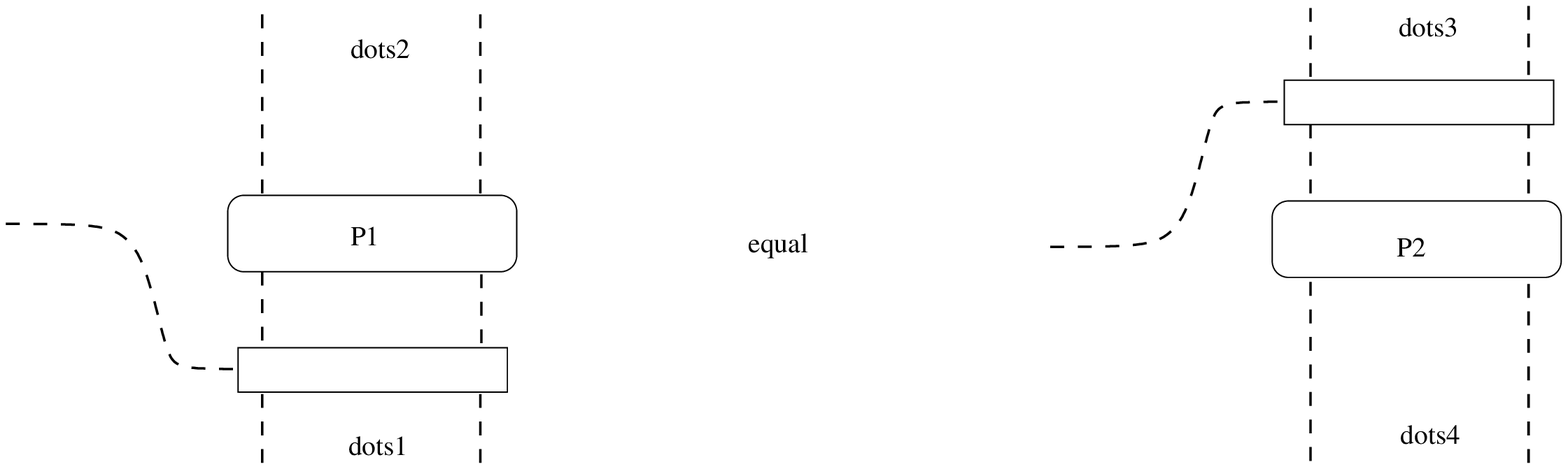}
\adjustrelabel <-0cm,-0.cm> {dots1}{$\dots$}
\adjustrelabel <-0cm,-0.cm> {dots2}{$\dots$}
\adjustrelabel <-0cm,-0.cm> {dots3}{$\dots$}
\adjustrelabel <0.0cm,-0.cm> {dots4}{$\dots$}
\adjustrelabel <-0cm,-0.cm> {equal}{$=$}
\adjustrelabel <-0cm,-0.cm> {P1}{$P$}
\adjustrelabel <-0cm,-0.cm> {P2}{$P$}
\endrelabelbox}
\caption{As a consequence of the AS and IHX relations, 
a box can be ``slided'' over a part $P$ of a Jacobi diagram \emph{without} external vertices.
(See for instance \cite{DKC}.)}
\label{fig:slide_box}
\end{figure}

\begin{lemma}
  \label{lem:symplectic_action}
  $\SpHQ$ acts trivially on each $\omega $-vertex in a \og-diagram.  That
  is, if $D$ is an \og-diagram with some $\omega $-vertices and
  other vertices $u_1,\ldots ,u_l$ labeled by $c_1,\ldots ,c_l\in \HQ$
  respectively, then, for each $F \in \SpHQ$, $F \cdot D\in \Ag{}$ is represented
  by the diagram obtained from $D$ by changing the labels for
  $u_1,\ldots ,u_l$ with $F(c_1),\ldots ,F(c_l)$ respectively and by leaving the $\omega$-vertices unchanged.
\end{lemma}

\begin{proof}
This is immediate from the fact that $F \cdot \omega =\omega \in \Lambda^2 \HQ$ for all $F \in \SpHQ$.
\end{proof}

An \og-diagram is called  an {\em $\omega $-smallest diagram} if the
smallest external vertex is an $\omega $-vertex.
Let $I^<\subset \Ag{}$ denote the (homogeneous) subspace spanned by
$\omega $-smallest diagrams.

\begin{proposition}
  \label{prop:Hopf_ideal}
  $I^<$ is a Hopf ideal in $\Ag{}$, closed under the $\SpHQ$-action.
\end{proposition}

\begin{proof}
  It is clear that $I^<$ is a right ideal in the algebra $\Ag{}$.
  Lemma \ref{lem:commutation} implies that $I^<$ is a left ideal in $\Ag{}$ as well.
  Thus $I^<$ is a two-sided ideal.  It is easy to check that $I^<$ is
  a Hopf ideal.  By Lemma \ref{lem:symplectic_action}, $I^<$ is closed under
  the $\SpHQ$-action.
\end{proof}

Therefore, the quotient $\Ag{}/I^<$ is a graded Hopf algebra with $\SpHQ$-action.
It also follows from Proposition \ref{prop:Hopf_ideal} that
\begin{gather*}
  I^{<,c}:=I^<\cap \Agc{}
\end{gather*}
is an ideal in the Lie algebra $\Agc{}$.
Using the fact that $\Ag{}=U(\Agc{})$, it can be seen that 
the subspace $I^{<,c}$ is spanned by connected $\omega $-smallest diagrams.

\subsection{Diagrammatic description of $\Gr^Y \Cg \oQ$}

We can now state the main result of this section, 
which is the analogue of Theorem \ref{th:surgery-LMO} in the closed surface case.
To relate the bordered surface case to the closed surface case,
we fix an embedding 
\begin{equation}
\label{eq:surface_embedding}
\bfi \colon\thinspace \sgg\hookrightarrow\Sigma _g,
\end{equation}
which induces a surjective homomorphism of graded Lie algebras with $\SpHQ$-action:
\begin{gather*}
  \Gr \bfi \oQ \colon\thinspace \Gr^Y\Cgg\oQ \longrightarrow \Gr^Y\Cg\oQ.
\end{gather*}

\begin{theorem}
  \label{th:surgery-LMO_closed}
  Let $g\ge0$.  There exist mutually inverse isomorphisms of graded Lie
  algebras with $\SpHQ$-action
  \begin{gather}
    \A^{<,c}\sHQ/I^{<,c}\sHQ \overset{\psi}{\underset{\LMO}{\longrightleftarrows}}
    \Gr^Y\Cg\oQ, 
  \end{gather}
  such that the following diagram is commutative:
  \begin{equation}
    \label{eq:bordered_to_closed}    
    \xymatrix{
      {\Agc{}} \ar[d]^-{\psi}_-\simeq \ar@{->>}[rr]^{\operatorname{proj}}
      && {\Agc{}/I^{<,c}}   \ar[d]_-{\psi}^-\simeq\\
	 {\Gr^Y\Cgg \oQ} \ar@{->>}[rr]^{\Gr \bfi \otimes \Q} 	 \ar@/^1pc/[u]^-{\LMO}
	 && {\Gr^Y\Cg\oQ}.  \ar@/_1pc/[u]_-{\LMO}
    }
  \end{equation}
\end{theorem}

\noindent
The proof of Theorem \ref{th:surgery-LMO_closed} is given 
in \S~\ref{subsec:surgery_closed} and \S~\ref{subsec:LMO_closed}.

\begin{remark}
The definition of the surgery map 
$\psi\colon\thinspace \Agc{}/I^{<,c}\longrightarrow\Gr^Y \Cg \otimes \Q$
is suggested in \cite{Habiro}, where it is also conjectured to be an isomorphism.
The degree $1$ case is done in \cite{MM}.
\end{remark}

\subsection{The surgery map $\psi$}

\label{subsec:surgery_closed}

Let $\psi: \A^{<,c} \to \Gr^Y \Cgg \otimes \Q$ be the surgery map defined 
in \S~\ref{subsec:surgery} in the case of the bordered surface $\sgg$. 

\begin{lemma}
\label{lem:surgery_closed}
The map $(\Gr \bfi \otimes \Q) \circ \psi$ vanishes on the subspace $I^{<,c}$.
\end{lemma}

\noindent
Consequently, the surgery map for $\sgg$ induces a \emph{surgery map} for $\sg$
\begin{gather*}
  \psi\colon\thinspace \Agc{}/I^{<,c} \longrightarrow \Gr^Y\Cg\oQ,
\end{gather*}
such that $\psi \circ \operatorname{proj} = (\Gr \bfi \otimes \Q) \circ \psi$.

\begin{proof}[Proof of Lemma \ref{lem:surgery_closed}]
  Let $D$ be a connected $\omega $-smallest diagram of degree $i$.
  By the rule (\ref{eq:split_omega}), we can assume that
  $D$ has only one $\omega$-vertex $v$, and 
  it suffices to show that $(\Gr \bfi \otimes \Q) \circ \psi (D)=0$.  
  Let $C$ be a connected graph clasper in $\sgg\times [-1,1]$, which 
  realizes ``topologically'' $D$ in the sense of \S~\ref{subsec:surgery}, 
  except that the leaf $L_v$ corresponding to $v$ goes along 
  $\partial \sgg\times \{t\}$, where $t\in (-1,1)$.
  
  The graph clasper $C$ can be regarded in the product 
  $\sg\times [-1,1]$ via the inclusion $\bfi: \sgg \hookrightarrow \sg$.
  There, the leaf $L_v$ bounds a disk whose interior does not intersect $C$.
  Therefore, we have
  \begin{gather}
    \label{e23}
    (\Sigma _g\times [-1,1])_{\bfi(C)}\cong \Sigma _g\times [-1,1].
  \end{gather}

  Since $v$ is the smallest external vertex, 
  one can arrange by an isotopy of $C$ that the level surface $\sgg\times \{t\}$, 
  bounded by the leaf $L_v$, does not intersect the edges of $C$. 
  So, by clasper calculus \cite{Habiro}, the graph clasper $C$ can be
  transformed to a clasper $C'$  by trading the leaf $L_v$ 
  for a box $B$ with $g$ input edges $e_1,\ldots ,e_g$ such that
  each $e_j$ is connected to a node $w_j$, which is itself
  connected by edges to two leaves representing the isotopy classes
  of the curves $\alpha_j$ and $\beta _j$, respectively.
  Recall that $(\alpha_1,\beta_1,\dots,\alpha_g,\beta_g)$ is a system of 
  meridians and parallels for the surface $\Sigma_{g,1}$, as shown on Figure \ref{fig:surface}.
 
  Next, by using the zip construction, we see that surgery along $C'$ is
  $Y_{i+1}$-equivalent to surgery along the disjoint union of $g$ graph claspers
  $C_1,\ldots ,C_g$, where $C_j$ is obtained from $C$ by replacing the leaf
  $L_v$ with a node connected by edges to two leaves
  representing $\alpha_j$ and $\beta_j$, respectively.

  Let $D_1,\ldots ,D_g$ denote the \og-diagrams obtained from $D$ by replacing
  $v$ with an internal vertex connected to two external vertices
  labeled by $\alpha _i$ and $\beta _i$, respectively, the former being declared
  smaller than the latter. Then we have
  \begin{gather*}
    \psi (\sum_{j=1}^gD_j) = \pm \sum_{j=1}^g
    \left\{ (\sgg\times [-1,1])_{C_j} \right\}_{Y_{i+1}}
    = \pm \left\{ (\sgg\times [-1,1])_{C'} \right\}_{Y_{i+1}},
  \end{gather*}
  where the second identity is proved by clasper calculus.
  By \eqref{e23}, $\Gr \bfi \oQ$ maps the right-hand side to $0$.
  Hence we have
  \begin{gather*}
    (\Gr \bfi \otimes \Q) \circ \psi(D) = (\Gr \bfi \otimes \Q) \circ \psi(\sum_{i=1}^gD_i)=0.
  \end{gather*}
\end{proof}

\subsection{The LMO map}

\label{subsec:LMO_closed}

In this subsection, we construct a Lie algebra homomorphism
\begin{equation}
  \label{eq:graded_LMO_closed}
  \LMO \colon\thinspace \Gr^Y\Cg\oQ \longrightarrow  \Agc{}/I^{<,c}.
\end{equation}
It is induced by the map $\LMO: \Gr^Y\Cgg\oQ \longrightarrow  \Agc{}$ of \S~\ref{subsec:LMO},
in the sense that $\LMO \circ (\Gr \bfi \oQ) = \operatorname{proj} \circ \LMO$.  
This will complete the proof of Theorem \ref{th:surgery-LMO_closed}.

To do this, it is enough to show that the monoid homomorphism
$\LMO= s \circ \varphi \circ \ZtildeY: \Cgg \to \A^<$ introduced in \S~\ref{subsec:LMO} 
induces a monoid homomorphism
\begin{equation}
\label{eq:LMO_closed}
\LMO: \Cg \longrightarrow \A^{<}/I^<.
\end{equation}

We start by recalling from \cite{CHM} how the maps $\ZtildeY$ and $\varphi$ are defined.
First of all,  a LMO functor
\begin{gather*}
  \Ztilde\colon\thinspace \LCob \longrightarrow \tsA
\end{gather*}
is defined from the category $\LCob$ of ``Lagrangian cobordisms''
to the category $\tsA$ of ``top-substantial Jacobi diagrams''.
The objects of  $\LCob$ are integers $g\geq 0$ 
which index the surfaces $\Sigma_{g,1}$ by their genus\footnote{
To be exact, one has to choose a parenthesizing of 
the handles of the surface $\sgg$ for each $g\geq 0$.}, 
while morphisms in $\LCob(g,h)$ are cobordisms from $\Sigma_{g,1}$ to $\Sigma_{h,1}$
satisfying certain homological conditions. As for the category $\tsA$, it has non-negative integers as objects.
For $g,h\ge 0$, $\tsA(g,f)$ is (the degree completion of) a $\Q $-vector space
spanned by Jacobi diagrams with external vertices labelled by
$\set{g}^+:=\{1^+,\ldots ,g^+\}$ or $\set{f}^-:=\{1^-,\ldots ,f^-\}$ 
such that each strut component, if any,
does {\em not} connect an $i^+$ with a $j^+$, for some $1\le i,j\le g$.
The identity of $g\geq 0$ in the category $\tsA$ 
is the following exponential of struts:
$$
\Id_g = \exp_\sqcup  \left(\sum_{i=1}^g\strutgraph{i^-}{i^+}\right).
$$

Since $\LCob(g,g)$ contains the monoid $\Cgg$ of homology
cylinders, $\Ztilde$ restricts to an invariant of homology cylinders.
Let $\AY\left(\set{g}^+ \cup \set{g}^-\right)$ be the
space of Jacobi diagrams with external vertices colored by $\set{g}^+ \cup \set{g}^-$
and without strut component, modulo the AS, IHX and multilinearity relations.
It is shown in \cite{CHM} that, for all $M\in \Cgg$, $\Ztilde(M)\in \tsA(g,g)$ splits as
$$
\Ztilde(M) = \Id_g \sqcup \ZtildeY(M)
$$
where 
$$
\ZtildeY(M) \in \AY\left(\set{g}^+ \cup \set{g}^-\right)
$$
denotes the $Y$-reduction of $\Ztilde(M)$. Thus, we obtain a monoid homomorphism
\begin{equation}
\label{eq:LMO_invariant}
\ZtildeY: \Cgg \longrightarrow \AY\left(\set{g}^+ \cup \set{g}^-\right),
\end{equation}
where the space $\AY\left(\set{g}^+ \cup \set{g}^-\right)$
is equipped with the multiplication $\star$ defined by
$$
D \star E :=
\left( \begin{array}{c}
\hbox{sum of all ways of gluing \emph{some} of the $i^+$-colored vertices of $D$}\\
\hbox{to \emph{some} of the $i^-$-colored vertices of $E$, for all $i=1,\dots,g$}
\end{array} \right) .
$$ 
This product was introduced in \cite{GL_tree-level}.

As for the graded algebra isomorphism 
$$
\varphi: \left(\AY\left(\set{g}^+ \cup \set{g}^-\right),\star\right) 
\longrightarrow \left(\A^<(-H_\Q), \osqcup\right),
$$ 
it is defined by declaring that ``each $i^-$-colored vertex should be
smaller than any $i^+$-colored vertex'' and by changing the colors of
external vertices according to the rules $(i^- \mapsto \alpha_i)$ and
$(i^+ \mapsto \beta_i)$. See \cite{CHM}.\\

Instead of considering cobordisms between surfaces with one boundary component, 
one can consider cobordisms between closed surfaces.
We defined in \cite{CHM} a congruence relation $\sim$ on $\LCob$
such that the quotient category $\LCob/\!\!\sim$ gives the category of 
``Lagrangian cobordisms'' between closed surfaces.

For $g\ge 0$, let  $a_g\in \tsA(g+1,g)$ be defined by\\[0.1cm]

\centerline{\relabelbox \small
\epsfxsize 3.5truein \epsfbox{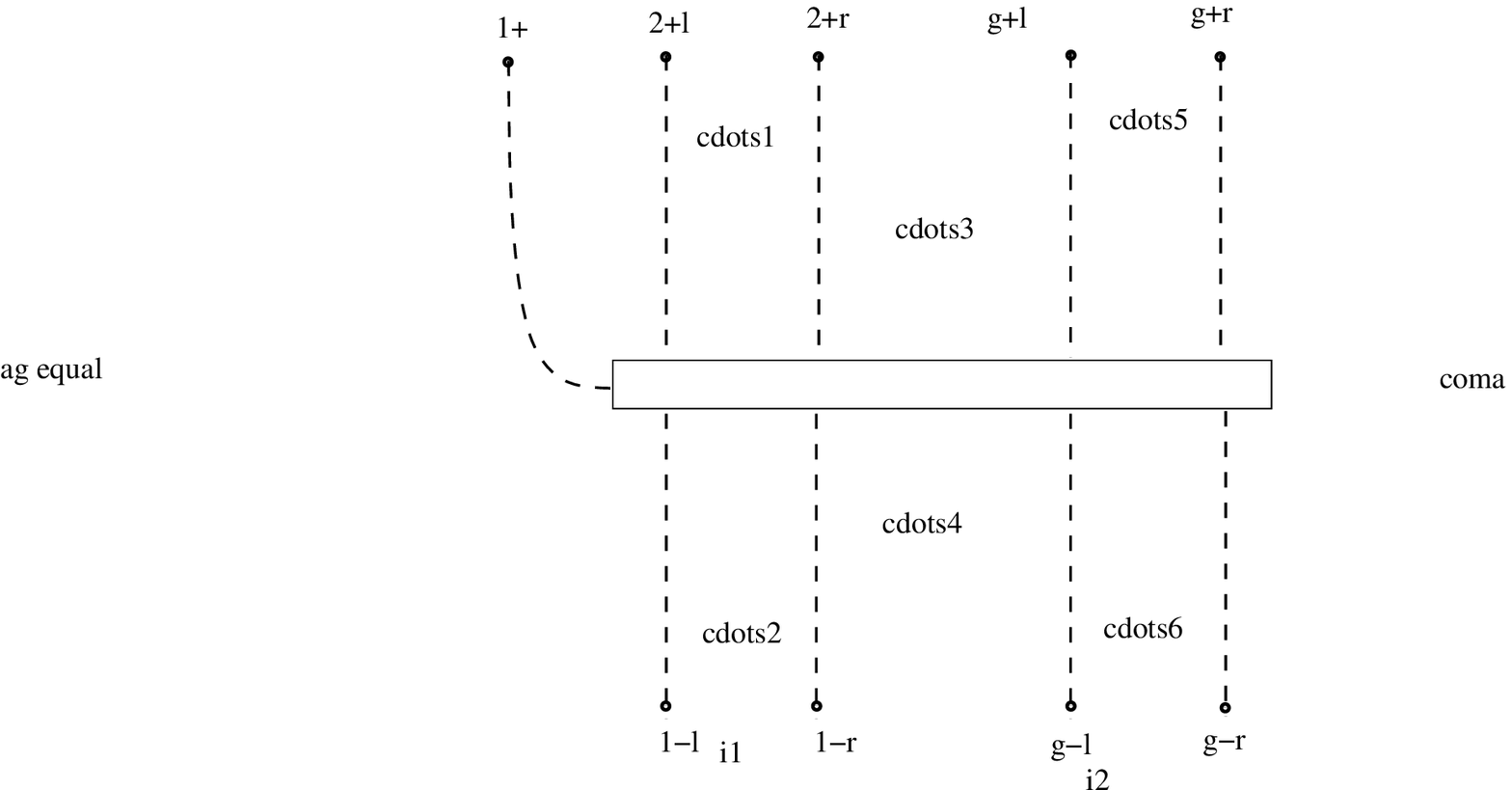}
\adjustrelabel <0cm,-0cm> {cdots1}{$\cdots$}
\adjustrelabel <0cm,-0cm> {cdots2}{$\cdots$}
\adjustrelabel <0cm,-0cm> {cdots3}{$\cdots$}
\adjustrelabel <0cm,-0cm> {cdots4}{$\cdots$}
\adjustrelabel <0cm,-0cm> {cdots5}{$\cdots$}
\adjustrelabel <0cm,-0cm> {cdots6}{$\cdots$}
\adjustrelabel <-1cm,-0cm> 
{ag equal}{${\displaystyle a_g := \sum_{i_1,\dots,i_g\geq 0} \frac{1}{i_1! \cdots i_g!}}$}
\adjustrelabel <0cm,-0cm> {1+}{\scriptsize $1^+$}
\adjustrelabel <0cm,-0cm> {2+l}{\scriptsize $2^+$}
\adjustrelabel <0cm,-0cm> {2+r}{\scriptsize $2^+$}
\adjustrelabel <0cm,-0cm> {g+l}{\scriptsize $(g+1)^+$}
\adjustrelabel <0cm,-0cm> {g+r}{\scriptsize $(g+1)^+$}
\adjustrelabel <0cm,-0cm> {1-l}{\scriptsize $1^-$}
\adjustrelabel <0cm,-0cm> {1-r}{\scriptsize $1^-$}
\adjustrelabel <0cm,-0cm> {g-l}{\scriptsize $g^-$}
\adjustrelabel <0cm,-0cm> {g-r}{\scriptsize $g^-$}
\adjustrelabel <-0.4cm,-0cm> {i1}{$\underbrace{\quad \quad \quad}_{i_1}$}
\adjustrelabel <-0.1cm,-0cm> {i2}{$\underbrace{\quad \quad \quad}_{i_g}$}
\adjustrelabel <0cm,-0cm> {coma}{$,$}
\endrelabelbox}
\vspace{0.5cm}

\noindent
where the ``box'' notation is used (see Figure \ref{fig:box}),
and with the convention that $a_0:=0$.

For $g,f\ge 0$, let $\mathcal{I}(g,f)$ denote the subspace of $\tsA(g,f)$ defined by
\begin{gather*}
  \mathcal{I}(g,f)= \{a_{f}\circ x\;|\; x\in \tsA(g,f+1)\}.
\end{gather*}
The vector spaces $\mathcal{I}(g,f)$, where $g,f\geq 0$, 
form an ideal in the linear category $\tsA$, i.e., we have
\begin{gather*}
  \tsA(g,f)\circ \mathcal{I}(h,g) \subset  \mathcal{I}(h,f),\\
  \mathcal{I}(g,f)\circ \tsA(h,g) \subset  \mathcal{I}(h,f).
\end{gather*}
Thus, we can consider the quotient category $\tsA/\mathcal{I}$. 
The ideal $\mathcal{I}$ of $\tsA$ is introduced 
in an equivalent way in \cite{CHM}, where it is proved that 
the LMO functor sends $\sim$ to the congruence relation defined by $\mathcal{I}$.
Hence a functor on the category of Lagrangian cobordisms between closed surfaces:
  \begin{gather*}
    \xymatrix{
      {\LCob} \ar[r]^{\Ztilde} \ar@{->>}[d]_{\operatorname{proj}}
      &{\tsA} \ar@{->>}[d]^{\operatorname{proj}}\\
      {\LCob/\!\!\sim} \ar@{-->}[r]^{\Ztilde} 
      &{\tsA/\mathcal{I}}
      }
  \end{gather*}

We are going to prove the following technical result.

\begin{lemma}
  \label{lem:I}
  Each element $y$ of $\AY\left(\set{g}^+ \cup \set{g}^-\right)$, 
  such that $y\sqcup \Id_g$ belongs to $\mathcal{I}(g,g)$, 
  is sent by the algebra map $s \circ \varphi$
  to an element of $I^<(H_\Q) \subset \A^<(H_\Q)$.
\end{lemma}

This lemma shows that $\LMO= s \circ \varphi \circ \ZtildeY: \Cgg \to \A^<$
induces a monoid homomorphism $\LMO:\Cg \to \A^</I^<$ as desired.
Indeed, assume that $M,M' \in \Cgg$ are such that $\bfi(M) = \bfi(M') \in \Cg$.
Then, the Lagrangian cobordisms $M$ and $M'$ are congruent so that 
$\Ztilde(M) - \Ztilde(M') \in \I(g,g)$. Lemma \ref{lem:I} implies that 
$$
\LMO(M)-\LMO(M') =s \circ \varphi(\ZtildeY(M) - \ZtildeY(M')) \in I^<.
$$
This will conclude the proof of Theorem \ref{th:surgery-LMO_closed}.

\begin{proof}[Proof of Lemma \ref{lem:I}]
Let $y \in \AY\left(\set{g}^+ \cup \set{g}^-\right)$ be such that
$$
t:= y \sqcup \exp_\sqcup\left(\sum_{i=1}^g\strutgraph{i^-}{i^+}\right) \ \in \I(g,g).
$$ 
Then, $t$ is a series of elements of the form
\begin{gather*}
u_{x}:=a_g\circ\left(x\sqcup 
\exp_\sqcup\left(\sum_{i=1}^g \strutgraph{(i+1)^-}{i^+}\quad \quad \right)\right),
\end{gather*}
where either 
\begin{enumerate}
\item $x\in \AY\left(\set{g}^+ \cup \set{g+1}^-\right)$ 
is a diagram involving exactly one label $1^-$,
or
\item $x= x' \sqcup x''$ with 
$x'\in \AY\left(\set{g}^+ \cup \set{g+1}^-\right)$ involving no label $1^-$,
and $x''=\strutgraph{1^-}{e^+}$ with $e\in \{1,\ldots ,g\}$ 
or $x''=\strutgraphbot{1^-}{e^-}$ with $e\in\{2,\ldots ,g+1\}$.
\end{enumerate}
Then, we have
\begin{gather*}
    u_{x}= u_{x}^Y
    \sqcup \exp_\sqcup\left(\sum_{i=1}^g\strutgraph{i^-}{i^+}\right),
\end{gather*}
where 

\centerline{\relabelbox \small
\epsfxsize 5truein \epsfbox{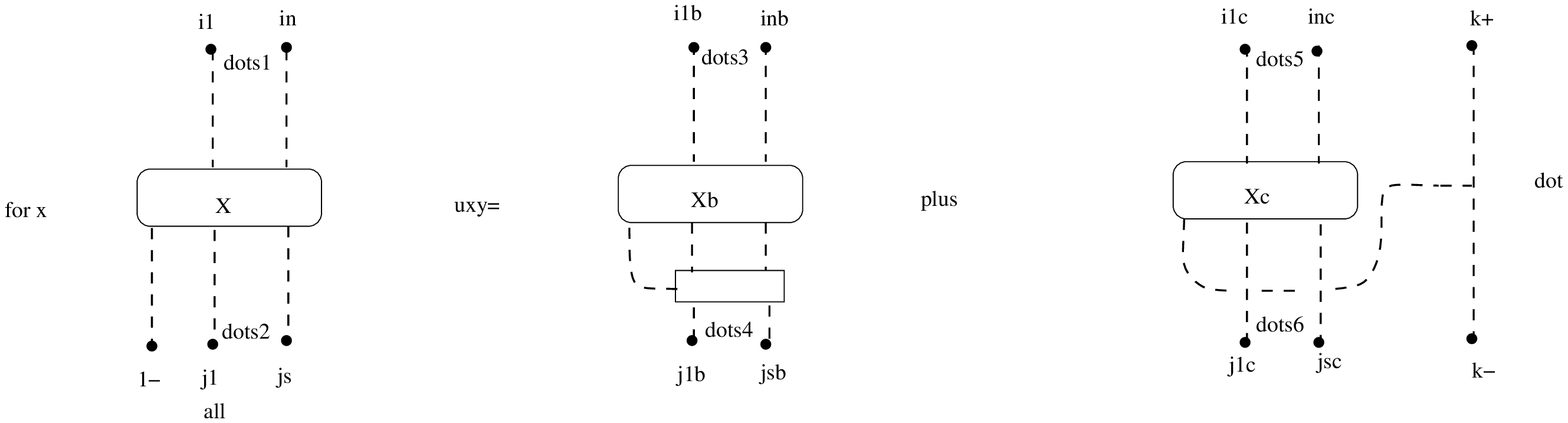}
\adjustrelabel <-0cm,-0.cm> {dots1}{$\dots$}
\adjustrelabel <-0cm,-0.cm> {dots2}{$\dots$}
\adjustrelabel <-0cm,-0.cm> {dots3}{$\dots$}
\adjustrelabel <-0cm,-0.cm> {dots4}{$\dots$}
\adjustrelabel <-0cm,-0.cm> {dots5}{$\dots$}
\adjustrelabel <-0cm,-0.cm> {dots6}{$\dots$}
\adjustrelabel <-0.25cm,0.05cm> {for x}{for \ $x=$}
\adjustrelabel <-0.6cm,-0.cm> {uxy= }{, \quad \quad $u_x^Y=$}
\adjustrelabel <-0cm,-0.05cm> {plus}{${\displaystyle +  \quad \quad \sum_{k=1}^g}$}
\adjustrelabel <-0cm,-0.cm> {1-}{$1^-$}
\adjustrelabel <-0cm,0.2cm> {all}{$\underbrace{\quad \ \ \quad}_{\neq 1^- }$}
\adjustrelabel <-0cm,-0.cm> {i1}{\scriptsize $i_1^+$}
\adjustrelabel <-0cm,-0.cm> {in}{\scriptsize $i_n^+$}
\adjustrelabel <-0cm,-0.cm> {j1}{\scriptsize $j_1^-$}
\adjustrelabel <-0cm,-0.cm> {js}{\scriptsize $j_s^-$}
\adjustrelabel <-0cm,-0.cm> {i1b}{\scriptsize $i_1^+$}
\adjustrelabel <-0cm,-0.cm> {inb}{\scriptsize $i_n^+$}
\adjustrelabel <-0.6cm,-0.cm> {j1b}{\scriptsize $(j_1\!-\!1)^-$}
\adjustrelabel <-0cm,-0.cm> {jsb}{\scriptsize $(j_s\!-\!1)^-$}
\adjustrelabel <-0cm,-0.cm> {i1c}{\scriptsize $i_1^+$}
\adjustrelabel <-0cm,-0.cm> {inc}{\scriptsize $i_n^+$}
\adjustrelabel <-0.6cm,-0.05cm> {j1c}{\scriptsize $(j_1\!-\!1)^-$}
\adjustrelabel <-0cm,-0.1cm> {jsc}{\scriptsize $(j_s\!-\!1)^-$}
\adjustrelabel <-0cm,-0.cm> {k-}{\scriptsize $k^-$}
\adjustrelabel <0cm,-0.cm> {k+}{\scriptsize $k^+$}
\adjustrelabel <0cm,-0.cm> {dot}{.}
\adjustrelabel <0cm,-0.cm> {X}{$X$}
\adjustrelabel <0cm,-0.cm> {Xb}{$X$}
\adjustrelabel <0cm,-0.cm> {Xc}{$X$}
\endrelabelbox}
\vspace{0.7cm}
\noindent
So, it is enough to show that, for all $x$ of the above form,
$s\circ \varphi(u_x^Y)$ is an $\omega$-smallest diagram. 
This follows from Lemma \ref{lem:STU-omega}.
\end{proof}

\begin{remark}
The surgery map $\psi: \A^{<,c}/I^{<,c} \to \Gr^Y \Cg \oQ$ 
does not depend on the choice (\ref{eq:surface_embedding}) 
of the embedding  $\bfi:\sgg \hookrightarrow \sg$
because, by clasper calculus, it can be defined directly with no reference to $\sgg$.
Consequently, $\LMO=\psi^{-1}: \Gr^Y \Cg \oQ \to \A^{<,c}/I^{<,c}$ 
is independent of $\bfi$ too. But, the monoid homomorphism
$\LMO: \Cg \to \A^{<}/I^<$ does depend on the choices of $\bfi$ 
and $(\alpha_1,\dots,\alpha_g,\beta_1,\dots,\beta_g)$, see Remark \ref{rem:basis}.
\end{remark}

\subsection{The ideals $I\subset \A$ and $I^c\subset \Ac{}$}

\label{sec:symmetrized_ideals}

We defined in \S~\ref{subsec:SJD}  an $\SpHQ$-equivariant, graded Hopf algebra
isomorphism
\begin{gather*}
  \chi  : \AHQ \simeqto \Ag{},
\end{gather*}
which restricts to an $\SpHQ$-equivariant, graded Lie algebra
isomorphism
\begin{gather*}
  \chi  : \Ac{} \simeqto \Agc{}.
\end{gather*}
The purpose of this subsection is to compute
\begin{gather*}
  I:=\chi ^{-1}(I^<)\subset \AHQ,\quad 
  I^c:=\chi ^{-1}(I^{<,c})\subset \Ac{},
\end{gather*}
which are ideals in $\A$ and $\Ac{}$, respectively, closed under the $\SpHQ$-action. 

\begin{proposition}
\label{prop:I_symmetrized}
The subspace $I$ of $\A$ is spanned by elements of the form
\begin{equation}
\label{eq:I_relations}
\begin{array}{c}
{\relabelbox \small
\epsfxsize 5.5truein \epsfbox{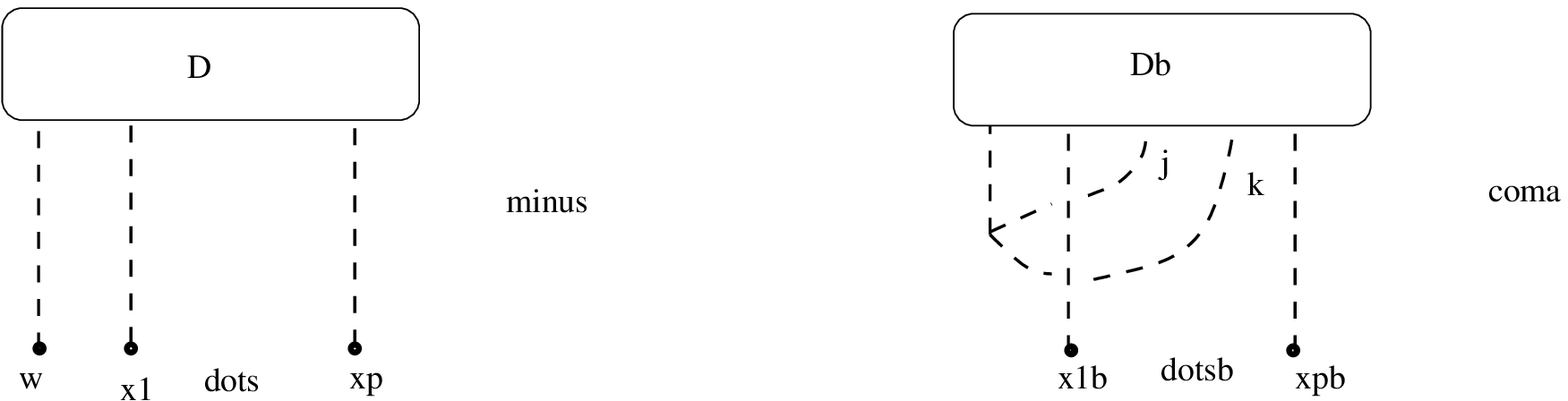}
\adjustrelabel <0cm,-0.cm> {w}{$\omega$}
\adjustrelabel <0cm,0.15cm> {x1}{\scriptsize $x_1$}
\adjustrelabel <0cm,0.1cm> {xp}{\scriptsize $x_e$}
\adjustrelabel <-0.1cm,0.05cm> {x1b}{\scriptsize $x_1$}
\adjustrelabel <0.005cm,0.05cm> {xpb}{\scriptsize $x_e$}
\adjustrelabel <0cm,0.1cm> {dots}{\scriptsize $\cdots \cdots$}
\adjustrelabel <0cm,-0.cm>
{minus}{${\displaystyle -\frac{1}{4}\sum_{1\leq j <k \leq e} \omega(x_j,x_k)}$}
\adjustrelabel <0cm,-0.cm> {D}{$D$}
\adjustrelabel <0cm,-0.cm> {Db}{$D$}
\adjustrelabel <-0.05cm,0.1cm> {j}{\scriptsize $j$}
\adjustrelabel <-0.05cm,0.25cm> {k}{\scriptsize $k$}
\adjustrelabel <-0.7cm,-0.05cm> {dotsb}{\scriptsize $\cdots \widehat{x_j}\cdots\widehat{x_k}\cdots $}
\adjustrelabel <-0.3cm,-0.cm> {coma}{,}
\endrelabelbox}
\end{array}\hspace{-1.8cm}
\end{equation}
where $x_1\dots,x_e\in H_\Q$ and where the $\omega$-vertex means
\begin{equation}
\label{eq:split_omega_bis}
\begin{array}{c}
{\relabelbox \small
\epsfxsize 2.5truein \epsfbox{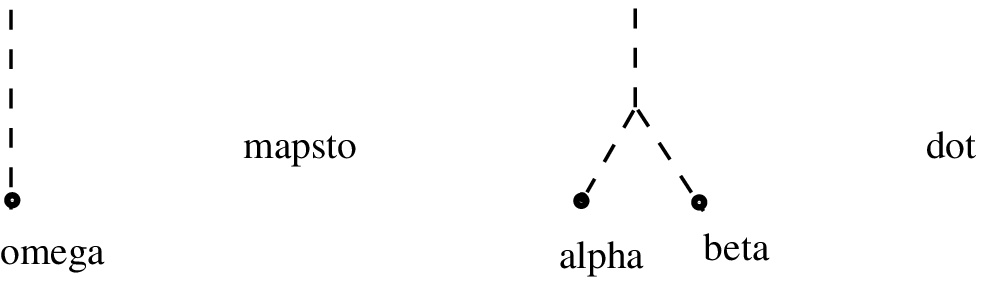}
\adjustrelabel <0cm,-0cm> {mapsto}{${\displaystyle = \quad \sum_{i=1}^g}$}
\adjustrelabel <0cm,-0cm> {omega}{$\omega$}
\adjustrelabel <0cm,-0.05cm> {alpha}{$\alpha_i$}
\adjustrelabel <0cm,-0.05cm> {beta}{$\beta_i$}
\adjustrelabel <0cm,-0cm> {dot}{.}
\endrelabelbox}
\end{array}
\end{equation} 
Moreover, the subspace $I^c$ of $\Ac{}$ is spanned 
by elements of the form (\ref{eq:I_relations}) where $D$ is assumed to be connected.
\end{proposition}

\begin{proof}
Let $\cI$ denote the $\Q $-vector space spanned by $\omega $-diagrams with
exactly one $\omega $-vertex, modulo the AS, IHX and multilinearity relations.  
Similarly, let $\cI^<$ denote the $\Q $-vector space spanned by $\omega $-smallest
diagrams with exactly one $\omega $-vertex, modulo the AS, IHX, STU-like
and multilinearity relations.  
(Here, the STU-like relation is not applicable to the $\omega $-vertex.)  
There is a natural surjective linear map
\begin{gather*}
  \lambda ^<\colon\thinspace \cI^< \longrightarrow I^<,
\end{gather*}
which maps each diagram in $\cI^<$ to itself, using the rule (\ref{eq:split_omega}).  

We consider the linear map
\begin{gather*}
  k\colon\thinspace \cI \longrightarrow \cI^<,
\end{gather*}
that maps each diagram $D \in \cI$ with $e$ 
$H_\Q$-colored external vertices to the average of the $e!$ $\omega $-smallest
diagrams obtained from $D$ by choosing any order on the $H_\Q$-colored external vertices,
and by declaring them to be larger than the $\omega $-vertex. 
One can prove that $k$ is an isomorphism, 
similarly to the proof in \S~\ref{subsec:SJD} that
$\chi \colon\thinspace \AHQ\to \Ag{}$ is an isomorphism.

Next, we define the linear map
\begin{gather*}
  \lambda :=\chi^{-1} \circ \lambda ^< \circ k\colon\thinspace 
  \cI\xto{k}
  \cI^< \xto{\lambda ^<}
  I^< \xto{\chi ^{-1}}
  I.
\end{gather*}  
Since $\lambda$ is surjective, it is enough to prove that, 
for all $\omega$-diagram with only one $\omega$-vertex\\

\centerline{\relabelbox \small
\epsfxsize 2.5truein \epsfbox{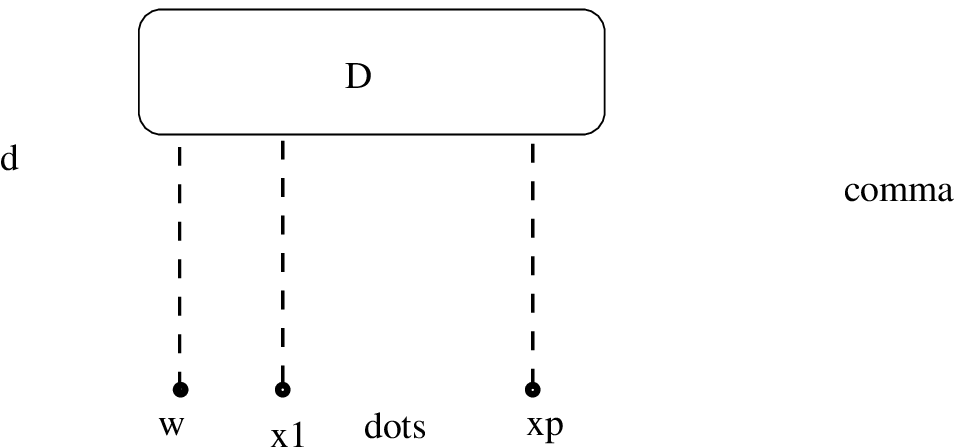}
\adjustrelabel <0cm,-0.cm> {w}{$\omega$}
\adjustrelabel <0cm,0.1cm> {x1}{\scriptsize $x_1$}
\adjustrelabel <0cm,0.05cm> {xp}{\scriptsize $x_e$}
\adjustrelabel <0cm,0.1cm> {dots}{\scriptsize $\cdots \cdots$}
\adjustrelabel <0cm,-0.cm> {D}{$D$}
\adjustrelabel <0cm,-0.4cm> {d}{$d=$}
\adjustrelabel <-0.5cm,-0.cm> {comma}{,}
\endrelabelbox}
\noindent
$\lambda(d)$ is given by formula (\ref{eq:I_relations}). Thus, we have to compute
\begin{eqnarray*}
\lambda(d) &= & \chi^{-1}\left(\begin{array}{c}{\relabelbox \small
\epsfxsize 2.5truein \epsfbox{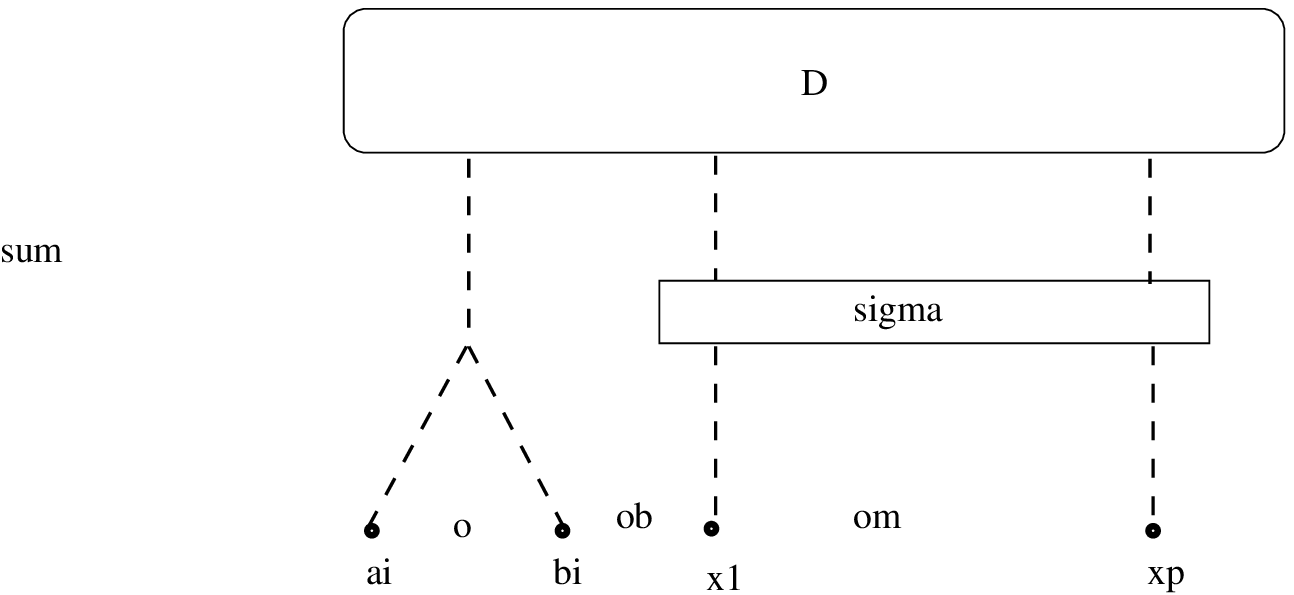}
\adjustrelabel <0cm,-0.3cm> 
{sum}{${\displaystyle \frac{1}{e!} \sum_{\substack{1\leq i \leq g\\ \sigma\in S_e}}}$}
\adjustrelabel <0cm,0cm> {x1}{$x_{\sigma(1)}$}
\adjustrelabel <0cm,0.0cm> {xp}{$x_{\sigma(e)}$}
\adjustrelabel <0cm,0cm> {ai}{$\alpha_i$}
\adjustrelabel <0cm,-0.1cm> {bi}{$\beta_i$}
\adjustrelabel <0cm,0cm> {o}{$<$}
\adjustrelabel <0cm,-0.05cm> {ob}{$<$}
\adjustrelabel <-0.1cm,-0.1cm> {om}{$<\cdots<$}
\adjustrelabel <0cm,-0.cm> {D}{$D$}
\adjustrelabel <0.2cm,-0.05cm> {sigma}{$\sigma$}
\endrelabelbox}\end{array}\right)\\
&=& t_0 + t_1 + t_2 + \cdots, 
\end{eqnarray*}
where $t_p$ is the term involving $p$ simultaneous contractions 
in the formula (\ref{eq:chi_inverse}) that describes $\chi^{-1}$.

The term $t_0$ is $d$ itself, seen as an element of $\A$ using  the rule (\ref{eq:split_omega_bis}).
The term $t_1$ involves three types of contraction:
\begin{enumerate}
\item The $\alpha_i$-colored vertex is contracted with the $\beta_i$-colored vertex: This gives 
a diagram with a looped edge, which is trivial by the AS relation.
\item The $\alpha_i$-colored vertex,
or the $\beta_i$-colored vertex, is contracted with an $x_{\sigma(j)}$-colored vertex.
The sum of all these contractions (over $i=1,\dots,g$ and $j=1,\dots,e$)
gives the following linear combination of diagrams:
$$
{\relabelbox \small
\epsfxsize 2.5truein \epsfbox{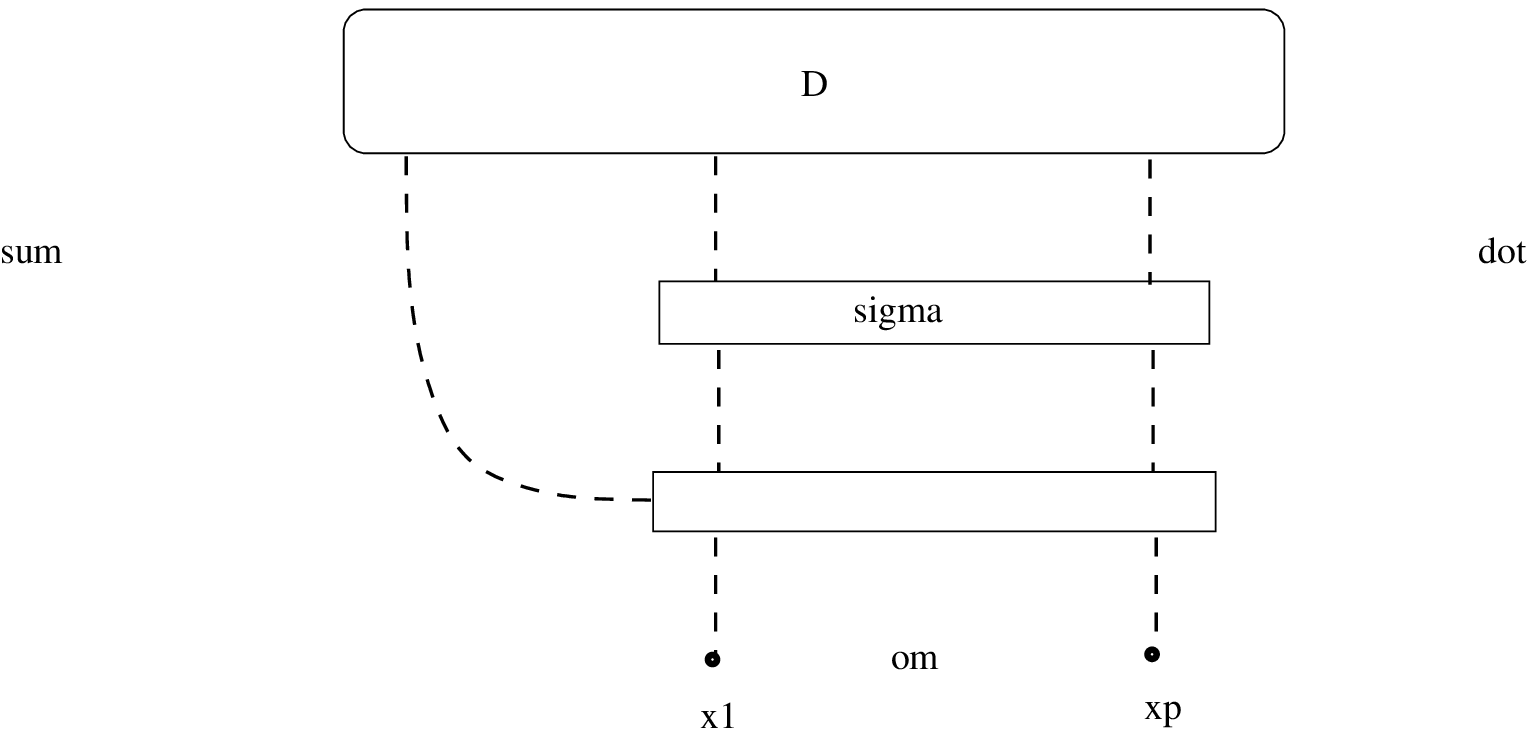}
\adjustrelabel <-0.5cm,-0.3cm> 
{sum}{${\displaystyle -\frac{1}{2\cdot e!} \sum_{\sigma\in S_e}}$}
\adjustrelabel <0cm,0cm> {x1}{$x_{\sigma(1)}$}
\adjustrelabel <0cm,0.0cm> {xp}{$x_{\sigma(e)}$}
\adjustrelabel <-0cm,-0cm> {om}{$\cdots$}
\adjustrelabel <0cm,-0.cm> {D}{$D$}
\adjustrelabel <0.2cm,-0.05cm> {sigma}{$\sigma$}
\adjustrelabel <0cm,-0.cm> {dot}{.}
\endrelabelbox}
$$
The generic term of this sum can be transformed 
to a diagram with a looped edge by sliding its box (see Figure \ref{fig:slide_box}).
So, it is trivial.
\item An $x_{\sigma(j)}$-colored vertex is contracted with an $x_{\sigma(i)}$-colored vertex:
Such a term will cancel with another one of the same kind since $\omega$ is skew-symmetric.
\end{enumerate}
So, $t_1$ is trivial.
As for the term $t_2$, the above arguments show that it is enough to consider
the contraction of the $\alpha_i$-colored vertex with an $x_{\sigma(j)}$-colored vertex
together with  the contraction of the $\beta_i$-colored vertex with an 
$x_{\sigma(k)}$-colored vertex (where $j\neq k$ and $j,k \in \{1,\dots,e\}$).
The sum of all these contractions gives the second term of the formula (\ref{eq:I_relations}).
Finally, the higher terms $t_p$ are trivial since 
a $p$-fold contraction with $p\geq 2$ should involve at least 
the contraction of an $x_{\sigma(j)}$-colored vertex with an $x_{\sigma(i)}$-colored vertex.
\end{proof}

\subsection{Algebraic description of $\Gr \mapcyl \oQ$}

In this subsection, we prove analogues of Theorems \ref{th:diagram}
and \ref{th:degree_2} for the closed surface $\Sigma_g$.
In this case, the first Johnson homomorphism is a map
$$
\tau_1: \Ig \longrightarrow \Lambda^3H/H,
$$
where $H$ embeds into $\Lambda^3H$ by $h \mapsto h\wedge \omega$. 
It is induced by the first Johnson homomorphism of the boundary case,
in the sense that the following diagram commutes:
$$
\xymatrix{
{\Igg} \ar[d]_-{\tau_1} \ar@{->>}[r]^-{\bfi} & {\Ig} \ar[d]^-{\tau_1}\\
{\Lambda^3 H} \ar@{->>}[r]_-{\operatorname{proj}} &
{\Lambda^3 H/H.}
} 
$$
The homomorphism $\tau_1$ induces an isomorphism
between $\Gr_1^\Gamma \Ig \otimes \Q$ and $\Lambda^3H_\Q/H_\Q$ \cite{Johnson},
hence a Lie algebra epimorphism
$$
J: \Lie(\Lambda^3 H_\Q/H_\Q) \longrightarrow \Gr^\Gamma \Ig \otimes \Q.
$$
Define $\operatorname{R}(\Ig):= \Ker(J)$. Hain's Theorem \ref{th:Hain} holds in the closed case as well.
See \cite{Hain} where the following result is also proved:

\begin{proposition}[Hain]
\label{prop:Hain}
If $g\geq 3$, then the $\Sp\sHQ$-module of quadratic relations $\operatorname{R}_2(\Ig)$ 
is spanned by the class of $r_1$, 
where $r_1 \in \Lie_2(\Lambda^3 H_\Q)$ has been defined in \S~\ref{subsec:quadratic_relations}. 
\end{proposition}

As an analogue of Theorem \ref{th:diagram} in the closed case, we obtain the following statement:

\begin{theorem}
  \label{th:diagram_closed}
  If $g\ge3$, then the following diagram in the category of graded Lie algebras 
  with $\Sp(H_\Q)$-actions is commutative:
  $$
    \xymatrix{
      {\Lie\left(\LHQ/H_\Q\right) \left/ {\rm{R}}(\Ig) \right. }
      \ar[d]_-{\overline{J}}^-\simeq \ar[rr]^-{\overline{Y}} 
      && {\A^{<,c}\sHQ/I^{<,c}\sHQ} \ar[d]_-{\psi}^-\simeq\\
	 {\Gr^\Gamma \Ig \oQ} \ar[rr]^{\Gr \mapcyl \oQ}
	 && {\Gr^Y \Cg \oQ} \ar@/_1pc/[u]_-{\LMO}
    }
  $$
  Here, $\overline{Y}$ is induced by the Lie algebra map
  $Y: \Lie\left(\LHQ/H_\Q\right)\to \A^{<,c}\sHQ/I^{<,c}\sHQ$ which, in degree $1$,
  sends the class of $x\wedge y\wedge z$ to the class of the $Y$-graph $\Ygraphtop{x}{y}{z}$.
\end{theorem}

\begin{proof}
The diagram commutes because its analogue for the surface $\sgg$ does.
\end{proof}

Let us now consider the situation in degree $2$. 
As in the bounded case, it is convenient to switch from $\A^</I^<$ to $\A/I$.
Thus, we consider the Lie bracket of $\A^c/I^c$ in degree $1+1$:
$$
c_2 := [-,-]_\star: \Lambda^2\left(\A^c_1/I_1^c\right) \longrightarrow \A^c_2/I^c_2.
$$
According to Proposition \ref{prop:I_symmetrized}, 
the subspace $I_2^c\subset \A_2^c$ is generated by 
$$
\Ygraphbotbottop{x}{y}{\omega} -\frac{\omega(x,y)}{4} \cdot \thetagraph\ , \quad \quad \forall x,y \in H_\Q.
$$
So, $I_2^c$ is contained in the even part $\A_{2,\ev}^{c}$.

\begin{proposition} 
If $g\geq 3$, then we have 
$$
\Img(c_2) = \A_{2,\ev}^{c}/I_2^c \quad  \left(\simeq \Gamma_0 + \Gamma_{2 \omega_2}\right)
\quad \quad \hbox{and} \quad \quad
\Ker(c_2) = \langle\, \{r_1\}\, \rangle_{\Sp\sHQ}.
$$
\end{proposition}

\begin{proof}
Proposition \ref{prop:im_bracket} and the  commutative diagram
$$
\xymatrix{
{\Lambda^2 \Lambda^3 H_\Q} \ar@{->>}[d]_-{\operatorname{proj}} \ar[rr]^-{b_2=\chi_2^{-1} Y_2}
& &{\A_2^{c}} \ar@{->>}[d]^-{\operatorname{proj}} \\
{\Lambda^2\left( \Lambda^3 H_\Q/H_\Q\right)} \ar[rr]^-{c_2=\chi_2^{-1} Y_2} 
& &{\A_2^{c}/I_2^c}
}
$$
shows that the image of $c_2$ coincides with $\A_{2,\ev}^{c}/I_2^c$. 
The irreducible decomposition of $\A_{2,\ev}^{c}/I_2^c$ is deduced from the decomposition of $\A_{2,\ev}^{c}$:
$$
\begin{array}{c|c}
\hbox{eigenvalue} & \hbox{eigenvector} \\
\hline 
2 \omega_2 & \Big\{ \Hgraph{\alpha_1}{\alpha_2}{\alpha_2}{\alpha_1} \Big\} \\
\hline
0 & \{ \strutgraphbot{\omega}{\omega} \} = \big\{ \frac{g}{4}\cdot \thetagraph \big\}
\end{array}
$$

Let us now identify the kernel of $c_2$. First, we observe that 
$$
\Ker \left( \A_{2,\ev}^{c} \to \A_{2,\ev}^{c}/I_2^c \right) =
\left\langle \Ygraphbotbottop{\alpha_1}{\alpha_2}{\omega},
\strutgraphbot{\omega}{\omega} - \frac{g}{4}\cdot \thetagraph
\right\rangle_{\Sp\sHQ}.
$$ 
Moreover, the proof of Proposition \ref{prop:im_bracket} shows that
$$
\begin{array}{ll}
& b_2\left( -
\Ygraphbotbottop{\alpha_1}{\alpha_2}{\beta_1} \wedge
\strutgraphbot{\omega}{\alpha_1} \right) =
\Ygraphbotbottop{\alpha_1}{\alpha_2}{\omega},\\
\hbox{and} & b_2\left({ \displaystyle\frac{1}{g-1}
\sum_{i=1}^g \strutgraphbot{\omega}{\alpha_i} \wedge \strutgraphbot{\omega}{\beta_i}}\right) = 
\strutgraphbot{\omega}{\omega} - \frac{g}{4}\cdot \thetagraph.
\end{array}
$$
Thus, 
$$
\Ker(c_2) = \operatorname{proj}
\left(\Ker(b_2) + \left\langle 
\Ygraphbotbottop{\alpha_1}{\alpha_2}{\beta_1} \wedge \strutgraphbot{\omega}{\alpha_1},
\frac{1}{g-1}\sum_{i=1}^g \strutgraphbot{\omega}{\alpha_i} \wedge \strutgraphbot{\omega}{\beta_i}
\right\rangle_{\Sp\sHQ} \right).
$$
We conclude with Lemma \ref{lem:ker_bracket} that $\Ker(c_2)$ is $\Sp\sHQ$-spanned by $\{r_1\}$.
\end{proof}

Finally, we have the following commutative diagram:
$$
\xymatrix{
{\Gr^\Gamma_2 \Ig \oQ} \ar[rr]^-{\Gr_2 \mapcyl \oQ} && 
{\Gr^Y_2 \Cg \oQ} \ar[d]^-{\LMO_2}_-\simeq\\
{\Lie_2\left(\Lambda^3 H_\Q/H_\Q \right)/\, \langle\, \{r_1\}\, \rangle_{\SpHQ}}
\ar[u]^-{\overline{J}_2} \ar[rr]^-{\overline{Y}_2} && {\A_2^{<,c}/I_2^{<,c}}.
}
$$
The previous proposition shows that $\overline{Y}_2$ is injective, 
which implies that $\overline{J}_2$ is injective (hence Hain's Proposition \ref{prop:Hain}).
We conclude that 
$$
\Ker(Y_2) = R_2\left(\Ig\right),
$$
which is the analogue of Theorem \ref{th:degree_2} in the closed case.
Thus, the map $\Gr \mapcyl \oQ: \Gr^\Gamma \Ig \oQ \to \Gr^Y \Cg \oQ$ 
is injective in degree $2$.

\vspace{0.5cm}

\section{Remarks and questions}

\label{sec:conclusion}

In this section, we assume $g\ge6$ for simplicity.
Some statements can be adapted to more general cases.
Also, we consider only the surfaces with one boundary component, 
though most of the arguments below generalize to closed surfaces.

\subsection{The Lie subalgebra $a(\HQ)\subset \A^c\sHQ$}

\label{subsec:a_in_A}

Let  $a(\HQ)$ be the Lie subalgebra of $\A^c\sHQ$ generated by its degree
$1$ part $\Ac{1}\sHQ \simeq \LHQ$.

\begin{question}
\label{ques:quadratic_presentation}
Is the graded Lie algebra $a(\HQ)$ quadratically presented?
\end{question}

\noindent
According to Lemma \ref{lem:ker_bracket}, $a(\HQ)$ has a quadratic presentation if and only if  
the kernel of the Lie algebra map
\begin{gather*}
\chi^{-1}\circ Y \colon\Lie\left(\LHQ\right)  \longrightarrow \A^{c}\sHQ,
\quad  x\wedge y\wedge z \longmapsto \Ygraphbottoptop{x}{y}{z}
\end{gather*}
is generated as an ideal by $\langle r_1,r_2 \rangle_{\SpHQ}$.

By Theorem \ref{th:diagram}, Question \ref{ques:quadratic_presentation}  
is equivalent to the fact that the Torelli Lie algebra is quadratically presented 
(i.e$.$ the second part of Hain's Theorem \ref{th:Hain}) \emph{and} the injectivity of
$\Gr \mapcyl \oQ\colon\Gr^\Gamma\Igg\oQ\longrightarrow \Gr^Y\Cgg\oQ.$

Another very interesting problem would be to characterize the elements of $\A^c\sHQ$ belonging to $a\sHQ$.
Recall from \eqref{eq:a} that $a\sHQ$ is contained in $\A^{c}_{*,\ev}\sHQ$.

\subsection{Johnson homomorphisms and symplectic tree diagrams}

\label{subsec:tree-level}

Let us explain how the questions asked in \S~\ref{subsec:a_in_A} 
connect to problems about Johnson homomorphisms. Let
$$
\Igg = \Igg[1] \supset \Igg[2] \supset  \Igg[3] \supset \cdots
$$ 
be the Johnson filtration of the Torelli group \cite{Johnson_survey},
defined by
\begin{equation}
\label{eq:Johnson_filtration}
  \Igg[i]:=\Ker\Bigl(\Igg\longrightarrow
  \Aut(\pi_1\sgg/\Gamma_{i+1}(\pi_1\sgg))\Bigr)\quad \text{for $i\ge1$.}
\end{equation}
It is proved in \cite{Morita_Abelian} that
$$
\big[\Igg[k],\Igg[l]\big] \subset \Igg[k+l] \quad \quad \forall k,l \geq 1,
$$
hence a Lie bracket on the graded Abelian group
$$
\Gr^{[-]} \Igg = \bigoplus_{i\geq 1} \Gr^{[-]}_i\Igg,
\quad \hbox{where } \Gr^{[-]}_i \Igg := {\Igg[i]}/{\Igg[i+1]}.
$$
Also, $\textrm{id}_{\Igg}$ induces a Lie algebra homomorphism
\begin{gather*}
  j\colon\Gr^\Gamma \Igg\oQ \longrightarrow \Gr^{[-]} \Igg \oQ  .
\end{gather*}
The following problem was proposed by Morita (in the case of closed surfaces):

\begin{question}[See {\cite[Problem 6.2]{Morita_survey}}]
\label{ques:Morita}
Is the map $j$ injective in degree $\neq 2$?
\end{question}

The Johnson homomorphisms (tensored by $\Q$)\footnote{
The definition of Johnson homomorphisms needs several sign conventions.
To be specific, the homomorphism $\tau_i: \Igg[i] \to \A^{c,t}_i\sHQ$ 
that we are considering here is the 
homomorphism $\tau_i:\Igg[i] \to \A^{Y,c,t}_i(\set{g}^+\cup \set{g}^-)$ as defined in \cite{CHM},
composed with the tree-reduction of the isomorphism 
$\chi^{-1}s\varphi:\A^{Y,c}(\set{g}^+\cup \set{g}^-) \to \A^c\sHQ$.} induce
a Lie algebra monomorphism
$$
\tau: \Gr^{[-]} \Igg \oQ \longrightarrow \A^{c,t}\sHQ,
$$ 
where
\begin{gather*}
  \A^{c,t}\sHQ:=\A^c\sHQ/\text{(diagrams with loops)}  
\end{gather*}
consists of \emph{symplectic tree diagrams}.  The Lie bracket on $\A^{c,t}\sHQ$,
introduced by Kontsevich \cite{Kontsevich} and Morita \cite{Morita_Abelian},
coincides with the tree-reduction of the Lie bracket $[-,-]_\star$ on $\A^c(H_\Q)$ defined in \S~\ref{subsec:SJD}. 

As announced in \cite{Habiro} and proved in \cite{GL_tree-level,Habegger}, 
the Johnson homomorphisms naturally extend to the monoid of homology cylinders,
and one can describe these extensions using symplectic tree diagrams.
Moreover,  we have 
$$
Y_i \Cgg \subset \Cgg[i] \quad \quad \forall i\geq 1.
$$
So, $\textrm{id}_{\Cgg}$ induces a Lie algebra homomorphism
$$
k\colon\Gr^Y \Cgg \oQ \longrightarrow  \Gr^{[-]} \Cgg \oQ.
$$
Then, the commutative square (\ref{eq:diagram}) 
can be expanded to the commutative cube drawn in Figure \ref{fig:cube}.
There, $r_3 \in \Lie_2(\LHQ) \simeq \Lambda^2 \A_1^{c}\sHQ$ 
is a pre-image of the graph $\thetagraph$ by $[-,-]_\star$.
Commutativity of the rightmost square follows from 
the domination of the Johnson homomorphisms by the LMO invariant \cite{CHM}.
\begin{figure}[h!]
$$
\newdir{ >}{{}*!/-10pt/\dir{>}}
\xymatrix{
&{\Gr^\Gamma \Igg \oQ}  \ar[rr]^{\Gr \mapcyl \oQ}  \ar[dl]_{j} && 
{\Gr^Y \Cgg \oQ} \ar[dd]_-{\chi^{-1}\LMO}^-\simeq  \ar@{->>}[dl]_{k} \\
 {\Gr^{[-]} \Igg \oQ\ } \ar@{>->}[rr]^{\hspace{5em} \Gr^{[-]} \mapcyl\oQ}  &&  
{\Gr^{[-]} \Cgg \oQ} \ar[dd]^(.7)\tau_(.7)\simeq &\\
&{\frac{\Lie\left(\LHQ\right)}{\left\langle \langle r_1,r_2 \rangle_{\SpHQ} \right\rangle_{\rm{ideal}}}}
\ar@{ >->>}'[u]_(.5){\swarrow^{\rm Hain}}^(0.8){\overline{J}}[uu] 
 \ar'[r][rr]^-{\chi^{-1}\overline{Y}} \ar@{->>}[dl]_{\text{proj}} && 
{\A^c\sHQ} \ar@{->>}[dl]^{\text{proj}} \ar@/_1pc/[uu]_-{\psi\chi} \\
{\frac{\Lie\left(\LHQ\right)}
{\left\langle\langle r_1,r_2,r_3 \rangle_{\SpHQ}\right\rangle_{\rm{ideal}}}} 
\ar[uu] \ar[rr] &&  {\A^{c,t}\sHQ} & \\
}
$$
\caption{A cube of graded Lie algebras with $\SpHQ$-actions.}
\label{fig:cube}
\end{figure}

Let  $a^t(\HQ)$ be the Lie subalgebra of $\A^{c,t}(H_\Q)$ generated 
by its degree $1$ part $\A^{c,t}_1(H_\Q) \simeq \LHQ$.
Question \ref{ques:quadratic_presentation} has the following analogue at the ``tree'' level.

\begin{question}
\label{ques:quadratic_presentation_tree}
Is the graded Lie algebra $a^t(\HQ)$  quadratically presented?
\end{question}

\noindent
By  Lemma \ref{lem:ker_bracket}, $a^t(\HQ)$ has a quadratic presentation if and only if  
the kernel of the Lie algebra map
$$
\Lie\left(\LHQ\right) \longrightarrow \A^{c,t}\sHQ, \quad 
x\wedge y\wedge z \longmapsto \Ygraphbottoptop{x}{y}{z}
$$
is generated as an ideal by $\langle r_1,r_2,r_3 \rangle_{\Sp\sHQ}$.

The Lie subalgebra $a^t(\HQ)$ is equal to the image of the Johnson homomorphisms:
\begin{gather*}
  a^t(\HQ) = \Img(\tau: \Gr^{[-]} \Igg \oQ \longrightarrow \A^{c,t}\sHQ).
\end{gather*}
We deduce that $a(H_\Q)$ is strictly contained in $\A^{c}_{*,\ev}(H_\Q)$, 
since $\tau$ is not surjective, as was shown by Morita \cite{Morita_Abelian}. 

\begin{remark}
We can verify Question \ref{ques:quadratic_presentation} in degree $3$, with a long computation by hand.
Since the projection $\A^{c}_{3,\ev}(H_\Q) \to \A^{c,t}_3(H_\Q)$ is 
bijective, it follows that Question \ref{ques:quadratic_presentation_tree} has a positive answer in degree $3$.
This fact might be well-known to experts (though it does not seem to have been proved or
announced explicitly in the literature), since a version for closed surfaces
is already proved by  Morita \cite[Proposition 6.3]{Morita_survey}.  

Moreover, we deduce that  $r_3$ is central in the Lie algebra
$$
\Lie\left(\LHQ\right)/\left\langle \langle r_1,r_2 \rangle_{\Sp\sHQ} \right\rangle_{\rm{ideal}},
$$ 
so that the projection
$$
\Lie\left(\LHQ\right)/\left\langle \langle r_1,r_2 \rangle_{\Sp\sHQ} \right\rangle_{\rm{ideal}}
\longrightarrow
\Lie\left(\LHQ\right)/\left\langle \langle r_1,r_2,r_3 \rangle_{\Sp\sHQ} \right\rangle_{\rm{ideal}}
$$ 
is bijective in all degree $\neq2$.  
Then, the commutative cube of Figure \ref{fig:cube}
shows that Question \ref{ques:quadratic_presentation_tree} implies 
Questions \ref{ques:quadratic_presentation} and \ref{ques:Morita}.  
Also, giving a positive answer to Question \ref{ques:quadratic_presentation_tree} by
algebraic means would provide a new, algebraic proof for the second half
of Hain's Theorem \ref{th:Hain}.
\end{remark}

\vspace{0.5cm}

\appendix

\section{On Malcev Lie algebras of filtered groups}

In this appendix, we define the ``Malcev completion'' and the ``Malcev Lie algebra'' of a filtered group. 
When the filtration is the lower central series,  
our definitions agree with the usual constructions of the Malcev completion 
and Malcev Lie algebra  of a group.
Our exposition is inspired by Quillen's appendix to his paper \cite{Quillen},
which deals with the usual case.
We fix a ground field $\K$ of characteristic $0$.
We start by recalling some general definitions.\\

A \emph{filtration} on a vector space $V$ is a descending chain of subspaces
$$
V= F_0 V \supset F_1 V \supset F_2 V \supset \cdots.
$$
A \emph{complete vector space} is a filtered vector space $V$
such that $V \simeqto \varprojlim_{i} V/F_i V$.
The \emph{complete tensor product} $V\widehat{\otimes} W$ 
of two filtered vector spaces $V$ and $W$ is the completion 
of $V\otimes W$ with respect to the filtration
$$
F_n (V \otimes W) := \sum_{i+j=n} F_i V \otimes F_j W.
$$
The category of complete vector spaces endowed with $\widehat{\otimes}$ 
is a symmetric monoidal category.
The unit object is $\K$ filtered by $\K=F_0 \K \supset F_1\K =\{0\}$,
and the transpose map $V \widehat{\otimes} W \to W \widehat{\otimes} V$ is the completion
of the usual transpose map $V \otimes W \to W \otimes V$.

\begin{definition}
A \emph{complete Hopf algebra} is a Hopf algebra 
in the symmetric monoidal category of complete vector spaces.
\end{definition}

\begin{remark}
The reader is referred to \cite[Appendix A]{Quillen} for a detailed
treatment of complete Hopf algebras.  It should be
observed that Quillen's definition of a ``complete Hopf algebra''
is more restrictive than ours, since he requires the associated graded
algebra to be generated by its degree $1$ subspace.  Consequently,
some of the results of \cite{Quillen} do not directly generalize to
our situation.
\end{remark}
 
A \emph{filtration}\footnote{Originally termed ``N-series'' by Lazard \cite{Lazard}.} 
of a group $G$ is a descending chain of subgroups
$$
G=F_1 G \supset F_2 G \supset F_3 G \supset \cdots
$$
such that $\left[F_i G,F_j G\right] \subset F_{i+j} G$.
A \emph{filtration} of a Lie algebra is defined in a similar way.

\begin{example}
The lower central series  of a group $G$ is a filtration of $G$,
and similarly for a Lie algebra.
\end{example}

In the sequel, we consider a filtered group $G$.
Let $\K[G]$ be its group algebra, and let $I$ be the augmentation ideal of $\K[G]$.
The given filtration on $G$ induces a filtration on the algebra $\K[G]$ by defining
\begin{equation}
\label{eq:filtration_group_algebra}
F_i \K[G] :=  \left\langle\ \left(g_1-1\right) \cdots \left(g_r-1\right) \
\left| \begin{array}{l} g_1 \in F_{k_1} G,\dots, g_r \in F_{k_r} G \\
 k_1+ \cdots + k_r \geq i \end{array} \right.\ \right\rangle_\K.
\end{equation}

\begin{example}
If $G$ is filtered by the lower central series, 
then the filtration (\ref{eq:filtration_group_algebra}) 
is the $I$-adic filtration $(I^k)_{k\geq 0}$ of $\K[G]$.
\end{example}

The completion of $\K[G]$ with respect to the filtration (\ref{eq:filtration_group_algebra}) is denoted by
$$
\widehat{\K}[G] := \varprojlim_{i} \K[G]/ F_i \K[G].
$$
The filtration on $\K[G]$ induces a filtration on $\widehat{\K}[G]$ defined by
$$
F_j  \widehat{\K}[G] := \varprojlim_{i\geq j} F_j \K[G]/ F_i \K[G].
$$
The group Hopf algebra structure on $\K[G]$ induces a cocommutative, complete
Hopf algebra structure on $\widehat{\K}[G]$.

It follows that $\Gr \widehat{\K}[G] \simeq \Gr \K[G]$ is a graded cocommutative Hopf algebra.
A result by Quillen describes it as the universal enveloping algebra of a graded Lie algebra.
See \cite{Quillen_graded} in the case of the lower central series 
and \cite{Massuyeau} in the general case:

\begin{theorem}
\label{th:graded_Quillen}
Let $\Gr G$ be the graded Lie algebra over $\Z$ induced by the filtration of $G$, namely
$$
\Gr G = \bigoplus_{i\geq 1} F_i G/ F_{i+1} G
$$
with Lie bracket induced by commutator in $G$.
Then, there is a graded Hopf algebra isomorphism
$$
\theta: U(\Gr G \otimes \K) \simeqto \Gr \K[G]
$$
defined by $\{g\}\otimes 1 \mapsto \{g-1\}$ for all $g\in F_iG$ and for all $i\geq 1$.
\end{theorem} 

The complete vector space $\widehat{\K}[G]$ being a complete Hopf algebra,
we can consider the primitive and the group-like elements of $\widehat{\K}[G]$.

\begin{definition}
The \emph{Malcev completion} of the filtered group $G$ 
is the group of group-like elements of $\widehat{\K}[G]$
$$
\MalcevGroup{G} := \GLike \widehat{\K}[G] =
\left\{x \in \widehat{\K}[G]: \widehat\Delta x= x \widehat{\otimes} x, x\neq 0 \right\}.
$$
filtered by 
$$
F_i \MalcevGroup{G} := \MalcevGroup{G} \cap \left(1+ F_i \widehat{\K}[G]\right), \quad  \forall i\geq 1.
$$
The \emph{Malcev Lie algebra} of $G$ is the Lie algebra of primitive elements of $\widehat{\K}[G]$
$$
\MalcevLie{G} := \Prim  \widehat{\K}[G] = \left\{x \in \widehat{\K}[G]: 
\widehat\Delta x= x \widehat{\otimes} 1 + 1 \widehat{\otimes} x \right\}
$$
filtered by 
$$
F_i \MalcevLie{G} := \MalcevLie{G} \cap F_i \widehat{\K}[G], \quad  \forall i\geq 1.
$$
\end{definition}

\noindent
The canonical map $G\to \MalcevGroup{G}$ is denoted by $\iota$.
It should be observed that $\iota$ is not necessarily injective.

\begin{example}
When $G$ is filtered by the lower central series, 
$\MalcevGroup{G}$ and $\MalcevLie{G}$ are the usual notions 
of Malcev completion and Malcev Lie algebra respectively.
This construction, via the completed group algebra, 
is due to Jennings \cite{Jennings} and Quillen \cite{Quillen}.
\end{example}

As is always true in a complete Hopf algebra, the primitive and the group-like elements
are in one-to-one correspondence by the exponential and logarithmic maps:
$$
\MalcevGroup{G} \subset 1 + \widehat{I} 
\overset{\log}{\underset{\exp}{\longrightleftarrows}} \widehat{I} \supset \MalcevLie{G}.
$$
Here, $\widehat{I} \subset\widehat{\K}[G] $ denotes the completion of the augmentation ideal $I\subset \K[G]$.
The $\log$ and $\exp$ series converge on $1 + \widehat{I}$
and $\widehat{I}$ respectively since $I^k \subset F_k \K[G]$ for all $k\geq 1$.\\

In the rest of this appendix, we extend two results, which are well-known
for the lower central series, to arbitrary group filtrations.
The first one is the following

\begin{theorem}
\label{th:graded_Malcev_Lie_algebra}
The logarithmic map $\log \iota: G \to \MalcevLie{G}$ induces a  graded Lie algebra isomorphism:
$$
(\Gr \log  \iota) \otimes \K: \Gr G \otimes \K  \simeqto \Gr \MalcevLie{G}.
$$
\end{theorem}

\begin{proof}
First of all, $\log(\iota(g)) = \log(1+(\iota g-1))= (\iota g-1) - (\iota g-1)^2/2 + \cdots$
belongs to $F_i \widehat{\K}[G]$ for all $g\in F_i G$ and for all $i\geq 1$.
Thus, the map $\log \iota:G \to \MalcevLie{G}$ preserves the filtration 
so that $\Gr \log \iota : \Gr G   \to \Gr \MalcevLie{G}$ is well-defined. 
The latter preserves the Lie brackets by the Baker--Campbell--Hausdorff formula.
It remains to prove that the graded Lie algebra map 
$(\Gr \log \iota)\otimes \K : \Gr G \otimes \K   \to \Gr \MalcevLie{G}$ is bijective.

Quillen's isomorphism $\theta$ (Theorem \ref{th:graded_Quillen})
gives a graded Lie algebra isomorphism at the level of primitive elements:
$$
\theta:\Gr G \otimes \K \simeqto \Prim \Gr \K[G].
$$ 
Moreover, there is the canonical Lie algebra homomorphism
$$
b: \Gr \MalcevLie{G} = \Gr \Prim \widehat{\K}[G] \longrightarrow 
\Prim \Gr \widehat{\K}[G] \simeq \Prim \Gr \K[G].
$$ 
The map $b$ is injective since the filtration that we are considering on $\MalcevLie{G}$ 
is the restriction of the filtration on $\widehat{\K}[G]$.
We conclude thanks to the following commutative triangle:
$$
\xymatrix{
{\Gr G \otimes \K} \ar[rrd]_-\theta^-\simeq \ar[rr]^-{(\Gr \log \iota)\otimes \K} 
& &{\Gr \MalcevLie{G}} \ar@{>->}[d]^-b \\
& &{\Prim \Gr \K[G]}.
}
$$
\end{proof}

The usual form of Theorem \ref{th:graded_Malcev_Lie_algebra} is as follows:

\begin{corollary}
\label{cor:graded_Malcev_Lie_algebra}
If the group $G$ is filtered by the lower central series,
then the logarithmic map $\log \iota: G \to \MalcevLie{G}$ induces a graded Lie algebra isomorphism
$$
(\Gr \log \iota)\otimes \K: \Gr G \otimes \K \simeqto \Gr^{\widehat{\Gamma}} \MalcevLie{G},
$$
where $\widehat{\Gamma}$ is the complete lower central series\footnote{
The \emph{complete lower central series} of a complete Lie algebra $L$
is the filtration $\widehat{\Gamma}$ whose $i$-th term  $\widehat{\Gamma}_i L$
is the closure of $\Gamma_i L$ in $L$.} of $\MalcevLie{G}$.
\end{corollary}

\begin{proof}[Proof of Corollary \ref{cor:graded_Malcev_Lie_algebra}]
The Baker--Campbell--Hausdorff formula shows that the map $\log \iota: G \to \MalcevLie{G}$
sends the lower central series of $G$ to the complete lower central series of $\MalcevLie{G}$, so that 
$\Gr \log \iota: \Gr G  \to \Gr^{\widehat{\Gamma}} \MalcevLie{G}$ is well-defined.
The same formula shows that $\Gr \log \iota$ preserves the Lie brackets.

Since $\Gamma_n \MalcevLie{G} \subset F_n \MalcevLie{G}$, 
we have $\widehat{\Gamma}_{n} \MalcevLie{G} \subset F_n \MalcevLie{G}$
hence a canonical map $c:\Gr^{\widehat \Gamma} \MalcevLie{G} \to \Gr \MalcevLie{G}$.
The commutative triangle
$$
\xymatrix{
\Gr G \otimes \K \ar[rr]^-{(\Gr \log \iota)\otimes \K} 
\ar[drr]_-{(\Gr \log \iota)\otimes \K\quad }^-\simeq  
&& \Gr^{\widehat{\Gamma}} \MalcevLie{G} \ar[d]^c \\
&& \Gr \MalcevLie{G} 
}
$$
shows that $c$ is surjective. So, we have 
$$
F_n \MalcevLie{G} \subset 
\left(\widehat{\Gamma}_n \MalcevLie{G} + F_{n+1} \MalcevLie{G}\right), \quad \forall n\geq 1.
$$ 
We deduce that $F_n \MalcevLie{G}$ is contained 
in the closure of $\widehat{\Gamma}_n \MalcevLie{G}$,
so that $F_n \MalcevLie{G} = \widehat{\Gamma}_n \MalcevLie{G}$. The conclusion follows.
\end{proof}

To state the second result, we shall recall that, classically, 
the Malcev completion of a nilpotent group refers to its uniquely divisible closure. 
A \emph{uniquely divisible closure} of a nilpotent group $N$ is a pair $(D,i)$, where 
\begin{itemize}
\item $D$ is nilpotent and is uniquely divisible: 
$\forall y\in D, \forall k\geq 1, \exists ! x\in D, x^k =y$,
\item $i:N\to D$ is a group homomorphism whose kernel is the torsion subgroup of $N$,
\item $\forall x\in D, \exists k\geq 1, x^k \in i(N)$.
\end{itemize}
Malcev proved that the uniquely divisible closure of a nilpotent group $N$
exists and is essentially unique (see \cite{KM} for instance): 
Let us denote it by $N\oQ$.
As proved by Jennings in the finitely generated case \cite{Jennings}
and by Quillen in general \cite{Quillen}, 
the group $\MalcevGroup{N}$, where $N$ is filtered by the lower central series and where $\K=\Q$,
is a realization of $N\oQ$. This fact can be generalized as follows:

\begin{theorem}
\label{th:divisible_closure}
If the ground field is $\K=\Q$, then the canonical map $\iota:G \to \MalcevGroup{G}$ induces a group isomorphism
$$
\iota: \varprojlim_{i} \left( \left(G/F_i G\right) \oQ \right)\simeqto \MalcevGroup{G}.
$$
\end{theorem}

\begin{proof}
For all $g \in G$, $\iota(g) \in \MalcevGroup{G}$ belongs to $1 + F_i \widehat{\Q}[G]$
if and only if $g$ belongs to $1 + F_i \Q[G]$. 
Consider the ``$i$-th dimension subgroup'' with coefficients in $\Q$ of the filtered group $G$, namely
$$
D_i(G):=\{g\in G: \hbox{the image of $g$ in $\Q[G]$ belongs to } 1 + F_i \Q[G]\}.
$$
We can also consider the radical closure of $F_i G$ in $G$:
$$
\sqrt{F_i G} := \{g\in G: \exists k\geq 1, g^k \in F_i G\}.
$$
A classical result by Malcev asserts that $D_i(G) = \sqrt{F_i G}$ 
in the case of the lower central series (see \cite{Passi} for instance),
and it can be generalized to arbitrary filtrations (see \cite{Massuyeau}).  
Thus, the homomorphism $\iota: G \to \MalcevGroup{G}$ induces a group monomorphism
\begin{equation}
\label{eq:monomorphism}
\iota: G/ \sqrt{F_i G} \longrightarrow \MalcevGroup{G}/ F_i\MalcevGroup{G}
\end{equation}
for all $i\geq 1$. 

Next, we observe several properties for the group $\MalcevGroup{G}/ F_i\MalcevGroup{G}$. 
First, it is nilpotent (being the quotient of a group by a term of a filtration).
Second, it is torsion-free: Let $x \in \MalcevGroup{G}$ and $n\geq 1$ be such that
$x^n \in F_i \MalcevGroup{G}$; setting $y:= \log(x)$, 
we deduce that $n y \in F_i \widehat{\Q}[G]$ or, equivalently, that 
$y \in F_i \widehat{\Q}[G]$; we conclude that $x=\exp(y) \in (1+F_i \widehat{\Q}[G])$.
Third, the group $\MalcevGroup{G}/ F_i\MalcevGroup{G}$ is divisible 
since $\MalcevGroup{G}= \exp \MalcevLie{G}$ is. 

Therefore, $\MalcevGroup{G}/ F_i\MalcevGroup{G}$ is a uniquely divisible nilpotent group.
To show that it is the uniquely divisible closure of $G/ \sqrt{F_i G}$
via (\ref{eq:monomorphism}), we prove
\begin{equation}
\label{eq:powering}
\forall x\in\MalcevGroup{G}, \exists k\geq 1,  x^k \in 
\iota(G) \cdot F_i \MalcevGroup{G}
\end{equation}
by induction on $i\geq 1$. 
Thus, we consider an $x\in\MalcevGroup{G}$ satisfying 
that there exists $k\geq 1$, $g\in G$ and $y\in F_{i-1} \MalcevGroup{G}$
such that $x^k= \iota(g) y$. According to Theorem \ref{th:graded_Malcev_Lie_algebra}, the map 
$$
\left(\Gr G\right) \oQ \longrightarrow \Gr \MalcevGroup{G}
$$
defined by $\{h\}\otimes q \mapsto \{\exp\left(q\log(\iota h)\right)\}$ 
for all $h\in F_jG$, $j\geq 1$ and $q\in \Q$, is a bijection.
We deduce that there exists $n\geq 1$ and $h\in G$ such that $y^n= \iota(h)$ modulo $F_i\MalcevGroup{G}$. 
Since $y\in F_{i-1} \MalcevGroup{G}$ is central modulo $F_i\MalcevGroup{G}$,
we have that $x^{kn} = \iota(g)^n y^n =\iota(g^n h)$ modulo $F_i\MalcevGroup{G}$. 
This proves claim (\ref{eq:powering}).

Thus,  (\ref{eq:monomorphism}) induces a  group isomorphism
\begin{equation}
\label{eq:monomorphism_bis}
\iota: \left(G/ \sqrt{F_i G}\right)\oQ \longrightarrow \MalcevGroup{G}/ F_i\MalcevGroup{G}
\end{equation}
defined by $\{x\}^{1/n} \mapsto \{\exp\left(\frac{1}{n}\log(\iota x)\right)\}$ 
for all $x\in G$ and $n\geq 1$. 
Passing to the inverse limit, we obtain an isomorphism
\begin{equation}
\label{eq:final_monomorphism}
\iota: \varprojlim_i \left(G/ F_i G\right)\oQ \simeq 
\varprojlim_i  \left(G/ \sqrt{F_i G}\right)\oQ \longrightarrow 
\varprojlim_i \MalcevGroup{G}/ F_i\MalcevGroup{G} \simeq \MalcevGroup{G}.
\end{equation}
\end{proof}

%
%
%
%
%
%

\bibliographystyle{amsalpha}

\end{document}